\documentclass{amsart}
\usepackage[utf8]{inputenc}
\usepackage{graphicx}
\usepackage{rotating}
\pdfoutput=1

\usepackage{graphicx, booktabs,enumerate,microtype}
\usepackage[dvipsnames]{xcolor}
\usepackage{amssymb,amsmath,amsthm, stmaryrd, mathtools} 
\usepackage{tikz, tikz-cd}
\usepackage{hyperref}
    \newcommand{\lm}[1]{\texorpdfstring{#1}{replacedLaTeXcode}} 
\usepackage{mathdots}
\usepackage{comment}
\usepackage{mathrsfs}
\newcommand{\ms}[1]{\mathscr{#1}}

\hypersetup{colorlinks=true, linkcolor=RoyalBlue, citecolor=JungleGreen}

\let\oldtocsection=\tocsection

\let\oldtocsubsection=\tocsubsection

\let\oldtocsubsubsection=\tocsubsubsection

\renewcommand{\tocsection}[2]{\hspace{0em}\oldtocsection{#1}{#2}}
\renewcommand{\tocsubsection}[2]{\hspace{1em}\oldtocsubsection{#1}{#2}}
\renewcommand{\tocsubsubsection}[2]{\hspace{2em}\oldtocsubsubsection{#1}{#2}}

\DeclareRobustCommand{\SkipTocEntry}[5]{}

\newcounter{prcounter}

\newcommand{\lf}{\left}
\newcommand{\ri}{\right}
\newcommand{\md}{\middle}
 
\newcommand{\f}{\frac} 
 
\newcommand{\into}{\hookrightarrow}
\newcommand{\onto}{\twoheadrightarrow}

\newcommand{\iso}{\xrightarrow{\sim}}
\newcommand{\wh}{\widehat}

\DeclareMathOperator{\Gal}{Gal}

\DeclareMathOperator{\End}{End}
\DeclareMathOperator{\Res}{Res}

\DeclareMathOperator{\diag}{diag}
\DeclareMathOperator{\ad}{ad}
\DeclareMathOperator{\Ad}{Ad}

\DeclareMathOperator{\Disc}{\mathrm{Disc}}

\newcommand{\m}[1]{\mathbf{#1}}
\newcommand{\mf}[1]{\mathfrak{#1}}
\newcommand{\mc}[1]{\mathcal{#1}}
\newcommand{\td}[1]{\tilde{#1}}

\newcommand{\R}{\mathbb R}
\newcommand{\Q}{\mathbb Q}
\newcommand{\N}{\mathbb N}
\newcommand{\Z}{\mathbb Z}

\newcommand{\F}{\mathbb F}
\newcommand{\A}{\mathbb A}

\newcommand{\SL}{\mathrm{SL}}

\newcommand{\SO}{\mathrm{SO}}

\newcommand{\GL}{\mathrm{GL}}

\newcommand{\PGL}{\mathrm{PGL}}
\newcommand{\PU}{\mathrm{PU}}

\newcommand{\SU}{\mathrm{SU}}

\newcommand{\Ga}{\mathbb G_a}

\newcommand{\Id}{\mathrm{Id}}
\newcommand{\Mat}{\mathrm{Mat}}

\newcommand{\lb}{\lambda}

\newcommand{\bs}{\backslash}

\newcommand{\disc}{\mathrm{disc}}

\newcommand{\spl}{\mathrm{spl}}

\newcommand{\der}{\mathrm{der}}

\newcommand{\temp}{\mathrm{temp}}

\newcommand{\inv}{\mathrm{inv}}

\newcommand{\Trd}{\mathrm{Trd}}
\usepackage{stackengine,scalerel}

\newtheorem{thm}{Theorem}[subsection]
\newtheorem{prop}[thm]{Proposition}
\newtheorem{clm}[thm]{Claim}
\newtheorem{cor}[thm]{Corollary}
\newtheorem{lem}[thm]{Lemma}

\newtheorem{problem}[thm]{Problem}
\newtheorem{construction}[thm]{Construction}

\theoremstyle{remark}

\newtheorem*{ex}{Example}
\newtheorem{rmk}[thm]{Remark}

\theoremstyle{definition}
\newtheorem{dfn}[thm]{Definition}
\newtheorem{assm}[thm]{Assumption}

\numberwithin{equation}{subsection}

\allowdisplaybreaks

\title[Unitary Ramanujan Complexes]{Ramanujan Complexes from Unitary Groups over Number Fields}
\date{\today}

\author{Rahul Dalal}
\address{Succinct Inc, 101 Mission St. San Francisco, CA}
\email{dalal.rahul@gmail.com}

\author{Alberto Mínguez}
\address{Departamento de \'Algebra and Instituto de Matem\'aticas (IMUS), Universidad de Sevilla, C/ Tarfia s/n, 41012 Sevilla \newline
University of Vienna, 
Fakult{\"a}t f{\"u}r Mathematik,
Oskar-Morgenstern-Platz 1,
1090 Wien
}
\email{alberto.minguez@univie.ac.at}

	\author{Jiandi Zou}
	\address{Institute for Advanced Study in Mathematics of Harbin Institute of Technology, Harbin, China}
	\email{idealzjd@gmail.com}

\begin{document}

\begin{abstract}

In this article, we construct new families of Ramanujan complexes with local structure distinct from all previously known examples. Our approach is based on unitary groups over number fields, more specifically on what we call \emph{super-definite unitary groups}---definite unitary groups that are anisotropic modulo their center at a finite place. These arise naturally as groups of units in central division algebras with involution of the second kind.

Our first main result gives a general construction of infinite families of Ramanujan complexes associated with a super-definite unitary group~$G$ over a totally real number field and a finite place~$v_0$. The structure of the resulting complex is governed by the type of the Bruhat--Tits building at $v_0$. It includes  new examples of type~$A_n$ when $v_0$ is split, and novel families of type~${}^2\!A'_n$, ${}^2 \! A''_n$ (with $n$ even), 
$B$-$C_n$, ${}^2 \! B$-$C_n$ and $C$-$BC_n$ in the non-split case. This construction works uniformly across all ranks.

Since much of the motivation for constructing expander complexes comes from computer science, we investigate the algorithmic explicitness of our construction in the latter part of the paper, and provide an example in rank~5 where it becomes fully explicit. In particular, this example yields \emph{golden gates} for the real Lie group~$\mathrm{PU}(5)$.

\end{abstract}

\maketitle

\tableofcontents

\section{Introduction}

\subsection{Context and Motivation}
Ramanujan complexes are particular types of \emph{expander complexes}. These are combinatorial objects that have recently become important in computer science and mathematics.

The simplest case of an expander complex is an \emph{expander graph}: very roughly, a sparse graph that still has very short paths between vertices. In computer science, such graphs are an important tool for understanding and utilizing pseudorandomness, and therefore have a wide range of applications, from complexity theory to error-correcting codes to network design; see \cite{Vad12} for a survey on pseudorandomness and the central role played by expander graphs. Some deep mathematical theories, such as Kazhdan's property (T) and the Ramanujan conjecture, have been used to give explicit constructions of expander graphs \cite{Margulis73, Margulis84,Margulis88,LPS88}, creating exciting connections between pure mathematics and computer science.

An expander complex $X$ is a higher-dimensional generalization: instead of just vertices and edges, we also keep track of ``higher-dimensional'' connections: triangles made up of 3 vertices, 3-simplices made up of 4 vertices, etc. Together, these give $X$ the structure of a simplicial complex that has an analogous ``expansion property''. See \cite{GK22} for a partial review of the applications of high-dimensional expanders. To name just a few, ideas related to expander complexes have been key to recent major breakthroughs in computer science and combinatorics: for example, in matroid theory \cite{ALGO19}, constructing locally testable classical and high-rate quantum error-correcting codes \cite{DELL22, DELLM25, PK22, EKG22}, constructing lossless vertex expanders \cite{hsieh2025explicita, hsieh2025explicitb}, and providing an alternate proof of the PCP theorem \cite{Din07, dinur2017high}. 
They are even of speculative industrial interest, because the error-correcting codes they give rise to may outperform the currently-used Reed--Solomon codes in the parameters needed to construct efficient zero-knowledge and succinct proof systems---see, for example \cite{Brakedown23, NR24} for some recent, initial explorations using ``exotic'' codes.

There is currently no universally accepted best generalization of the expansion property characterizing expander graphs; various possibilities are surveyed in \cite{Lub18}. Of these, \emph{coboundary expansion} and \emph{topological expansion} seem to be the most relevant in applications. However, these two variants tend to be hard to establish, leading to the consideration of the mathematically more convenient \emph{optimal spectral expansion} condition—this essentially constrains the eigenvalue gaps of operators in a ``Hecke algebra'' of generalized adjacency operators (see \cite{CSZ}). A \emph{Ramanujan complex} is a simplicial complex satisfying this optimal condition of spectral expansion, see Definition \ref{def ramanujancomplex} below for details. Unlike in the graph case, the full relationship between spectral expansion and other types of expansion is still open. Nevertheless, a tight connection is still expected: for example, \cite{KKL14, EK24} construct topological expanders from certain spectral expanders.

\subsection{Our Contribution}
The goal of this paper is to construct novel families of Ramanujan complexes. Our approach builds on the philosophy underlying the construction of Ramanujan graphs by Lubotzky--Phillips--Sarnak \cite{LPS88} and its higher-dimensional extension in type $A$ by the Lubotzky--Samuels--Vishne  \cite{LSV05a, LSV05b}, W. Li \cite{Wli} and A. Sarveniazi \cite{Sarveniazi}. In those works, Ramanujan complexes were obtained as quotients of the Bruhat--Tits building associated with $G = \GL_N(F)$, where $F$ is a local field, just as Ramanujan graphs were obtained in \cite{LPS88} as quotients of the Bruhat--Tits tree for $G = \GL_2(F)$. These quotients arise from certain arithmetic lattices obtained by globalizing $G$ into an inner form of a general linear group, and the optimal spectral expansion condition comes from the Ramanujan conjecture. For $\GL(2)$, this conjecture was proved by Deligne over number fields \cite{DeligneWeil} and by Drinfeld over function fields \cite{Drinfeld}. In higher rank, however, the  conjecture over number fields remains open, and therefore all of the above constructions assume that $F$ has positive characteristic, where L. Lafforgue’s work on $\GL(N)$  over function fields applies \cite{Lafforgue}.

Instead of globalizing $G$ into an inner form of a general linear group, the idea underlying our work is to globalize $G$ into an inner form of a unitary group, more precisely into what we call a \emph{super-definite unitary group}—that is, a definite unitary group such that at one finite place it is anisotropic modulo its center (i.e., the units of a division algebra). Such unitary groups can be constructed using involutions of the second kind. 
Unfortunately, the Langlands correspondences in positive characteristic are not yet developed enough to extend the argument of \cite{LSV05b} to these unitary groups. 

The key insight of this work is that, surprisingly, the situation over number fields is much better due to powerful trace-formula techniques that are unavailable over function fields.  Specifically, the Ramanujan property can be deduced from the (almost complete) endoscopic classification of representations of non-quasi-split unitary groups \cite{KMSW14} together with the work of A. Caraiani \cite{Car12}. Moreover, nothing prevents us from constructing Ramanujan complexes at the other finite places of our super-definite unitary groups as well, producing complexes with significantly different local structures compared to previous examples. 

In the case of expander graphs, having a large ``library'' of local structures to select from is critical for applications. For example, the Sipser--Spielman construction of linear-time encodable and decodable expander codes \cite{SS96, Spiel95} specifically requires biregular graphs with restricted bidegrees instead of purely regular graphs. In the case of higher-dimensional complexes, a recent work \cite{DDL24} on agreement testing needed constructions coming from symplectic type-$C_n$ complexes. 

More precisely, our first main result, Theorem~\ref{thm:abstractramanujan}, provides infinite families of Ramanujan complexes of increasing size, given a super-definite unitary group $G$ defined over a totally real number field $F$ and a place $v_0$ of $F$. Information about the structure of the resulting complex can be read off from the type of the Bruhat--Tits building $\mc B(G_{v_0})$, using the tables in \cite{Tit79}, which determine the reductive quotients of the special fibers of various integral models of $G$; see \S\ref{sec:complextype}. In particular, in the split case, this yields new families of Ramanujan complexes of type $A_n$ while in the non-split case, we get families of various novel types ${}^2 \! A'_n$, ${}^2 \! A''_n$ (with $n$ even),  
$B$-$C_n$, ${}^2 \! B$-$C_n$ and $C$-$BC_n$. 
This abstract construction works for all ranks $n$. 

\subsubsection{Explicitness}

Since our motivation comes from computer science applications, a substantial and important part of this article focuses on one particular example where we can make our complexes \emph{explicit}. By an ``explicit construction'', we mean a construction of expander complexes that can potentially serve as a part of a larger \emph{practical} algorithm. 

At present, there are very few constructions of good higher-dimensional Ramanujan complexes beyond the type-$A_n$ case also covered by \cite{LSV05a, LSV05b}; to our knowledge, the only example is \cite{EGGG}, via $\SO(5)$. Here, only a weaker ``Ramanujan density'' property is achieved, and the construction is not fully explicit. In the case of Ramanujan \emph{graphs}, there are many examples coming from rank-$3$ super-definite unitary groups: Ballantine et al. \cite{BFGKMW} and Ballantine--Ciubotaru \cite{BC} give some that are partially explicit while Evra--Parzanchevski and Evra--Feigon--Maurischat--Parzanchevski \cite{EP22, EFMP23} give ones that are fully explicit in this strong sense.

Part 2 of the paper focuses on making our construction fully explicit. For that, we require an additional assumption that $G$ has class number one, i.e., that $G(\A^\infty) = K^\infty G(F)$ (see Paragraph~\ref{sub:notation} for standard notation on ad\`ele groups), and that $K^\infty \cap G(F) = \{1\}$. Under these assumptions, for all finite places $v_0$, the group $\Lambda_{v_0} := G(F) \cap K^{\infty, v_0}$ acts simply transitively on $G_{v_0}/K_{v_0}$, enabling a highly explicit construction of the complex. In Section~\ref{sec:explicitconstruction}, we explain how this additional ``class-number-one'' property facilitates a strategy to render our abstract construction explicit. More precisely, we describe explicit constructions for the related Ramanujan complexes in the split case (type $A_{n}$) and unramified inert case (type $\,^{2}A_{n}'$). 
With a bit more effort, a similar construction should work for other types we mentioned before; however, for the sake of brevity, we restrict our attention to these two most important cases, which occur at almost all non-archimedean local places.

The caveat is that these assumptions imply, by \cite{MSG12}, that the rank of the unitary group must be less than or equal to~$5$. That same paper provides an example of a super-definite unitary group in five variables, along with a compact subgroup $K^\infty$ satisfying these conditions. We implement our explicit construction using this example in the second half of the paper. More precisely, in Section \ref{sec:thegroup}-\ref{sec:integralstructure}, we explicitly construct our desired division algebra $D$ of degree 5 over $E=\Q(\sqrt{-7})$, an involution $\iota$ of second kind on $D$ together with a $\iota$-stable maximal order $\Lambda_D^{\max}$ of $D$. In Section \ref{sec:thelattice}, using this maximal order we construct the group $\Lambda_{0,p}=\Lambda_{v_0}$ mentioned above with $v_0$ being a fixed prime $p$. In Section \ref{sec:gatesets}, we explain the explicit construction of the related Ramanujan complex, where the key point is to fully describe the finite set of the so-called ``golden gates'' $S_{\bar \Lambda_{0,p}}$. Interestingly, this finite set of precomputed elements is named ``golden gates'', because they can computably and efficiently approximate arbitrary group elements on $\PU(5)$ over $\R$. Golden gates on unitary groups are of interest in quantum computing and were previously studied in \cite{PS18, EP22, DEP}; see Definition \ref{def:gates}. The precomputation of such a set requires solving an explicit norm equation in $\Lambda_{D}^{\max}$, whose detailed procedure is described in Section \ref{sec:findinggates}. Finally in Section~\ref{sec:algo} we summarize our algorithm. 

While the machinery used in Part 1 to prove that our expanders are optimal relies on the latest advances in the theory of automorphic forms, the explicit construction in Part 2 is carried out using classical algebraic number theory---our main reference is a paper by Ralph Hull from 1935! 

Finally, we sum up our main claims for explicit constructions, one formal and one informal:

\begin{thm}\label{thm:introalgoclaim}
Let $p \neq 2,7$ be prime and let $\mc B_p$ be the Bruhat-Tits building for 
\begin{itemize}
    \item  $\GL_5 (\Q_p)$ if $p \equiv 1,2,4 \pmod 7$,
    \item $U_5(\Q_p)$ (the quasisplit unitary group with respect to the unramified quadratic extension of $\Q_p$) if $p \equiv 3,5,6 \pmod 7$.
\end{itemize}
Then we give an explicit algorithm which, given any $n$ relatively prime to $2 \cdot 7 \cdot p$  produces a Ramanujan Complex $\mc X_n$ with universal cover $\mc B_p$. The algorithm runs in time polynomial in $n$ and the size of  $\mc X_n$ grows polynomially in $n$.

\end{thm}

This algorithm and the exact size of $\mc X_n$ is described in \S \ref{ssec:algodetails}. Since the universal cover of $\mc X_n$ is 
$\mc B_p$, its local structure is completely determined: in particular, 
the types of simplices and their incidence relations are fixed 
(see \S\ref{sec:complextype}). Furthermore, there is an efficient way of explicitly describing this local structure (\emph{cf.} \S \ref{sec: localstru}).

\begin{clm}[Informal]\label{clm:practicality}
Fix $p$. Given a once-and-for-all precomputed finite list of data depending just on $p$, the algorithm in Theorem \ref{thm:introalgoclaim} is practically executable on a modern computer for $p = 3,11$. 
\end{clm}

We emphasize that we have not implemented the algorithm---due to its complexity, doing so would constitute a substantial software engineering project. Also, since the number of vertices of $\mc X_n$ is roughly given by a polynomial in $n$ of degree $25$, when $n$ is large it is hopeless to fully compute and store the entire complex explicitly. Thus, by ``practically executable" we rather mean that for any vertex in the complex, we can determine all the adjacent vertices and the related simplices, which requires on the order of several million small matrix multiplications (\emph{cf.} \S \ref{sssec: Countgates})  over a degree 10 extension of $\Q$ 
and is thus practically feasible. 

The practicality of our algorithm also comes with an additional important qualification: the list of precomputed data in Claim \ref{clm:practicality}. Although we do not compute these lists here, as one possible strategy, we reduce the problem to finding all vectors of a given small norm in a certain 25-dimensional integer lattice and checking if their ``characteristic polynomials'' are equal to the given one, see Problem \ref{prob Qmaxboundcharpoly}. This is related to lattice-norm problems that are well-studied due to their applications in cryptography, but are correspondingly known to be computationally challenging in general.

Solving the problem in this particular case, or even determining if it is computationally tractable on modern hardware, is far out of the scope of the authors' current expertise. Nevertheless, we are able to find at least one explicit gate (\emph{cf.} \S \ref{subsection findingexplicitgates}), giving hope that more computational expertise could fully solve the problem. 

 \subsubsection{Possible Extensions}
 We expect there to also be examples of class number one, super-definite unitary groups in rank $4$ (beyond the already-studied rank $2$ and $3$ cases---see Remark \ref{rem:superdefiniteexamples}). In addition, the techniques of \cite{EP24} suggest a  (much more complicated) extension of this construction that works for non-trivial but small class number. 
 
 The results here should also be extended to the function field setting. At the moment, there isn't a good replacement for the endoscopic classification in positive characteristic so the argument of Theorem \ref{thm:abstractramanujan} doesn't generalize. However, the geometric methods used to study the positive characteristic case are rapidly progressing, so it is reasonable to expect a workaround to be available soon---see Remark \ref{rem:abstractramanujanff}. It is therefore worth exploring if the constructions of class-number-one inner forms of $\GL_N$ from \cite{CS98} extend to unitary groups, potentially giving examples in arbitrarily high rank. 

\subsubsection{Conditionality}
The proof of Theorem \ref{thm:abstractramanujan} depends heavily on Mok and Kaletha--Mínguez--Shin--White's endoscopic classifications for unitary groups \cite{Mok15} and \cite{KMSW14}. Both depend on the unpublished weighted twisted fundamental lemma. The second in addition pushes many technical details to a specific reference ``[KMSb]'' that is not yet publicly available. 

We note that \cite{AGIKMS24} recently resolved the dependence of \cite{Mok15} and \cite{KMSW14} on the unitary analogues of the unpublished references ``[A25][A26][A27]'' in \cite{Art13}.

\subsection{Acknowledgments}
We thank Alexander Bertoloni-Meli and Daniel Li-Huerta for explaining many aspects of function-field Langlands correspondences to us and their potential relevance to this project. We also thank Goulnara Arzhantseva, Shai Evra, Mathilde Gerbelli-Gauthier, David Schwein and Hongjie Yu for helpful conversations. 

All authors were supported by Principal Investigator project PAT4832423 of the Austrian Science Fund (FWF) while working on this project. The second author was also supported by the project ATR2024-154613 funded by MICIU/AEI/10.13039/ 501100011033.

\subsection{Notation}\label{sub:notation}
\subsubsection{Conventions}

Let $E$ be a number field and let $v$ be a place of $E$. We denote by $E_v$ the completion of $E$ at $v$. If $v$ is finite, we write $\mc O_{E_v}$ for the ring of integers of $E_v$.

The ring of adèles of $E$ is
\[
\A_E := \prod_v'  E_v,
\]
the restricted product with respect to the subrings $\mc O_{E_v}$ at finite places.

For an $E$-algebra $B$, we define
\[
B_v := B \otimes_E E_v.
\]
Similarly, if $\mc O_B$ is an $\mc O_E$-algebra, we define
\[
\mc O_{B,v} := \mc O_B \otimes_{\mc O_E} \mc O_{E_v}.
\]

If $F \subseteq E$ is a subfield and $p$ is a place of $F$, we define
\[
B_p := B \otimes_F F_p 
= B \otimes_E \bigl(E \otimes_F F_p\bigr) 
= \prod_{v \mid p} B_v
\]
and
\[
\mc O_{B,p} := \mc O_B \otimes_{\mc O_F} \mc O_{F_p}
= \mc O_B \otimes_{\mc O_E} \bigl(\mc O_E \otimes_{\mc O_F} \mc O_{F_p}\bigr)
= \prod_{v \mid p} \mc O_{B,v},
\]
where we use, for example, \cite[Proposition II.\S 3.4]{Ser13} to compute $\mc O_E \otimes_{\mc O_F} \mc O_{F_p}$.

Let $S$ be a subset of the finite places of either $E$ or $F$. We define
\[
B_S := \prod_{v \in S} B_v,
\qquad
B^S := \prod_{v \notin S} B_v,
\qquad
\mc O_{B,S} := \prod_{v \in S} \mc O_{B,v},
\qquad
\mc O_B^S := \prod_{v \notin S} \mc O_{B,v}.
\]
If $\gamma \in \mc O_E$, we define
\[
\mc O_{B,(\gamma)} := B \cap \mc O_{B,S},
\]
where $S$ is the set of places $v$ dividing $\gamma$.

For a variety $X$ over $E$, we define $X_S := X(E_S)$ and $X^S := X(E^S)$. Depending on the context, these may be regarded via their natural coordinate embeddings into $X(\A_E)$.

\subsubsection{Notation Reference} 
We use the following non-standard notation across multiple sections:

\noindent \underline{Abstract Construction}
\begin{itemize}
    \item $\mc B(G)$ is the reduced Bruhat--Tits building of a group $G$ as in \S \ref{sub:BT}.

    \item $X(K^{\infty, v_0}) := X_G(K^{\infty, v_0})$ is the Ramanujan complex from group $G$, finite place $v_0$, and open compact $K^{\infty, v_0}$ from Definition \ref{def:complex}.
    \item  $\Lambda_{0, v_0}$, $\Lambda_{v_0}$ are lattices in $G_{v_0}$ assigned to Ramanujan complexes satisfying Assumption \ref{a:main}. They are related to the open compact groups $K_0^{\infty}$ and $K^{\infty}=K(\mf n)$ of $G^{\infty}$ respectively. 
    \item $\bar \Lambda_{0,v_0}$ is a central quotient of $\Lambda_{0,v_0}$ from \eqref{eq:barlambda}. 
    \item $S_{\bar \Lambda}$ is the generating set of ``gates'' of an appropriate lattice $\bar \Lambda$ as in Definition \ref{def:gates}.
\end{itemize}

\noindent \underline{The Explicit Example}
\begin{itemize}
    \item $E$ is the number field $\Q(\sqrt{-7})$.
    \item $D$ is a particular division algebra over $E$ constructed in \S\ref{sec:Dcons}.
    \item $L,M$ are number fields used to construct $D$ in \S\ref{sssec:prime}.
    \item $p_0$ is a specific prime number, and $\rho_2,\bar \rho_2,\rho_{p_0},\bar \rho_{p_0},\delta, \alpha_i, a_i,b$ are specific algebraic numbers in \S\S \ref{sssec:prime}, \ref{sssec divalg}.

    \item $u$ is a particular cyclic generator of $D$.
    \item $\eta$ is a realization of $D$ as a subalgebra of matrices in \eqref{eq Dgenerators}.
    \item $\iota$ is an involution of the second kind on $D$ defined in \S\ref{sssec:iota}.
    \item $G$ is a unitary group defined in \S\ref{sec:Gcons}.
    \item $\Lambda_D^{\max}$ is a maximal order of $D$ and $u_+,u_-,y_+,y_-$ are specific elements of $D$ needed to define $\Lambda_D^{\max}$ in \S\ref{ssec:maximalorder}.

    \item $A_{\max},B_{\max}$ are transition matrices defined in \S \ref{subsection despLambdaDmaxiota}.

    \item $K_0^\infty$ is an open compact subgroup of $G(\A^\infty)$ defined in \S\ref{sec:K0infty}.

    \item The ``dagger'' $(\cdot)^\dagger$ denotes the transpose-conjugate of a  matrix.

    \item $\m U_q$, $\m A$, $\m B_n$ are matrices defined in \S \ref{sec: Explicit Tri}.

    \item $Q_{\max}$ is a positive definite quadratic form over $\Z^{25}$,  and $P_{\gamma}(t)$ is the characteristic polynomial of some $\gamma\in(\Lambda_D^{\max})^{\iota}$ with coefficients being polynomials $P_1(\vec{x}),P_2(\vec{x}),P_3(\vec{x}),P_4(\vec{x}),P_5(\vec{x})$ over $\vec{x}\in\Z^{25}$ defined in \S \ref{subsection QFCPMV}.

\end{itemize}

\part{General Theory}

\section{Background Material}
In this section, we review simplicial complexes and their quotients. We then introduce Bruhat--Tits buildings and their quotients by discrete subgroups, leading to the notion of Ramanujan complexes as higher-dimensional analogues of Ramanujan graphs.
\subsection{Complexes and Buildings}

\subsubsection{Simplicial Complexes} 

A simplicial complex is a generalization of a graph also encoding ``higher-dimensional'' relationships among sets of three or more vertices. Instead of thinking of these relationships as line segments connecting vertices, we think of them as triangles, tetrahedra, $4$-simplices, and so on, depending on their number of vertices. Formally,

\begin{dfn}

A \emph{simplicial complex} $X$ is a  set $X^0$ of \emph{vertices} together with a set of \emph{simplices} still denoted by $X$, so that each simplex $F \in X$ is a finite subset of $X^0$ such that all the subsets of $F$ are also in $X$ (called the \emph{faces} of $F$). We let $X^i$ denote the set of $i$-simplices (i.e. simplices of $i+1$ vertices) and say $X$ is \emph{finite} if $X^0$ is.  We regard $X$ as the union of all simplices and write $X=\bigsqcup_{i=0}^{+\infty}X^i=\bigsqcup_{i=0}^{+\infty}\bigsqcup_{F\in X^i}F$. In particular, we say that $X$ is of dimension $d$ if $X^{d}\neq\emptyset$ and $X^{d+1}=X^{d+2}\dots=\emptyset$, and in this case we call a simplex of size $d$ a \emph{chamber} of $X$. 
\end{dfn}

In particular, $X^0\sqcup X^1$ is called the $1$-skeleton of $X$, which forms a graph.

By realizing each $i$-simplex as an open subset of $\R^i$ equipped with the corresponding topology, we realize $X$ as a topological space as well. Then, a simplex $F'$ is a face of another simplex $F$ if and only if $F'$ lies in the closure of $F$. 

From now on, we further assume $X$ to be a locally finite, connected simplicial complex of dimension $d$.

\subsubsection{Quotients of Complexes}
Now we consider an $\ell$-group $G$ acting faithfully \emph{simplicially} on $X$. Topologically, each element $\gamma\in G$ induces a homeomorphism with respect to the above simplicial topology. Combinatorially, for any $F\in X^i$ and a face $F'$ of $F$, we have that $\gamma(F)\in X^i$ with $\gamma(F')$ being a face of $F$. We further assume that the stabilizer of each $v\in X$ in $G$ is an open compact subgroup. Then $X$ is called a $G$-complex in the sense of \cite[\S 3C]{Fir16}.

Let $\Gamma$ be a closed subgroup of $G$. We would like to realize $\Gamma\backslash X$ as a quotient simplicial complex of $X$, i.e. the quotient topological space $\Gamma\backslash X$ would be a simplicial complex such that the projection $X\rightarrow \Gamma\backslash X$ is a covering map. In general, even if a quotient topological structure of $\Gamma\backslash X$ could be defined, a simplicial structure compatible with that of $X$ requires more conditions. 

\begin{ex}

Let $X$ be the Bruhat--Tits building of $G=\PGL_2(\Q_p)$ (see Paragraph \ref{sub:BT} below) as an infinite regular tree of degree $p+1$ and $\Gamma$ a lattice acting simply transitively on the set of vertices $X^0$ of $X$ (\emph{cf.} \cite[Lemma 7.4.1]{Lub94}, where $\Gamma=\Gamma(2)$). Let the segment $[0,1]$ represent the closure of a chamber of $X$, then the quotient $\Gamma\backslash X$ is homeomorphic to the segment $[0,1/2]$. Thus $\Gamma \backslash X^0$ reduces to the point ${0}$, while $\Gamma \backslash X^1$ becomes the segment $(0,1/2]$, which is not open. It is impossible to give a ``quotient'' simplicial complex structure of $\Gamma\backslash X$ compatible with that of $X$. This counterexample shows that there cannot exist an element $\gamma$ that stabilizes a chamber without fixing it pointwise.

\end{ex}

We have the following proposition determining whether $\Gamma\backslash X$ is a quotient simplicial complex of $X$.

\begin{prop}\label{prop:GammaonX}\cite[Proposition 3.7, 3.9]{Fir16}
The quotient space $\Gamma\backslash X$ is a simplicial complex if and only if
\begin{itemize}
\item $\{\Gamma u_0,\dots,\Gamma u_s\}=\{\Gamma v_0,\dots,\Gamma v_t\}$ implies $\Gamma\{u_0,\dots,u_s\}=\Gamma\{v_0,\dots,v_t\}$ for all pairs of simplices $\{u_0,u_1,\dots,u_s\}$,$\{v_0,v_1\dots,v_t\}\in X.$
\end{itemize}
Moreover, the projection $X\rightarrow \Gamma\backslash X$ is a covering map if and only if
\begin{itemize}
\item The action of $\Gamma$ is free, meaning that the stabilizer of $v$ in $\Gamma$ is trivial for every $v\in X^0$.
\end{itemize}

\end{prop}

The following proposition determines whether the quotient simplicial complex is finite.

\begin{prop}\cite[Proposition 3.13]{Fir16}
The quotient simplicial complex $\Gamma\backslash X$ is finite if and only if $G\backslash X^0$ is finite and $\Gamma$ is cocompact in $G$.
    
\end{prop}

We conclude with the following definition.

\begin{dfn}\cite[Appendix A]{AB08}
Let $X$ be a simplicial complex of dimension $d$ and $\Delta_d$ a fixed simplex of dimension $d$. We call $X$ \emph{colorable} if there exists a simplicial map $f_X:X\rightarrow \Delta_d$ that maps a simplex in $X$ to a simplex in $\Delta_d$ of the same dimension.
    
\end{dfn}

Identifying $\Delta_d$ with ${0,1,\dots,d}$, the map $f_X$ sends each $i$-simplex $F$ of $X$ to a subset of $\{0,1,\dots,d\}$ of cardinality $i+1$, which is called the \emph{type} of $F$.

\subsubsection{Bruhat--Tits Buildings}\label{sub:BT}

Bruhat--Tits theory assigns to a reductive group $G$ over a non-Archimedean local field $\F$ an infinite simplicial complex with a $G$-action. We only sketch some results and leave \cite{Tit79}, \cite{AB08}, \cite{KP23} for more details.

Assume $G$ is simple over $\F$ and write $G=G(\F)$ for short.  Let $\mc B=\mc B(G)$ be the reduced Bruhat--Tits building of $G$, which is a simplicial complex of dimension $d$ with $d$ being the $\F$-semisimple rank of $G$. Moreover, there is a related $G$-action on $\mc B$ which is simplicial and acts transitively on the set of chambers. 
If we further assume the center of $G$ to be finite, then the stabilizer of any point in $\mc B$ is an open compact subgroup of $G$, making $\mc B$ a $G$-complex\footnote{In general, we may replace $G$ by $G^{\mathrm{sc}}$ or $G
_{\mathrm{ad}}$.}. Moreover, $\mc B$ is colorable, allowing us to discuss the types of its facets.

Given a facet $F$ in $\mc B$, there is an associated open compact subgroup $P_{F}$ of $G(\F)$, called a parahoric subgroup, that fixes $F$ pointwise. 
More generally, a related group scheme $\mc G_F$ over $\mc O_{\F}$ can also be constructed whose generic fiber is $G$ and whose $\mc O_{\F}$-points form $P_F$.  
In the literature, there are discrepancies in the definition of a parahoric group. For us, we require the group scheme $\mc G_F$, or equivalently its special fiber to be connected. More precisely, our group $P_{F}$ is exactly the group $G(\F)_{F}^0$ in \cite[Definition 7.4.5.(3)]{KP23}. When $F$ is a chamber, the related parahoric subgroup is called an Iwahori subgroup. At the other extreme, if $F$ is a vertex, then the related parahoric subgroup is maximal among all of them. We may introduce the terminology of ``special, extra-special and hyperspecial'' vertices following \cite[Definition 1.3.39, 7.11.1]{KP23}. In particular, a vertex is hyperspecial if and only if its related group scheme is reductive (\cite[Theorem 9.9.3]{KP23}). Finally, we define the type of a parahoric subgroup as the type of the related facet. Then two parahoric subgroups of the same type are conjugate to each other (but the converse is  not always true in general).

Now we consider a closed discrete subgroup $\Gamma$ of $G$. We consider the related quotient $\Gamma\backslash \mc B$ as a topological space. Under the two conditions of Proposition \ref{prop:GammaonX}, the quotient $[X]:=\Gamma\backslash \mc B$ inherits a simplicial complex structure.

Furthermore, the quotient $G\backslash \mc B^0$ is always finite. If we further assume $\Gamma$ to be cocompact in $G$, then the related simplicial complex $[X]$ is finite.

In the cocompact case, we  give a more practical assumption of $[X]$ being a quotient simplicial complex. One such assumption is: 

\begin{assm}\label{a:GammaXsimplicialw}
For any nontrivial $\gamma \in \Gamma$ and any vertex $v$, the distance between $\gamma v$ and $v$ in the $1$-skeleton of $\mc B(G)$ is greater than $2$. 

\end{assm}

In particular, for any facet $F$, we have $\gamma \bar F \cap \bar F = \emptyset$, and the assumptions of Proposition \ref{prop:GammaonX} are satisfied.

\begin{rmk}

 Assume that $\Gamma$ acts freely and preserves types of vertices---these are natural and mild assumptions.  
Assume there exists a vertex $v$ and $1\neq \gamma\in \Gamma$ such that $\gamma v$ and $v$ are at distance less than $3$. First note that the distance can only be 2, otherwise $\gamma v$ and $v$ are either the same or of the same type. In the distance-2 case we may find another vertex $v_0$ adjacent to both $v$ and $\gamma v$. If we have $\{\Gamma v_0,\Gamma v\}=\{\Gamma v_0,\Gamma \gamma v\}$ implying that $\Gamma\{v_0,v\}=\Gamma\{v_0,\gamma v\}$, then either $v$ and $\gamma' v$ are adjacent for some $\gamma'\in\Gamma$, or $v_0$ is fixed by an element of $\Gamma$, which is impossible. Thus this assumption is, in some sense, necessary.
 
\end{rmk}

\begin{rmk}
Using the formalism of simplicial sets instead of simplicial complexes would allow us to define quotients much more easily and in all cases. We do not do this for consistency with the previous literature on Ramanujan complexes. 
\end{rmk}

Theoretically such $\Gamma$ can always be found. Since in the cocompact case, there are only finitely many $\Gamma$-orbits of vertices, for each representative $v_i$ we can simply find a related normal sublattice $\Gamma_i$ of $\Gamma$ such that vertices in $\Gamma_i v_i$ are at distance greater than 2. Then the sublattice $\cap_i\Gamma_i$  satisfies the assumption. 
In practice, the above assumption can be verified directly for a given class of cocompact lattices defined by congruence subgroups (for instance, $\Lambda_{v_0}$ in Assumption \ref{a:main}) for sufficiently large level.

\subsection{Ramanujan Complexes}

Ramanujan complexes are particular finite simplicial complexes that generalize the well-established notion of Ramanujan graphs (\emph{cf.} \cite[Definition 4.5.6]{Lub94}) to higher dimensions.

Following \cite{Lub18}, we consider Ramanujan complexes arising as quotients of Bruhat--Tits buildings. Let $\Gamma\subset G$ and $\mc B$ be as above. In particular, the quotient $[X]:=\Gamma\backslash \mc B$ is a finite simplicial complex. Let $I$ be an Iwahori subgroup of $G$, which in particular fixes a chamber in $\mc B$.

Associated to $\Gamma$ is the space of functions $L^2(\Gamma \bs G)^{I} = \mc C^\infty(\Gamma \bs G)^{I}$ where $I$ is an Iwahori subgroup. This is a module for the Iwahori Hecke algebra
\[
\ms H := \ms H(G, I) := \mc C_c^\infty(I \bs G / I).
\]
The action of $\ms H$ on this function space encodes the information of many interesting combinatorial adjacency and boundary operators on various spaces of functions on the simplices of $\mc B$. For example, if $G$ is semisimple and simply connected, then $\mc C^\infty(\Gamma \bs G)^{I}$ is the set of functions on the chambers of $\Gamma \bs \mc B$. 

Representations of $G$ with an $I$-fixed vector correspond bijectively to irreducible $\ms H$-modules through $\pi \mapsto \pi^{I}$. Let $\check G^{I, \temp}$ be the set of such modules corresponding to tempered $\pi$. An irreducible $\ms H$-module $\pi^{I}$ is weakly contained in $L^2(G)^{I}$ if and only if $\pi \in \check G^{I, \temp}$, motivating:

\begin{dfn}[{\cite[Definition 2.7]{Lub18}}]\label{def ramanujancomplex}
For $\Gamma \subseteq G$ as above, we say that $[X]=\Gamma \bs \mc B$ is a $\emph{Ramanujan complex}$ if for all irreducible $\ms H$-modules $\pi^{I} \subseteq \mc C^\infty(\Gamma \bs G)^{I}$, we have that either $\pi^{I} \in \check G^{I, \temp}$ or  $\pi^{I}$ comes from a $1$-dimensional representation $\pi$ of $G$. 
\end{dfn}

In other words, the spectra of these various combinatorial operators on $\Gamma \bs \mc B$, up to exceptions coming from constant functions, are contained in that of those operators on $\mc B$ itself. This should be seen as a higher-dimensional generalization of the notion of a $d$-regular Ramanujan graph that, except for the constant function, the adjacency operator has the spectrum contained in that of the adjacency operator for the infinite regular tree. 

\begin{rmk}

In \cite{Fir16}, Uriya First proposed an alternative definition of Ramanujan complexes $\Gamma\backslash X$ (with $X$ being a general $G$-complex) using the spectrum of a collection of related operators. We refer to \emph{ibid.}, especially \cite[Theorem 6.22]{Fir16} for the relationship between this definition and Definition \ref{def ramanujancomplex}.

\end{rmk}

\begin{rmk}
 In the classical setting where $\F=\Q_p$ and $G=\PGL_2(\Q_p)$, the set of vertices $[X]^0$ and edges $[X]^1$ of the finite $(p+1)$-regular graph $[X]=\Gamma \bs \mc B$ are given by $\Gamma \bs G / K$ and $\Gamma \bs G / I$ respectively, where $K$ is a maximal compact subgroup of $G$. In \cite[Theorem 5.5.1]{Lub94}, it is shown that $\lambda$ is an eigenvalue of the adjacency matrix $\delta$ as an operator acting on $L^2(\Gamma \bs G / K)=L^2(\Gamma \bs G)^{K}$ if and only if the related unramified representation $\rho^{\lambda}$ occurs in $L^2(\Gamma \bs G)$. Proving that $\Gamma \bs \mc B$ is Ramanujan is equivalent to showing that  all unramified representations occurring in $L^2(\Gamma \bs G / K)$ are either principal series or one-dimensional (\cite[Theorem 5.4.3]{Lub94}). Equivalently, we need to show that those spherical representations occurring in $L^2(\Gamma \bs G)$ are either tempered or one-dimensional.
 This justifies Definition \ref{def ramanujancomplex}.   
\end{rmk}

\section{Abstract Construction of Complexes}

\subsection{Relation to Automorphic Spectrum}

Let $F$ be a number field, and let $\infty$ denote the set of Archimedean places of $F$.  
Let $G$ be a connected reductive group over $F$ such that $G_v$ is compact for all $v \in \infty$.  
Fix a finite place $v_0$.

Let $K$ be a compact open subgroup of $G(\A_F)$. For the related compact open subgroup $K^{\infty, v_0} \subseteq G^{\infty, v_0}$, define
\[
\Gamma_{v_0} := G(F) \cap K^{\infty, v_0}.
\]
This gives an embedding 
\begin{equation}\label{eq: adelicembedding}
\Gamma_{v_0} \backslash G_{v_0} / I_{v_0} \hookrightarrow 
G(F) \backslash G(\A_F) / K^{\infty, v_0} I_{v_0} G_\infty,
\end{equation}
where $I_{v_0}$ is a fixed Iwahori subgroup of $G_{v_0}$  
By \cite[Theorem 2.6.1]{GH24}, the adelic quotient
\[
G(F) \backslash G(\A_F) / K^{\infty, v_0} I_{v_0} G_\infty
\]
is finite, and therefore so is $\Gamma_{v_0} \backslash G_{v_0} / I_{v_0}$.

Let $\mathcal B(G_{v_0})$ be the reduced Bruhat--Tits building of $G_{v_0}$, viewed as a polysimplicial complex. 
Define the quotient complex $\Gamma_{v_0} \backslash \mathcal B(G_{v_0})$ as follows:
\begin{itemize}
    \item Its $0$-simplices are the $\Gamma_{v_0}$-orbits of $0$-simplices of $\mathcal B(G_{v_0})$;
    \item Its $k$-simplices are the $\Gamma_{v_0}$-orbits of $k$-simplices of $\mathcal B(G_{v_0})$, i.e.\ distinct simplices $\{[x_0], [x_1],\dots,[x_k]\}$ represented by equivalence classes of $k$-simplices $\{x_0, x_1,\dots,x_k\}$ in $\mathcal B(G_{v_0})$.  
    (Some of these may be degenerate.)
\end{itemize}
This complex is finite. Indeed, let $C$ be a fixed chamber in $\mathcal B(G_{v_0})$. Then the map
\[
\Gamma_{v_0} \backslash G_{v_0} / I_{v_0} \longrightarrow 
\Gamma_{v_0} \backslash \{\text{chambers of }\mathcal B(G_{v_0})\},
\qquad g \longmapsto g \cdot C,
\]
is surjective. Hence the quotient has finitely many chambers and, consequently, finitely many $k$-simplices for all $k$.

As explained in the previous section, in general the quotient complex need not inherit a simplicial structure from $\mathcal B(G_{v_0})$.  
It does inherit one if the assumptions in Proposition~\ref{prop:GammaonX} hold, and in particular if Assumption~\ref{a:GammaXsimplicialw} is satisfied.

\begin{dfn}\label{def:complex}
Denote the simplicial complex defined above by
\[
X(K^{\infty, v_0}) := X_G(K^{\infty, v_0}),
\]
and set
\[
\mathcal C^\infty(X_G(K^{\infty, v_0})) := \mathcal C^\infty(\Gamma_{v_0} \backslash G_{v_0})^{I_{v_0}}.
\]
\end{dfn}
This may be viewed roughly as the space of functions on the top-dimensional simplices of $X(K^{\infty, v_0})$, although the description is slightly more subtle when $G$ contains elements acting on a top-dimensional simplex by a nontrivial automorphism.

The key property we will use is that the embedding in \eqref{eq: adelicembedding} induces
\[
\mathcal C^\infty(X_G(K^{\infty, v_0})) \hookrightarrow 
L^2(G(F)\backslash G(\mathbb A_F))^{K^{\infty, v_0} I_{v_0} G_\infty}.
\]
As $\ms H(G_{v_0},I_{v_0})$-modules, this can be expressed as an injection
\[
\mathcal C^\infty(X_G(K^{\infty, v_0})) \hookrightarrow
\bigoplus_{\substack{\pi \in \mathcal{AR}_\mathrm{disc}(G) \\ \pi_\infty \ \mathrm{trivial}}}
(\pi_{v_0})^{I_{v_0}} \otimes (\pi^{\infty, v_0})^{K^{\infty, v_0}},
\]
where $\mathcal{AR}_\mathrm{disc}(G)$ denotes the multiset of automorphic representations in the discrete spectrum of $G$, and the Hecke action is on the first tensor factor.  
Applying Definition~\ref{def ramanujancomplex} to this decomposition yields:

\begin{prop}\label{prop ramancond}
The complex $X_G(K^{\infty, v_0})$ is Ramanujan 
if and only if for every 
$\pi \in \mathcal{AR}_\mathrm{disc}(G)$ such that $\pi_\infty$ is trivial and 
$(\pi^{\infty, v_0})^{K^{\infty, v_0}} \neq 0$, the representation $\pi_{v_0}$ is either tempered or a one-dimensional representation (i.e. a character).
\end{prop}

\subsection{Super-Definite Unitary Groups}

In \cite{Clo93Coho}, Clozel studies special inner forms of unitary groups whose associated arithmetic locally symmetric spaces have cohomology essentially concentrated in middle degree.  
An analogous phenomenon was used in the case of forms of $\GL_n$ over function fields to construct Ramanujan complexes in \cite{LSV05a}.

The construction in \cite{LSV05a} relies on the following two properties:

\begin{assm}\label{assm sdef}
The group $G$ satisfies:
\begin{enumerate}
    \item $G_v$ is compact for all $v \mid \infty$;
    \item there exists a finite place $v_a$ such that $G_{v_a}$ is anisotropic modulo its center.  
    Such a place is called an \emph{anisotropic place} of $G$.
\end{enumerate}
\end{assm}

The second condition is more restrictive than it may appear:

\begin{thm}[{\cite[Theorem 10.3.1]{KP23}}]
The only absolutely almost simple groups $G_{v_a}$ over non-Archimedean local fields that are anisotropic are of type~$A$.  
More precisely, $G_{v_a}/Z_{G_{v_a}}$ is isogenous to $D_{v_a}^\times / F_{v_a}^\times$ for some division algebra $D_{v_a}$ over $F_{v_a}$.
\end{thm}

We will use the fact that an anisotropic place admits a large open compact subgroup:

\begin{lem}\label{lem coabelian}
Let $G_{v_a}$ be anisotropic modulo its center.  
Then there exists an open compact subgroup $K_0 \subseteq G_{v_a}$ such that $G_{v_a}/K_0$ is abelian.
\end{lem}

\begin{proof}
Let $G_{v_a}^1$ be the intersection of the kernels of $g \mapsto |\chi(g)|_{v_a}$ as $\chi$ ranges over all $F_{v_a}$-rational characters of $G$.  
Then $G_{v_a}^1$ contains $G_{v_a}^\mathrm{der}$, hence $G_{v_a}/G_{v_a}^1$ is abelian.  
Moreover, $G_{v_a}^1$ is isogenous to the norm-$1$ subgroup of a division algebra and is therefore compact.  
Thus, $K_0 := G_{v_a}^1$ satisfies the claim.
\end{proof}

Groups defined over number fields and satisfying the properties in Assumption~\ref{assm sdef} indeed exist:

\begin{prop}[{\cite[\S0.3.3]{KMSW14}}]\label{prop sdefexistence}
Let $F$ be a totally real number field and $E/F$ a totally complex quadratic extension.  
Let $G^* := U^{E/F}_N$ be the associated quasisplit unitary group.  
Then $G^*$ has an extended pure inner form $G$ satisfying the properties in Assumption~\ref{assm sdef}.
\end{prop}

\begin{dfn}
We call the groups constructed in Proposition \ref{prop sdefexistence}   \emph{super-definite unitary groups}.
\end{dfn}

\subsection{Complexes from Unitary Groups}

We can now abstractly construct our Ramanujan complexes using the same methods as in \cite{LSV05b} for general linear groups over function fields.

\begin{thm}\label{thm:abstractramanujan}
Let $E/F$ be a CM extension of number fields and let $G$ be a super-definite unitary group that is an extended inner form of the quasi-split unitary group $G^* = U^{E/F}(N)$. Fix a finite place $v_0$ and a compact open subgroup $K^{\infty, v_0} \subseteq G^{\infty, v_0}$. Assume that one of the following holds:
\begin{enumerate}
    \item $N$ is prime;
    \item there exists an anisotropic place $v_a$ of $G$ such that $G_{v_a}/K_{v_a}$ is abelian. 
\end{enumerate}
Then $X_G(K^{\infty, v_0})$ is a Ramanujan complex.
\end{thm}

\begin{rmk}
Condition (2) is satisfied if for example $K_{v_a}$ is the subgroup of norm one elements in $G_{v_a} \cong D_{v_a}^\times$, as in Lemma \ref{lem coabelian}.
\end{rmk}

\begin{proof}
We verify the condition of Proposition \ref{prop ramancond}.

Let $\pi \in \mc{AR}_{\disc}(G)$ be such that $\pi_\infty$ is trivial and $(\pi^{\infty,v_0})^{K^{\infty,v_0}} \neq 0$. Let $\psi$ be the Arthur parameter of $\pi$ as in \cite[Theorem 1.7.1]{KMSW14}. Write
\[
\psi = \bigoplus_{i=1}^k \tau_i[d_i],
\]
where the $\tau_i$ are simple parameters. Let $v_a$ be an anisotropic place of $G$. Since $\psi_{v_a}$ is relevant and $G_{v_a}$ is anisotropic (and hence its dual group admits no proper relevant parabolic subgroups), we must have $k=1$.

\noindent\underline{Case 1: $N$ prime.}

In this case, either $\dim \tau_1 = 1$ and $d_1 = N$, or $\dim \tau_1 = N$ and $d_1 = 1$.

In the first situation, let $v_s$ be a split place of $E/F$ such that $G_{v_s}$ is split. By the description of $A$-packets for $\GL_N$ (which coincide with the corresponding singleton $L$-packets, see \cite[\S1.2.2]{KMSW14}), the local representation $\pi_{v_s}$ is a character. By \cite[Lemma 6.2]{KST16}, this implies that the global representation $\pi$ is a character, and hence $\pi_{v_0}$ is also a character.

In the second situation, since $\pi_\infty$ is an irreducible representation of a compact group, it is finite-dimensional and therefore cohomological. By \cite[Corollary 4]{NP21}, the global parameter $\tau_1$ is cohomological, which by the main result of \cite{Car12} implies that $(\tau_1)_{v_0}$ is tempered. By the local Langlands correspondence for unitary groups \cite[Theorem 1.6.1(6)]{KMSW14}, it follows that $\pi_{v_0}$ is tempered.

\noindent\underline{Case 2: $G_{v_a}/K_{v_a}$ abelian.}

If $G_{v_a}/K_{v_a}$ is abelian, then $\pi_{v_a}$ is a character. By \cite[Proposition 3.7]{Bad08}, the local Arthur parameter $\psi_{v_a}$ must be of the form $\chi_0[1][N]$ or $\chi_0[N][1]$, for some self-dual character $\chi_0$ of $\GL_1(E \otimes_F F_{v_a})$. Indeed, in Badulescu's notation, $\psi_{v_a}$ corresponds to a representation of $\GL_N$ that transfers to a character of $G_{v_a}$. More precisely:
\begin{itemize}
\item In case (a) of \cite[Proposition 3.7]{Bad08}, $C(\sigma)$ is a character and $k=1$, so $\psi_{v_a} = \chi_0[N][1]$.
\item In case (b) of \cite[Proposition 3.7]{Bad08}, $C(\tau)$ is a character and $l=1$, so $\psi_{v_a} = \chi_0[1][N]$.
\end{itemize}
In either case, the global Arthur parameter $\psi$ is restricted to the two possibilities appearing in Case 1, and the same argument applies.
\end{proof}

\begin{rmk}\label{rem:abstractramanujanff}
At present, we are only able to carry out the construction over number fields. While the rapidly developing theory over function fields should eventually allow a generalization to positive characteristic, extending the above argument encounters two main obstacles:

First, we only have a transfer of $\pi$ from $G$ to $\Res^E_F \GL_n$ when the parameter of $\pi$ pushes forward to a cuspidal representation, since the global Langlands correspondence of \cite{Laf18} applies only to cuspidal representations. Second, we do not know that the $L$-packet associated to a tempered parameter at a non-split place consists entirely of tempered representations, as the global--local compatibility of \cite{GL17} provides a local Langlands correspondence only up to semisimplification.

Overcoming these difficulties would allow the argument to extend to the function field case.
\end{rmk}

\subsection{Structure}\label{ssec:abstractstructure}
The local structure of $X_G(K^{\infty, v_0})$ can be read off from the label of $\mc B(G_{v_0})$ in the tables of \cite{Tit79}. In our case of unitary groups, we get labels (see \cite[Table 3.1]{DEP}):
\begin{itemize}
    \item $A_n$ when $v_0$ is split,
    \item ${}^2 \! A'_n$ when $v_0$ is inert and $G_{v_0}$ is quasisplit,
    \item $B$-$C_{n/2}$ if $n$ is even, $G_{v_0}$ is quasisplit, and $v_0$ is ramified,
    \item ${}^2 \! A''_n$ if $n$ is even, $G_{v_0}$ is non-quasisplit, and $v_0$ is inert,
    \item ${}^2 \! B$-$C_{n/2}$ if $n$ is even, $G_{v_0}$ is non-quasisplit, and $v_0$ is ramified,
    \item $C$-$BC_{(n-1)/2}$ if n is odd and $v_0$ is ramified. 
   
\end{itemize}
For simplicity, we focus on the $A_n$ and ${}^2 \! A'_n$ cases in this paper. There is no essential difficulty in extending constructions beyond---just slightly more combinatorial complexity. 

The affine root system underlying the label determines the structure of the apartment, with types of simplices corresponding to subsets of the vertices in the corresponding diagram in an inclusion-reversing way. The valency of one type of simplex in another can be computed as the size of the flag variety in the reductive quotient of the special fiber of the reductive model corresponding to the simplex stabilizer. Computing these valencies in general requires input from Bruhat--Tits theory and formulas for point counts of finite groups of Lie type (e.g. \cite[\S2]{ATLAS}). This is easier for top-dimensional facets of a chamber: each one corresponds to leaving out one vertex in the affine Dynkin diagram and the number of chambers containing it is $q_{v_0}^d +1$, where $d$ is the integer attached to the excluded vertex in \cite{Tit79}. 

Finally, while at split places the type-$A_n$ local structure of these complexes is the same as in the constructions of \cite{LSV05a}, their universal cover is a building for a group over a $p$-adic field instead of a function field. Since buildings for groups of split, semisimple rank $\geq 4$ encode the information of their base field (see \cite[\S11.9]{AB08} for a summary), these universal covers are necessarily different when $G$ has full rank $\geq 5$.

We work out all these details precisely for the $A_n$ and ${}^2 \! A'_n$ cases in rank-$5$ in \S\S\ref{sssec:splitcomplextype}, \ref{sssec:urcomplextype} respectively.

\section{Explicit Construction of Complexes}\label{sec:explicitconstruction}

Let $F$ be a totally real number field and let $E/F$ be a quadratic CM extension. We now describe how to make the abstract construction of $X_G(K^{\infty,v_0})$ explicit. Our approach adapts the ideas from the function field case \cite{LSV05a}. 

\subsection{Class Number One}

We require one additional assumption on $G$:

\begin{dfn}
Let $G/F$ be a reductive group such that $G_\infty$ is compact. We say that an open compact subgroup $K^\infty \subseteq G(\A^\infty)$ has class number $1$ if $G(\A^\infty) = K^\infty G(F)$.

We call $K^\infty$ \emph{golden} if, in addition, $K^\infty \cap G(F) = \{1\}$.  
\end{dfn}

The key property is:

\begin{lem}[e.g. {\cite[Lemma 4.1.3(1)]{DEP}}]\label{lem:simpletransitiveaction}
Let $G/F$ be a connected reductive group such that $G_\infty$ is compact, and let $K^\infty \subseteq G(\A^\infty)$ have class number $1$. Then for every finite place $v_0$, the group $\Lambda_{v_0} := G(F) \cap K^{\infty, v_0}$ acts transitively on $G_{v_0}/K_{v_0}$.

If $K^\infty$ is furthermore golden, then $\Lambda_{v_0}$ in addition acts simply. 
\end{lem}

We can now list the precise inputs and assumptions for our explicit construction of $X_G(K^{\infty,v_0})$:

\begin{assm}\label{a:main}
Assume the following conditions on $G$ and $K^{\infty, v_0}$:
\begin{itemize}
    \item $G$ is a super-definite unitary group of rank $N$ with respect to $E/F$, such that there exists $K^\infty_0 \subseteq G(\A^\infty)$ with class number $1$.
    \item Either $N$ is prime, or there exists an anisotropic place $v_a$ such that $G_{v_a}/K_{0,v_a}$ is abelian.
    \item $K^\infty_0$ is golden.
    \item $v_0$ is a finite split or inert place of $F$ such that $G_{v_0}$ is quasisplit and $K_{0,v_0}$ is hyperspecial\footnote{For simplicity, we exclude the ramified and non-quasisplit cases.}. 
    \item $K^{\infty,v_0} \subseteq K_0^{\infty,v_0}$ is a normal subgroup. For example, choose an ideal $\mf n$ of $\mc O_F$ relatively prime to $v_0$ and all places where $G$ or $K_0^\infty$ is ramified, and let $K^\infty := K(\mf n) \subseteq K_0^\infty$ be the corresponding principal congruence subgroup.
\end{itemize}

As convenient notation, whenever these assumptions are satisfied, we define the lattices:
\begin{itemize}
    \item $\Lambda_{0,v_0} := G(F) \cap K_0^{\infty, v_0}$.
    \item $\Lambda_{v_0} := G(F) \cap K^{\infty, v_0}$.
\end{itemize}
\end{assm}

\begin{rmk}\label{rem:superdefiniteexamples}
There are many known examples of $G$ and $K^{\infty,v_0}$ satisfying Assumption \ref{a:main}: see \cite{PS18} for the rank-$2$ case, and \cite{EFMP23} for rank-$3$. However, these small ranks only produce Ramanujan graphs instead of higher-dimensional complexes. The paper \cite{MSG12} constructs an example in rank $5$ which will be our main object of study for part 2 of this paper. Unfortunately, it also proves that there are no other examples over number fields in rank $\geq 5$. 

Finally, we expect many examples in rank $4$---for example, for $E/F = \Q(\sqrt{-7})/\Q$, numerics with the mass formulas of \cite{MSG12} suggest that a rank-$4$ unitary group that is compact at infinity, equivalent to the units of a division algebra over $2$, non-quasisplit at either $3$ or $7$ together with a $K^\infty_0$ that is special at $2$, extra-special at the non-quasisplit prime, and special everywhere else should work.
\end{rmk}
 
\begin{rmk}
The assumption that $K^\infty_0$ is golden can be weakened to class-number-one at the cost of making Constructions \ref{sssec:splitconstruction} and \ref{sssec:nonsplitconstruction} combinatorially more painful---basing them off a Schreier coset graph instead of a Cayley graph. We do not present this for brevity. 
\end{rmk}

\subsection{Construction Overview}

We now describe an explicit construction of the complex $X_G(K^{\infty, v_0})$ under the conditions of Assumption \ref{a:main}. For consistency of notation across the different splitting types of $v_0$, define
\begin{equation}\label{eq:barlambda}
\bar \Lambda_{0, v_0} := \Lambda_{0,v_0}/(\Lambda_{0,v_0} \cap Z_{G_{v_0}}^\spl),
\end{equation}
where $Z_{G_{v_0}}^\spl$ denotes the maximal split torus in the center of $G_{v_0}$.

The construction relies on finding certain special elements called ``gates'':

\begin{dfn}\label{def:gates}
Let $v_0$ be an unramified finite place of $F$, and let $\bar \Lambda$ be a group acting transitively on $G_{v_0} x_0$ for some hyperspecial point $x_0 \in \mc B_{v_0}$. Then define the \emph{gates} of $\bar \Lambda$ to be
\[
S_{\bar \Lambda} := \{s \in \bar \Lambda : sx_0 \text{ is at graph distance } d_{v_0} \text{ from } x_0\},
\]
where $d_{v_0} = 1$ if $v_0$ is split and $d_{v_0} = 2$ if $v_0$ is inert.

If $\bar \Lambda$ in addition acts simply transitively on $G_{v_0} x_0$, then for each $x$ at graph distance $d_{v_0}$ from $x_0$, let $s_x \in S_{\bar \Lambda}$ denote the unique element satisfying $s_x x_0 = x$.
\end{dfn}

\begin{rmk}
By the work of \cite{DEP}, the gates of $\bar \Lambda_{0,v_0}$ form a set of ``golden gates'' for $G_\infty$ as in \cite[Definition 1.2.1]{DEP}. Note that this is independent of the work of \cite[\S6]{DEP} and in particular the rank $4$ or $8$ condition needed for \cite[Conjecture 6.3.4]{DEP} since the Ramanujan conjecture holds on the nose for super-definite $G$ by Theorem \ref{thm:abstractramanujan}. 
\end{rmk}

Theorem \ref{thm:abstractramanujan} allows us then explicitly construct Ramanujan complexes. We present both the split and inert cases as Cayley graph-type constructions:

\subsubsection{The split case}\label{sssec:splitconstruction}
If $v_0$ is split, then $G_{v_0}/K_{0,v_0}$ is the set of vertices of the extended Bruhat--Tits building $\td{\mc B}_{v_0}$ associated with $G_{v_0}$, and $\Lambda_{0,v_0}$ acts simply transitively on these vertices by Lemma \ref{lem:simpletransitiveaction}. Modding out by the center, $\bar \Lambda_{0,v_0}$ acts simply transitively on the vertices of the reduced building $\mc B_{v_0}$.

Fix a point $x_0 \in \mc B_{v_0}$ stabilized by $K_{0,v_0}$, and define $\bar \Lambda_{v_0}$ analogous to $\bar \Lambda_{0,v_0}$. Since the action is simply transitive, the elements $s_x$ in $S_{\bar \Lambda_{0,v_0}}$ generate $\bar \Lambda_{0,v_0}$, and the $1$-skeleton of $\mc B_{v_0}$ is the corresponding Cayley graph.

This allows us to construct:

\begin{construction}\label{cons:split}
Let $\mc X$ be the simplicial complex that is the union of all simplices containing $x_0$, let $Q$ be the vertices of $\mc X$ different from $x_0$, and consider the gate set 
\[
S_{\bar \Lambda_{0,v_0}} = \{s_x \,:\, x \in Q\}
\]
of $\bar \Lambda_{0, v_0}$. Then 
we have the following simplicial structure:
\begin{itemize}
\item The $0$-simplices are the elements of $\bar \Lambda_{v_0} \bs \bar \Lambda_{0,v_0}$. 
\item For $r > 0$, the $r$-simplices are the subsets of the form $x_1 \cup \{s_x x_1, x \in \Delta\}$ for $x_1 \in \bar \Lambda_{v_0} \bs \bar \Lambda_{0,v_0}$ and $x_0 \cup \Delta$ an $r$-simplex of $\mc X$. 
\end{itemize}
\end{construction}

\subsubsection{The inert case}\label{sssec:nonsplitconstruction}
Let $x_0$ be the hyperspecial vertex stabilized by $K_{0,v_0}$, which by our convention we take to be of type $0$. A key property is that $G_{v_0}x_0$ coincides with the set of type-$0$ vertices of $\mc B_{v_0}$. These are precisely the vertices at even graph distance from $x_0$ (see, e.g., \cite[Example 3.3.5]{DEP}). By Lemma \ref{lem:simpletransitiveaction}, $\bar \Lambda_{0,v_0} = \Lambda_{0,v_0}$ acts simply transitively on $G_{v_0}/K_{0,v_0}$, and hence also on the set of even-distance vertices.

In addition, the lattice formed by the vertices in an apartment is exactly half the cocharacter lattice. Thus, each vertex $x$ at distance $2$ from $x_0$ determines a unique vertex $x'$ at distance $1$ that is the midpoint of the segment $[x,x_0]$. Conversely, by extending the line from $x_0$ through such an $x'$ to twice its length in any apartment, every vertex at distance $1$ can be obtained in this way.

We can therefore carry out a construction analogous to that in Construction \ref{cons:split}, with the additional step of ``filling in'' the vertices of other types.

\begin{construction}\label{cons:nonsplit}
Let $Q = \{x_1, \dotsc, x_l\}$ be the set of type-$0$ vertices at distance exactly $2$ from $x_0$ and let $\mc X$ be the simplicial complex that is their convex hull. For each $x_i \in Q$, there is a unique vertex at distance $1$ in the convex hull of $x_0$ and $x_i$. The fibers of the resulting map define a partition $Q = \bigsqcup T_j$. We label the vertices at distance $1$ from $x_0$ as $y_{T_j}$ according to the corresponding piece of the partition. Note that the $y_{T_j}$ range over all vertices at distance $1$ from $x_0$. 

Consider the gate set 
\[
S_{\bar \Lambda_{0,v_0}} = \{s_x \,:\, x \in Q\}.
\]
Then 
we have the following simplicial structure:
\begin{itemize}
    \item A $0$-simplex for each element of $\Lambda_{v_0} \backslash \Lambda_{0,v_0}$.
    \item A $0$-simplex for each subset of the form $x \cup \{s_t x : t \in T_j\}$, for $x \in \Lambda_{v_0} \backslash \Lambda_{0,v_0}$ and one of the $T_j$. We denote this $0$-simplex by $y_{T_j}(x)$, though this labeling is not unique.
    \item A top-dimensional simplex (i.e., a simplex of dimension $\lfloor N/2 \rfloor$) for each collection of $0$-simplices of the form
    \[
    x \cup \{y_{T_j}(x) : y_{T_j} \in \Delta\}
    \]
    where $x \in \Lambda_{v_0} \backslash \Lambda_{0,v_0}$ and $\Delta$ is a top-dimensional simplex of $\mc X$ containing $x_0$.
    \item All subsimplices of the top-dimensional simplices above.
\end{itemize}

\end{construction}

\begin{rmk}
Except for the simplifying result characterizing the possible $y_{T_i}$, Construction \ref{cons:nonsplit} applies whenever $G_{v_0} x_0$ comprises all vertices of the same type as $x_0$. In the unitary case with $v_0$ non-split, this condition holds in all cases except when $N$ is even, $v_0$ is ramified, and $G_{v_0}$ is quasisplit; see, for example, \cite[Table 3.1]{DEP}.

\end{rmk}

\begin{prop}\label{prop:explicitramanujan}
Let $G$, $v_0$, $K^\infty_0$ and $K^{\infty, v_0}$ satisfy Assumption \ref{a:main}, with $v_0$ split or inert and the related $\bar \Lambda_{v_0}$ or $\Lambda_{v_0}$ satisfying Assumption \ref{a:GammaXsimplicialw}. Then Constructions \ref{cons:split} and \ref{cons:nonsplit} provide an explicit description of the Ramanujan simplicial complex $X_G(K^{\infty, v_0})$.
\end{prop}

\begin{rmk}

We emphasize that the Assumption~\ref{a:GammaXsimplicialw} plays a crucial role of guaranteeing that the complexes constructed in Construction \ref{cons:split} and \ref{cons:nonsplit} are indeed the quotient complexes $\bar \Lambda_{v_0}\backslash\mc B_{v_0}$ and $\Lambda_{v_0}\backslash\mc B_{v_0}$ respectively. On the other hand, for small enough $K^{\infty,v_0}$ the related $\bar \Lambda_{v_0}$ and $\Lambda_{v_0}$ indeed satisfy this assumption, thus to the end of constructing an infinite family of Ramanujan complexes such an assumption won't cause any trouble.

\end{rmk}

\subsection{Finding Gates}\label{ssec:findinggates}

The constructions in Proposition \ref{prop:explicitramanujan} require finding the gate sets $S_{\Lambda_{0, v_0}}$. These can be done analogously to \cite[\S\S3.3.1, 4.3.1]{DEP}.

\subsubsection{Setup}\label{sub:setup}
Let $B$ be a central simple algebra over $E$. An involution of the second kind on $B$ is a map $\iota$ satisfying $\iota(ab) = \iota(b) \iota(a)$ and such that $\iota$ acts as Galois conjugation over $F$ on $E \subseteq B$.

By the classification of inner forms of unitary groups (see, for example, \cite[\S0.8]{KMSW14}), there exists a central simple algebra $B$ of dimension $N \times N$ over $E$ with an involution of the second kind $\iota$ such that
\[
G = \{x \in \Res^E_F B : \iota(x)x = 1\},
\]
where $\Res^E_F B$ is given by $\Res^E_F B(R) = B \otimes_F R$. Define $\theta$ to be the composition of $\iota$ with inversion as an involution on $\Res^E_F B^{\times}$.

From now on, assume that $v_0$ is a place not lying over $2$ where $B$ is split. Thus, $B_{v_0} \cong \Mat_{N \times N}(E_{v_0})$ and $G_{v_0}\cong \GL_{N}(E_{v_0})$, where we abbreviate $E_{v_0} = E \otimes_F F_{v_0}$. In particular, the involution $\theta$ induces an affine embedding of reduced buildings (\emph{cf.}\ \cite[Theorem 12.7.1, Section 14, p.\ 490–491]{KP23})
\begin{equation}\label{eq Bembed}
\mc B(G_{v_0}) \hookrightarrow \mc B(B_{v_0}^{\times}),
\end{equation}
whose image is precisely the set of $\theta$-fixed points in $\mc B(B_{v_0}^{\times})$. Note that the condition $Z(G_{v_0}) \subset Z(B_{v_0}^{\times}) \cap G_{v_0}$ is always satisfied.

We now make an extra assumption on $K_0^\infty$. It is rather minor: we believe it can always be explicitly realized in practice (\emph{cf.} Section \S \ref{sec:thelattice} for instance).

\begin{assm}\label{a:orderforgates}
There is an $\mc O_E$-order $\Lambda_B$ of $B$ such that:
\[
K_0^\infty = G^\infty \cap (\Lambda_B \otimes_{\mc O_E} \wh{\mc O}_E)
\]
and $(\Lambda_B)_{v_0} := \Lambda_B \otimes_{\mc O_E} \mc O_{E_{v_0}}$ is a maximal $\iota$-stable order of $B_{v_0}$.

\end{assm}

In particular
\[
\Lambda_{0, v_0}  = G(F) \cap \prod_{v \nmid v_0} (\Lambda_B)_v= B \cap  K_0^{\infty, v_0}
\]
and  $K_{B,v_0} := (\Lambda_B)_{v_0} \cap B_{v_0}^\times$ is a $\theta$-stable hyperspecial subgroup of $B_{v_0}^{\times}$ containing $K_{0, v_0}$. Moreover, the vertex $x_0$  related to $K_{0,v_0}$ in $\mc B(G_{v_0})$ corresponds to the vertex $x_0'$ related to $K_{B,v_0}$ in $\mc B(B_{v_0}^{\times})$ via \eqref{eq Bembed}.

\subsubsection{Fixing Coordinates} 
Let $A$ be a maximal split torus of $G_{v_0}$, such that $x_0$ lies in the apartment $\mc A(G_{v_0},A)$ of $G_{v_0}$ defined by $A$. Let $A_B$ be a $\theta$-stable maximal split torus of $B_{v_0}^{\times}$ that contains $A$, whose existence is guaranteed by \cite[Proposition 2.3]{HW93}. In particular, $A$ equals the connected component of $A_{B}^{\theta}$. 

Then the apartment $\mc A(B_{v_0}^{\times},A_B)$ of $B_{v_0}^{\times}$ defined by $A$ contains the vertex $x_0'$, since the embedding of apartment $\mc A(G_{v_0},A)\hookrightarrow\mc A(B_{v_0}^{\times},A_B)$ induced by \eqref{eq Bembed} maps $x_0$ to $x_0'$. As a result,
 $A \subseteq A_B$ are maximally split tori in good position to $K_{0,v_0} \subseteq K_{B, v_0}$. Then there is an induced embedding $X_*(A) \into X_*(A_B)$ 
 corresponding to a containment of Cartan double cosets: for $\lambda\in X_*(A)$ and $\pi_{v_0}$ a uniformizer in $E_{v_0}$:
\[
K_{0, v_0} \lb(\pi_{v_0}) K_{0, v_0} \subseteq K_{B, v_0} \lb(\pi_{v_0}) K_{B,v_0}. 
\] 

Up to a change of basis, we may, without loss of generality, choose the isomorphism $B_{v_0} \cong \Mat_{N \times N}(E_{v_0})$ so that:
\begin{itemize}
    \item $(\Lambda_B)_{v_0} = \Mat_{N \times N}(\mc O_{E_{v_0}})$ and $K_{B,v_0} = \GL_N(\mc O_{E_{v_0}})$;
    \item $A_B = \{\mathrm{diag}(x_1, \dots, x_N) : x_i \in E_{v_0}^{\times}\}$;
    \item The involution $\theta$ can be identified with the involution on $\GL_N(E_{v_0})$ given by $g \mapsto h^{-1} \bar{g}^{-T} h$ for a Hermitian matrix $h \in \GL_N(E_{v_0})$;
    \item When $v_0 = w_1 w_2$ is split, we have
    \[
    h = (h_1, \bar{h}_1^T) = \iota((h_1, I_N))(h_1, I_N)
    \]
    for some $h_1 \in \GL_N(E_{w_0})$. Then, conjugating by $(h_1, I_N)$, we may assume without loss of generality that $h$ is the identity;
    \item When $v_0$ is not split in $E$, since $\theta$ fixes both $K_{B,v_0}$ and $A_B$, up to conjugation by a permutation matrix, we may write
    \[
    h = w_0 \mathrm{diag}(x_1, x_2, \dots, x_r, h_0, \bar{x}_r, \dots, \bar{x}_2, \bar{x}_1),
    \]
    where $x_1, \dots, x_r \in \mc O_{E_{v_0}}^{\times}$, $w_0 \in \mf S_N$ is the longest permutation matrix, $r$ denotes the split rank of $G_{v_0}$, and $h_0 \in \GL_{N-2r}(\mc O_{E_{v_0}})$.
\end{itemize}

Finally, define
\[
X_N := \{(a_1, \dotsc, a_N) : a_1 \geq \cdots \geq a_N, \, a_i \in \Z\}.
\]

As in Assumption \ref{a:main}, we restrict to the cases where $G_{v_0}$ is quasisplit and unramified; the other cases can be treated similarly using the more intricate theory in \cite[\S3.3.2]{DEP}.

\subsubsection{Split Case}
When $v_0 = w_1w_2$ further splits in $E$, we may identify $X_*(A_B)$ with $\Z^N\times\Z^N$ and $X_+(A_B)$ with $X_N\times X_N$ by taking the valuation of each diagonal coordinate in $E_{v_0}$. Then $X_*(A)$ is identified with $\Z^N$ and $X_+(A)$ is identified with $X_N$. More precisely, we have
\[
G_{v_0} \cong \GL_{N,w_1} \xhookrightarrow{a \mapsto (a, \bar a^{-T})} \GL_{N,w_1} \times \GL_{N,w_2} \cong B^\times_{v_0} 
\]
which induces
\[
X_+(A) \cong X_N \xhookrightarrow{\lb \mapsto (\lb, -\lb)} X_N \times (-X_N). 
\]
Next, the condition that $sx$ is at distance $1$ from $x$ for $s \in \bar{\Lambda}_{0,v_0}$ is equivalent to $s$ being contained in a $K_{0,v_0}$-double coset corresponding to some $(a_i)_i \in X_N$ with $a_1 - a_N = 1$ (see, for example, \cite[Example 3.3.4]{DEP}).
Up to multiplication by a central element, we may, without loss of generality, choose a representative for $s$ with $a_N = 0$ and $a_1 = 1$.

We can also compute that
\[
(\Lambda_B)_{v_0} \cap B^\times_{v_0}
= \Mat_{N \times N}(\mc O_{E_{v_0}}) \cap \GL_N(E_{v_0})
\]
is the union of $K_{B,v_0}$-double cosets corresponding to $(a_i)_i \times (a_i')_i \in X_N \times X_N$ with $a_i, a_i' \geq 0$.

Note that any uniformizer $\pi_{w_2}$ of $E_{w_2}$, viewed as an element of $A_B$, corresponds to
\[
\lambda = ((0, \dotsc, 0), (1, \dotsc, 1)) \in \Z^N \times \Z^N.
\]
Therefore, our condition on $s$ is equivalent to requiring that 
\[
s \in \pi_{w_2}^{-1} (\Lambda_B)_{v_0} \setminus ((\Lambda_B)_{v_0} \cup \pi_{w_2}^{-1} \bar\pi_{w_2}(\Lambda_B)_{v_0}\]

In total, if we can find a “globalization” $\pi \in \mc O_E$ that is a uniformizer at $w_2$ and integral at all other places (for example, if $\mc O_E$ is a PID), then 
\[
S_{\bar \Lambda_{0,v_0}} = \{\pi^{-1} g \in \bar \Lambda_{0, v_0} : \iota(g)g = \pi \bar \pi, \, g \in \Lambda_B, \, g \notin \pi \Lambda_B \cup\bar \pi\Lambda_B\},
\]
noting that $\pi \bar{\pi}$ is a uniformizer for $F_{v_0}$.

\subsubsection{Inert Case}
In the case when $v_0$ is moreover inert in $E$ and $G_{v_0}$ is quasisplit, we may identify $X_{*}(A_B)$ with $\Z^N$ and $X_+(A_B)$ with $X_N$ by taking valuations. Then $X_{*}(A)$ is identified with $\Z^r$ and $X_+(A)$ with $X_r$. More precisely, we have
\[
G_{v_0}
\cong U^{E_{v_0}/F_{v_0}}_N(h)
\hookrightarrow \GL_N(E_{v_0})
\cong B^\times_{v_0},
\]
inducing the embedding
\[
X_*(A)
\cong \Z^r
\xhookrightarrow{\lambda \mapsto (\lambda, 0, \dotsc, 0, -\lambda)}
\Z^N
\cong X_*(A_B).
\]

Here, the condition that $sx$ is at distance $2$ from $x$ is equivalent to $s$ being contained in a double coset corresponding to some $(a_i)_i \in X_r \subset \Z^r$ with $a_1 = 1$ and $a_r = 0$ (see, for example, \cite[Example 3.3.4]{DEP}).

Similarly, we have
\[
(\Lambda_B)_{v_0} \cap B^\times_{v_0}
\]
is the union of double cosets for $(a_i)_i \in X_N$ with all $a_i \geq 0$.

Therefore, as in the split case, whenever we can find an appropriate “globalization” $\pi \in \mc O_E$ that is a uniformizer at the place of $E$ dividing $v_0$ and integral at all other places, our condition on $s$ becomes $\pi s \in (\Lambda_B)_{v_0}$ and $s \notin (\Lambda_B)_{v_0}$. This yields
\[
S_{\bar \Lambda_{0,v_0}}=\{\pi^{-1} g \in \bar \Lambda_{0,v_0}: \iota(g)g = \pi^2, \, g \in \Lambda_B, \, g \notin \pi \Lambda_B \}.
\]

\vspace{1cm}

\part{An Explicit Example}

We now apply the general theory above to a specific example of a class-number-one group $G$ and $K^{\infty, v_0} \subseteq G$ satisfying Assumption \ref{a:main}. This example was originally introduced in \cite[Theorem A(2)]{MSG12}. 

Let $E = \Q(\sqrt{-7})$, and let $\mf p_0$ be the prime ideal of $\mc O_E$ lying over $2$, generated by $(1 + \sqrt{-7})/2$. Consider the degree-$5$  division algebra $D$ over $E$ with:
\begin{itemize}
    \item Hasse invariants $\inv_{\mf p_0} D = -\inv_{\bar{\mf p}_0} D \neq 0$, 
    \item Hasse invariants $\inv_\mf p D = 0$ for all $\mf p \neq \mf p_0, \bar{\mf p}_0$.
\end{itemize}
Let $\iota$ be an involution of the second kind on $G$ (see \S\ref{sub:setup}) and set $G$ to be the subgroup of $\iota\circ(\cdot)^{-1}$-fixed points in $D^\times$. Let $G^{\der}$ be the derived subgroup and $G_{\ad}$ the adjoint subgroup of $G$.

For each $p$, let $K_{0,p}^{\der} \subseteq G_p^{\der}$ be parahoric subgroups such that 
\begin{itemize}\label{parahoricconditions}
    \item They are coherent; that is, there exists a group scheme $\ms G_0^{\der}$ over $\Z$, such that for almost all primes $p$, $\ms G_0^{\der}(\Z_p)=K^{\der}_{0,p}$.
    \item $K^{\der}_{0,2}$ is the unique parahoric subgroup of $G_2^{\der}\cong \mathrm{SL}_1(D_2)$.
    \item $K^{\der}_{0,7}$ belongs to one of the two conjugacy classes of special parahoric subgroups (see \cite[\S11.2]{MSG12}). Beware that these two classes behave differently---one is extra-special and the other isn't (though they have the same volume). 
    \item $K^{\der}_{0,p}$ for $p \neq 2,7$ is in the unique conjugacy class of hyperspecial subgroups.
\end{itemize}
Fix a split prime $p$ with respect to $E/\Q$. Let 
\[
\Lambda_{0,p}^{\der}=G^{\der}(\Q)\cap \prod_{p'\neq p}K_{0,p'}^{\der} .
\]
We realize $\Lambda_{0,p}^{\der}$ as a lattice in $G_p^{\der}$ and let $N_{G_{\ad,p}}(\Lambda_{0,p}^{\der})$ be the normalizer of $\Lambda_{0,p}^{\der}$ in $G_{\ad,p}$. Then, \cite[Theorem A(2)]{MSG12} implies that $N_{G_{\ad,p}}(\Lambda_{0,p}^{\der})$ acts simply transitively on the set of vertices $\mc B^0(G_p)$ of $\mc B(G_p)$. 

Over the next sections we find an explicit matrix representation of this example, and then apply our construction to it. 
\begin{itemize}
    \item In \S\ref{sec:thegroup}, we construct the division algebra $D$ as matrices over a number field and realize $G \subseteq D$ rationally.
    \item In \S\ref{sec:integralstructure}, we find a convenient maximal order $\Lambda_D^{\max}$ inside $D$.
    \item In \S\ref{sec:thelattice}, we check that $\Lambda_D^{\max}$ defines subgroups $K_{0,p}$ and lattices $\Lambda_{0,p}$ consistent with \cite{MSG12}. This requires a technical step to relate our context of working with $G$ to the context in \cite{MSG12} working with $G^\der$. 
    \item In \S\ref{sec:gatesets}, we compute local isomorphisms between $G_p$ and standard models of $\GL_5(\Q_p)$ and $U_5(\Q_p)$ so that we can compute the action of $G_p$ on a standard description of its Bruhat--Tits building. This allows us to implement constructions \ref{sssec:splitconstruction} and \ref{sssec:nonsplitconstruction}. 

    \item In \S \ref{sec:findinggates}, we provide an effective way of explicitly constructing all the elements in the gate set $S_{\bar \Lambda_{0,p}}$. 

    \item In \S \ref{sec:algo}, we sum up our full algorithm.
    
\end{itemize}

\section{A Group \lm{$G$} with Class Number One}\label{sec:thegroup}
In \S\ref{sec:Dcons}, we construct our desired division algebra $D$ as explicit matrices over a number field. In \S\ref{sec:Gcons}, we construct an involution of the second kind on $D$ that allows us produce our desired group $G$ rationally.

\subsection{Constructing \lm{$D$}}\label{sec:Dcons}

\subsubsection{Basic setup}\label{sssec:prime}
The paper \cite{Hul35}, though it only applies to division algebras over $\Q$, suggests ideas for finding particularly nice realizations of $D$ as a cyclic algebra. 

Let $E=\Q(\sqrt{-7})$ as before. We write \[\rho_2=(1+\sqrt{-7})/2\quad \text{and}\quad \bar \rho_2=(1-\sqrt{-7})/2.\] Then $2=\rho_2\bar \rho_2$ and $\mc O_E=\Z[\rho_2]$. Notice that a prime $p$ is split over $E/\Q$ if and only if $p\equiv 1,2,4\pmod {7}$, and inert over $E/\Q$ if and only if $p\equiv 3,5,6\pmod {7}$. We fix a split odd prime $p_0$. Write \[p_0=\rho_{p_0}\bar \rho_{p_0},\quad \text{where}\quad \rho_{p_0}=a_0+b_0\rho_2,\quad \ \bar \rho_{p_0}=a_0+b_0\bar \rho_2,\quad a_0,b_0\in\N.\] 
We consider the totally real cyclic extension $M$ of $\Q$ of degree 5 and discriminant $p_0^4$, whose only ramified place is $p_0$. Theoretically $M$ can be realized as a subfield of the totally real field $\Q(\zeta_{p_0}+\zeta_{p_0}^{-1})$ over $\Q$, where $\zeta_{p_0}$ is a primitive $p_0$-th root of unity. As a result, $p_0$ is in the norm group $N_{M/\Q}(M^{\times})$. We fix $\delta\in M$ such that  $N_{M/\Q}(\delta)=p_0$. We fix a $\Z$-basis $\alpha_i$, $i=0,1,2,3,4$ such that $\mc O_{M}=\bigoplus_{i=0}^4\Z \alpha_i$. In particular, when $\mc O_M$ is monogenic over $\Z$, we may find some $\alpha_1\in \mc O_M$ and choose $\alpha_i=\alpha_1^i$, $i=0,1,2,3,4$.

Let $L$ be the compositum of $M$ and $E$, which is a CM field of degree $10$ over $\Q$. Since $M/\Q$ and $E/\Q$ are both Galois, we get that $L/\Q$ is Galois with $\Gal(L/\Q) = \Z/10$ and that $L/E$ is also cyclic and degree $5$. Let $\sigma$ be a fixed order 5 automorphism in $\Gal(L/E)\cong\Gal(M/\Q)$ and let $x\mapsto \bar{x}$ be the order 2 automorphism in $\Gal(L/M)\cong\Gal(E/\Q)$. We have $\mc O_L=\mc O_E\mc O_M$, since the discriminant of $M$ and $E$ are coprime. Also, $L/E$ is only ramified at $\rho_{p_0}$ and $\bar \rho_{p_0}$ where it is totally tame.

\subsubsection{Constructing $D$}\label{sssec divalg}

Now we give a realization of our division algebra $D$ of degree 5 over $E$ (which is non-split only at $\rho_2$ and $\bar \rho_2$). 

We assume that both $\rho_{p_0}$ and $\bar \rho_{p_0}$ are not in the norm group $N_{L/E}(L^{\times})$. Then, we may find an integer $d\in\{0,1,2,3,4\}$ such that $\rho_2/\bar \rho_2\bar \rho_{p_0}^{2d}$ is a fifth power in $\F_{p_0}\cong \mc O_E/\rho_{p_0} \mc O_E$. For the same $d$, we have $ \rho_2\rho_{p_0}^{2d}/\bar \rho_2$ is a fifth power of $\F_{p_0}\cong \mc O_E/\bar \rho_{p_0} \mc O_E$ as well. Define
\[a_1=\rho_2\rho_{p_0}^d/\bar \rho_2\bar \rho_{p_0}^d\quad \text{and}\quad a_e=a_1^e,\quad e=1,2,3,4.\] 
We claim that $a_e$ is a norm element with respect to $L_v/E_v$ at all places $v$ of $E$ not over $2$, but not a norm element at $v=\rho_2,\bar\rho_2$. Indeed, 
\begin{itemize}

\item $L/E$ is either split or inert at a place $v$ of $E$ not over $2$ or $p_0$, and $a_e$ is of valuation 0 at $v$. Then $a_e$ is a norm element at $v$.
\item Both $p_0$ and $\rho_2\rho_{p_0}^{2d}/\bar\rho_2$ are norm elements at $\bar \rho_{p_0}$, where the latter follows from the fact that its image in $\F_{p_0}\cong \mc O_E/\bar \rho_{p_0} \mc O_E$ is a fifth power. Thus $a_e=p_0^{-de}\times (\rho_2\rho_{p_0}^{2d})^e/\bar\rho_2^e$ is a norm element at $\bar \rho_{p_0}$.
\item In parallel $a_e$ is also a norm element at $ \rho_{p_0}$.
\item Finally, $a_e$ is not a norm element at $\rho_2$ or $\bar \rho_2$, since $L/E$ is unramified at these two places and the valuation of $a_e$ at these two places is not divisible by $5$.
\end{itemize}   
As a result, for a fixed $e$, with $a=a_e$, the division algebra
\[D=D(L,a):=\bigoplus_{i=0}^4u^i L,\quad u^5=a,\quad ulu^{-1}=l^{\sigma}\ \text{for all}\ l\in L\]
is non-split only at $\rho_2$ and $\bar \rho_2$. By choosing $e$ and using the Hasse principle we can make it be in the exact isomorphism class we want. 

We can also realize $D \subseteq \Mat_{5 \times 5}(L)$ through
\begin{equation}\label{eq Dgenerators}
u \mapsto \begin{pmatrix}
 & 1 & & & \\
& & 1 & &\\
& & & 1 &\\
& & & & 1\\
a & & & & 
\end{pmatrix}, \qquad l \mapsto 
 \begin{pmatrix}
l &  & & & \\
& l^\sigma & & &\\
& &  l^{\sigma^2} & &\\
& & & l^{\sigma^3} & \\
& & & & l^{\sigma^4}
\end{pmatrix} \text{ for } l \in L.
\end{equation}
This indeed gives us an embedding of algebraic monoids over $E$
\begin{equation}\label{eq Dembedding}
 \eta : D \into \Res^L_E \Mat_{5 \times 5}.
\end{equation}
For later use, we may and will pick an element $b\in 2\mc O_E$ such that 
\[\bar{b}\equiv b^{-1}\pmod{p_0},\quad b^5\equiv p_0^{de} a \pmod{ \bar \rho_{p_0}\mc O_E},\quad b^5\equiv p_0^{-de} a \pmod{\rho_{p_0}\mc O_E}.\] 

\subsection{Constructing \lm{$G$}}\label{sec:Gcons}
\subsubsection{The Involution}\label{sssec:iota} 

We define the following involution (as an anti-homomorphism) on $D$:
\[
\iota : u \mapsto u^{-1}, \quad l \mapsto \bar l \quad \text{for all } l \in L.
\]
Since $\bar a = a^{-1}$, it is indeed an involution of the second kind. A useful formula for this involution is:
\begin{equation}\label{eq involutionform}
\iota\left(\sum_{i=0}^4 l_i u^i\right) 
= \sum_{i=0}^4 u^{-i} \bar l_i 
= \bar l_0 + \sum_{i=1}^4  \bar a\, u^i \bar l_{5-i}
= \bar l_0 + \sum_{i=1}^4 a^{-1} \, \bar l_{5-i}^{\sigma^i} u^i.
\end{equation}

\subsubsection{The Group}

Consider the algebraic group $G$ over $\Q$ defined by
\[
G(R) := U_{1,D}(R) := \{ d \in (D \otimes_\Q R)^\times : \iota(d)d = 1 \}.
\]
Since $\bar a = a^{-1}$, the involution $\iota$ corresponds to the pullback of the conjugate-transpose involution:
\[
\iota : A \mapsto \bar A^T,
\]
via $\eta$. This realizes $G$ as a subgroup under the embedding:
\begin{equation}\label{eq Gembedding}
\eta_G : G \hookrightarrow \Res^M_\Q U^{L/M, \Id}_5 \hookrightarrow \Res^L_\Q \GL_5,
\end{equation}
where $U^{L/M, \Id}_5$ denotes the (non-quasisplit) unitary group preserving the diagonal Hermitian form on $\Res^L_M \Ga^5$. Since this Hermitian form is positive definite, $\Res^M_\Q U^{L/M, \Id}_5$ is anisotropic over $\mathbb{R}$, and thus so is $G$.

\subsection{A Concrete Choice of Initial Data}\label{subsection initialdata}

In this subsection, we consider the construction of the previous two paragraphs in the particular case when $p_0=11$. In this case, 
\begin{itemize}
\item We have $\rho_{11}=2+\sqrt{-7}$ and $\bar \rho_{11}=2-\sqrt{-7}$.
\item The field $M$ is
defined by the minimal polynomial $x^5 - x^4 - 4x^3 + 3x^2 + 3x - 1$ and indexed by \href{https://www.lmfdb.org/NumberField/5.5.14641.1}{Number field 5.5.14641.1} in LMFDB.

\item Write $\alpha_1$ for a root of the above polynomial and realize it as an element in $M$, we further have $\mc O_M=\Z [\alpha_1]$. We fix a choice of $\sigma$, for instance $\sigma(\alpha_1)=-\alpha_1^4+4\alpha_1^2-2$. We also have $\mc O_M=\Z\langle\alpha_1,\alpha_1^{\sigma},\alpha_1^{\sigma^2},\alpha_1^{\sigma^3},\alpha_1^{\sigma^4}\rangle$.

\item Using Sagemath we find the element $\delta=\alpha_1 + 2$ such that  $N_{M/\Q}(\delta)=11$. 

\item We pick $d=3$. In this case, we have \[\bar \rho_{11}=4-\rho_{11}= 4,\quad \rho_2=5,\quad\bar \rho_2=-4 \quad\text{in}\quad \F_{11}\cong \mc O_E/\rho_{11} \mc O_E \] 
and  \[\rho_2/\bar \rho_2\bar \rho_{11}^{2d}=5/(-4\times 4^{2\times 3})=(-1)^5\quad\text{in}\quad \F_{11}\cong \mc O_E/\rho_{11} \mc O_E.\] Similarly, we have \[\rho_{11}=4-\bar\rho_{11}=4,\quad \rho_2=-4,\quad\bar \rho_2= 5 \quad\text{in}\quad \F_{11}\cong \mc O_E/\bar \rho_{11} \mc O_E\] and  \[\rho_2\rho_{11}^{2d}/\bar \rho_2=(-4\times 4^6)/5=(-1)^5\quad\text{in}\quad \F_{11}\cong \mc O_E/\bar \rho_{11} \mc O_E.\] 

\item We pick $e=1$ and $b=10$.

\end{itemize}  

\begin{rmk}

Another possible choice is $p_0=331$. In this case we may simply pick $d=0$, since it is straightforward to verify that $\rho_{2}/\bar \rho_2$ is already a fifth power in $\mc O_E/\rho_{331} \mc O_E$ and $\mc O_E/\bar \rho_{331} \mc O_E$. This special feature will largely simplify our discussion below, but on the other hand the discriminant as well as the ring of integers of $M$ are much more complicated than those for $p_0=11$.

\end{rmk}

\subsection{Some Technical Lemmas}

We record here two technical lemmas relating to our arithmetic setup that will be useful later:

\begin{lem}\label{lem H1N1trivialglobal}

Define $N^1(L):=\{x\in L^\times :  x\bar x=1\}$. Then the first Galois cohomology $H^1(\Gal(L/E),N^1(L))$ is trivial. In other words, for any $x_0\in E^\times$ such that $x_0\bar x_0=1$, there exists $l_0\in L^\times$ such that $l_0\bar l_0 =1$ and $N_{L/E}(l_0)=x_0$. 

\end{lem}

\begin{proof}

Consider the short exact sequence 
\[
1\rightarrow N^1(L)\rightarrow L^\times \rightarrow N_{L/M}(L^\times )\rightarrow 1.
\]
Taking the related long exact sequence and using that $H^1(\Gal(L/E),L^\times)=\{1\}$, by Hilbert's Theorem 90, we get
\[
H^1(\Gal(L/E),N^1(L))=N_{L/M}(L^\times )^{\Gal(L/E)}/N_{E/\Q}(E^\times).
\]
Therefore, the Hasse norm theorem gives that $H^1(\Gal(L/E),N^1(L))=\{1\}$ if and only if for all prime numbers $p$: 
\begin{equation*}\label{eq H1N1local}
N_{L_p/M_p}(L_p^\times )^{\Gal(M_p/\Q_p)}/N_{E_p/\Q_p}(E_p^\times)
=\{1\}.   
\end{equation*}
We check this for each $p$:

When $p$ is split over $M/\Q$, we have
$L_p\cong E_p^{\oplus 5}$ and thus $$N_{L_p/M_p}(L_p^\times )^{\Gal(M_p/\Q_p)}=N_{E_p/\Q_p}(E_p^\times).$$
When $p$ is split over $E/\Q$, we have 
\[
N_{L_p/M_p}(L_p^\times )=M_p^\times\quad\text{and}\quad N_{E_p/\Q_p}(E_p^\times)=\Q_p^{\times}.
\]
Furthermore, $M_p^{\times\Gal(M_p/\Q_p)}=\Q_p^{\times}.$

When $p$ is not split over either $M/\Q$ or $E/\Q$ (in particular $p\neq 2$) we have that 
\[
N_{L_p/M_p}(L_p^\times )^{\Gal(M_p/\Q_p)}\quad\text{and}\quad
 N_{E_p/\Q_p}(E_p^\times)
 \]
are subgroups of $\Q_p^{\times}$, the latter being of index 2. Thus, we only need to show that
 \[
 N_{L_p/M_p}(L_p^\times )^{\Gal(M_p/\Q_p)}\neq \Q_p^{\times}
 \]
 When $L_p/M_p$ is unramified, a uniformizer of $\Q_p$ is not in the former norm. When $L_p/M_p$ is tamely ramified (i.e. $p=7$), an element in $\Z_p^{\times}$ that is not in $\Z_p^{\times 2}$ is not in the former norm. This completes the proof of the claim and the lemma.

\end{proof}

\begin{lem}\label{lem:Gcenter}
$Z_G(\Q) = N^1(E):=\{x\in E^{\times}: x\bar x=1\}$.
\end{lem}

\begin{proof}
First, $N^1(E) = G(\Q) \cap E^\times \subseteq Z_G(\Q)$. Given an element $x$ in $G$ that is not in $E^{\times}$, consider the field $E[x]$ in $D$ which is of degree 5, then there is another field $L\subset D$ of degree 5 over $E$ that does not commute with $x$. Furthermore, there exists a norm 1 element in $L$ that is not in $E$, thus we may write $L=E[y]$ with $\bar y y=1$. As  a result, $x$ and $y$ do not commute, hence $x$ is not in the center $Z_G(\Q)$.
\end{proof}

\section{Construction of a Maximal \lm{$\iota$}-Stable Order}\label{sec:integralstructure}

In this section, we construct a $\iota$-stable maximal order $\Lambda_D^{\max}$ of $D$. 

\subsection{Constructing a Maximal Order \lm{$\Lambda_D^{\max}$}}\label{ssec:maximalorder}

Under the decomposition $D = \bigoplus_{i=0}^4 u^i L$, we define an order $\Lambda_D$ of $D$ as follows:
$$\Lambda_D=\bigoplus_{i=0}^4u^{i} \bar{\rho}_2^{\lceil \frac{ei}{5}\rceil }\bar\rho_{p_0}^{\lceil \frac{dei}{5}\rceil }\rho_2 ^{\lceil -\frac{ei}{5} \rceil}\rho_{p_0}^{\lceil -\frac{dei}{5} \rceil}\mc O_L.$$
Using the fact that  $u^5=a=(\rho_2 \rho_{p_0}^d)^{e}/(\bar{\rho}_2\bar\rho_{p_0}^d)^{e}$, it is straightforward to check that $\Lambda_D$ is indeed an order of $D$. Moreover, by direct calculation we have 
\[\iota(u^{i} \bar{\rho}_2^{\lceil \frac{ei}{5}\rceil }\bar\rho_{p_0}^{\lceil \frac{dei}{5}\rceil }\rho_2 ^{\lceil -\frac{ei}{5} \rceil}\rho_{p_0}^{\lceil -\frac{dei}{5} \rceil})=u^{5-i}\bar{\rho}_2^{\lceil \frac{e(5-i)}{5}\rceil }\bar\rho_{p_0}^{\lceil \frac{de(5-i)}{5}\rceil }\rho_2 ^{\lceil -\frac{e(5-i)}{5} \rceil}\rho_{p_0}^{\lceil -\frac{de(5-i)}{5} \rceil}.\]
Therefore, $\Lambda_D$ is $\iota$-stable.

Now we construct a maximal order $\Lambda_{D}^{\max}$ of $D$ containing $\Lambda_D$. Define
\[\Lambda_{D,+}=\bigoplus_{i=0}^4u_+^i \bar{\rho}_2^{\lceil \frac{ei}{5}\rceil }\rho_2 ^{\lceil -\frac{ei}{5} \rceil}\rho_{p_0}^{\lceil -\frac{2dei}{5} \rceil} \mc O_L,\quad\text{where}\ u_+=u\delta^{de}\ \text{and}\ u_+^5=p_0^{de}a\]
\[\Lambda_{D,+}^{\max}=\bigoplus_{i=0}^4 y_{+}^i \bar{\rho}_2^{\lceil \frac{ei}{5}\rceil }\rho_2 ^{\lceil -\frac{ei}{5} \rceil}
\rho_{p_0}^{\lceil -\frac{4i}{5} \rceil}\mc O_L,\quad \text{where}\ y_+=(b-u_+)\delta^{4}/\bar \rho_{p_0}\]
\[\Lambda_{D,-}=\bigoplus_{i=0}^4u_-^i \bar{\rho}_2^{\lceil \frac{ei}{5}\rceil }\rho_2 ^{\lceil -\frac{ei}{5} \rceil} \bar\rho_{p_0}^{\lceil \frac{2dei}{5}\rceil}\mc O_L,\quad\text{where}\ u_-=u\delta^{-de}\ \text{and}\ u_-^5=p_0^{-de}a\]
\[\Lambda_{D,-}^{\max}=\bigoplus_{i=0}^4 y_{-}^i \bar{\rho}_2^{\lceil \frac{ei}{5}\rceil }\rho_2 ^{\lceil -\frac{ei}{5} \rceil}\bar\rho_{p_0}^{\lceil \frac{2dei}{5}\rceil}\mc O_L,\quad \text{where}\ y_-=(b-u_-)\delta^{4} /\rho_{p_0}.\]
From our construction, $\Lambda_{D,+}$ and $\Lambda_{D,-}$ are orders of $D$, but on the other hand $\Lambda_{D,+}^{\max}$ and $\Lambda_{D,-}^{\max}$ are a priori only lattices of $D$. 
Finally, define
\[\Lambda_{D}^{\max}:=\Lambda_{D,+}^{\max}+\Lambda_{D,-}^{\max}\subset D.\]

\begin{prop} \label{prop Lambdamax}

The following hold:

\begin{enumerate}

\item For any place $v$ of $E$ not dividing $p_0$, we have $\Lambda_{D_v,+}^{\max}=\Lambda_{D_v,-}^{\max}=\Lambda_{D_v,+}=\Lambda_{D_v,-}=\Lambda_{D_v}$ which is furthermore a maximal order of $D_v$;

\item $\Lambda_{D_{\bar \rho_{p_0}},+}^{\max}$ is a maximal order of $D_{\bar \rho_{p_0}}$, which contains $\Lambda_{D_{\bar \rho_{p_0}}}$,  $\Lambda_{D_{\bar \rho_{p_0}},+}$, $\Lambda_{D_{\bar \rho_{p_0}},-}$ and $\Lambda_{D_{\bar \rho_{p_0}},-}^{\max}$.

\item $\Lambda_{D_{ \rho_{p_0}},-}^{\max}$ is a maximal order of $D_{ \rho_{p_0}}$ which contains $\Lambda_{D_{\rho_{p_0}}}$,  $\Lambda_{D_{\rho_{p_0},+}}$,  $\Lambda_{D_{ \rho_{p_0}},-}$, and $\Lambda_{D_{ \rho_{p_0}},+}^{\max}$.

\item $\Lambda_{D}^{\max}$ is a maximal order of $D$ that contains $\Lambda_D$.

\end{enumerate}

\end{prop}

\begin{proof}

\noindent \underline{Statement (1)}

The equalities follow directly from the definitions of these orders, so we only need to show that $\Lambda_{D_v}$ is a maximal order of $D_v$. If $v$ does not divide $2$ or $p_0$, then the division algebra $D_v$ is split. In this case, the extension $L_v/E_v$ is either split or unramified. In particular, the norm map $$N_{L_v/E_v}: \mc O_{L_v}^{\times}\rightarrow \mc O_{E_v}^{\times}$$ is surjective.
Since $a$ is integral in $\mc O_E$, we may pick $\gamma_v\in \mc O_{L_v}^{\times}$ such that $N_{L_v/E_v}(\gamma_v)=a$. Write $u=u'\gamma_v$. Since $\rho_2,\bar \rho_2$ are also integral in $\mc O_{E_v}$, we have
\[
\Lambda_{D_v}=\bigoplus_{i=0}^4u^{i} \mc O_{L_v}=\bigoplus_{i=0}^4u'^{i} \mc O_{L_v}.
\]
Noting that $(u')^5=1$, we identify
\[
D_v=\bigoplus_{i=0}^4u'^i L_v
\]
with $\End_{E_v}(L_v) \cong \Mat_{5 \times 5}(E_v)$, where $u'$ is identified with an order-$5$ Galois action and elements of $L_v$ with multiplications. Moreover, $\Lambda_{D_v}$ is exactly the order in $D_v$ fixing the lattice $\mc O_{L_v}$ of $L_v$. In the split or unramified case, $\mc O_{L_v}$ is a maximal lattice in $L_v$, thus $\Lambda_{D_v}$ is a maximal order in $D_v$.  

Now assume $v=\rho_2$ or $\bar{\rho}_2$, which means that $D_v$ is not split. Let $w$ be the unique place of $L$ over $v$. Then $L_w/E_v$ is an unramified extension of degree 5. By definition, 
\[\Lambda_{D_v}=\lf\{\sum_{i=0}^4 u^i l_i \; \md| \; l_i \in L_w, \; 5w(l_i) \geq - i w(a) \ri\}.
\]
Moreover, for $x=\sum_{i=0}^{4}u^i l_i\in D_v$, we have that the valuation of $u^i l_i,i=0,\dots,4$ with respect to $D_v$ are pairwise different. As a result, $\Lambda_{D_v}$ consists of elements in $D_v$ having non-negative valuation, thus it is the maximal order of $D_v$.

\noindent \underline{Statement (2)} 

The first claim follows from the argument in \cite[Proof of Theorem 4, p526-527]{Hul35}. More precisely, we may use the $\mc O_{E_{\bar \rho_{p_0}}}$-basis  $\{y_+^i\alpha_j \,:\, 0\leq i,j\leq 4\}$ to show that $\Lambda_{D_{\bar \rho_{p_0}},+}^{\max}$ is an order and to calculate its discriminant as in \cite[Proof of Theorem 4, p526-527]{Hul35} , which turns out to be the ideal $\mf p_{E_{\bar{\rho}_{p_0}}}^{5(5-1)}$ of  $\mc O_{E_{\bar{\rho}_{p_0}}}$. By \cite[Theorem 14.9]{Rei03},  $\Lambda_{D_{\bar{\rho}_{p_0}},+}^{\max}$ is therefore a maximal order. In particular, the above argument shows that
\[ \Lambda_{D_{\bar \rho_{p_0}},+}=\bigoplus_{i=0}^4u_{+}^i\mc O_{L_{\bar \rho_{p_0}}}\subset \bigoplus_{i=0}^4y_{+}^i\mc O_{L_{\bar \rho_{p_0}}}=\Lambda_{D_{\bar{\rho}_{p_0}},+}^{\max} .\]
Also, by direct calculation we have
\[\Lambda_{D_{\bar \rho_{p_0}},-}=\bigoplus_{i=0}^4u_{+}^i\delta^{-2dei}\bar \rho_{p_0}^{\lceil\frac{2dei}{5}\rceil}\mc O_{L_{\bar \rho_{p_0}}}\subset \bigoplus_{i=0}^4u_{+}^i\delta^{-dei}\bar \rho_{p_0}^{\lceil\frac{dei}{5}\rceil}\mc O_{L_{\bar \rho_{p_0}}}=\Lambda_{D_{\bar \rho_{p_0}}}\subset \Lambda_{D_{\bar \rho_{p_0}},+},\]
where we used the fact that the valuation of $\delta$ at $\bar \rho_{p_0}$ is $1/5$, and $0\leq \lceil\frac{dei}{5}\rceil-\frac{dei}{5}\leq \lceil\frac{2dei}{5}\rceil-\frac{2dei}{5}$. Finally, we have
that 
\[y_{-}\delta^{-2de}=(b-u_{-})\delta^{-2de}/\rho_{p_0}=(b\delta^{-2de}-u_{+})/\rho_{p_0}\]
and as a result for $i=0,1,2,3,4$ we have
\[y_{-}^i \bar{\rho}_2^{\lceil \frac{ei}{5}\rceil }\rho_2 ^{\lceil -\frac{ei}{5} \rceil}\bar\rho_{p_0}^{\lceil \frac{2dei}{5}\rceil}\mc O_{L_{\bar \rho_{p_0}}}=(b\delta^{-2de}-u_{+})^{i}\bar\rho_{p_0}^{\lceil \frac{2dei}{5}\rceil}\delta^{2dei}\mc O_{L_{\bar \rho_{p_0}}}\subset \bigoplus_{i=0}^4u_{+}^i\mc O_{L_{\bar \rho_{p_0}}},\]
since $b\delta^{-2de}$ has valuation $2de/5\geq 0$ and $\bar \rho_{p_0}^{\lceil \frac{2dei}{5}\rceil}\delta^{2dei}$ has valuation $\lceil\frac{2dei}{5}\rceil-\frac{2dei}{5}\geq 0$ at $\bar \rho_{p_0}$. Thus we have shown that
$\Lambda_{D_{\bar{\rho}_{p_0}},-}^{\max}\subset \Lambda_{D_{\bar \rho_{p_0}},+}$. 

\noindent \underline{Statement (3)}

This argument is analogous to  (2). The only change is to show that $\Lambda_{D_{ \rho_{p_0}},+}^{\max}$ is contained in $\Lambda_{D_{ \rho_{p_0}},-}$. This is because 
\[y_{+}=(b-u_+)\delta^{4}/ \bar \rho_{p_0}=(b-u_{-}\delta^{2de})\delta^{4}/ \bar \rho_{p_0}\]
and as a result for $i=0,1,2,3,4$ we have
\[y_{+}^i \bar{\rho}_2^{\lceil \frac{ei}{5}\rceil }\rho_2 ^{\lceil -\frac{ei}{5} \rceil}
\rho_{p_0}^{\lceil -\frac{4i}{5} \rceil}=(b-u_{-}\delta^{2de})^{i}\delta^{4i}\rho_{p_0}^{\lceil -\frac{4i}{5} \rceil}\mc O_{L_{ \rho_{p_0}}}\subset \bigoplus_{i=0}^4u_{-}^i\mc O_{L_{ \rho_{p_0}}},\]
since $b,\delta^{2de}$ are integral and $\delta^{4i}\rho_{p_0}^{\lceil -\frac{4i}{5} \rceil}$ has valuation $\frac{4i}{5}+\lceil-\frac{4i}{5}\rceil\geq 0$ at $\rho_{p_0}$.

\noindent \underline{Statement (4)} 

This follows from the fact that $\Lambda_{D_v}^{\max}$ is a maximal order of $D_v$ and $\Lambda_{D_v}^{\max} \supset \Lambda_{D_v}$ for any place $v$ of $E$.\footnote{We tacitly used the following fact: Let $\Gamma$ be a lattice of a division algebra $D$ over a number field $E$, then $\Gamma=D\cap\bigcap_{v\ \text{finite}}\Gamma\otimes_{\mc O_E}\mc O_{E_v}$. Indeed, we may find a free $\mc O_{E}$-basis $e_1,\dots,e_{d^2}$ of $\Gamma$, such that we have $\Gamma=\bigoplus_{i=1}^{d^2}\mc O_{E}e_i$ and $\Gamma_v=\bigoplus_{i=1}^{d^2}\mc O_{E_v}e_i$. In general, any element $x=\sum_{i=1}^{d^2}x_ie_i,\ x_i\in E$ of $D$ lies in $\Gamma$ if and only if $x_i$ lies in $\mc O_{E_v}$ for each $i$ and $v$, since $\mc O_{E}=E\cap \bigcap_{v\ \text{finite}}\mc O_{E_v}$.}
 
\end{proof}

\subsection{\lm{$\iota$}-Stability}

In this part, we check that $\Lambda_D^{\max}$ is $\iota$-stable. 

\begin{lem}

\begin{enumerate}

\item For a prime number $p$ different from  $p_0$, we have
$\iota(\Lambda_{D_p})=\Lambda_{D_p}$.

\item We have \[\iota(\Lambda_{D_{\bar \rho_{p_0}}}^{\max})=\iota(\Lambda_{D_{\bar \rho_{p_0}},+}^{\max})=\Lambda_{D_{\rho_{p_0}},-}^{\max}=\Lambda_{D_{\rho_{p_0}}}^{\max}\] and \[\iota(\Lambda_{D_{\rho_{p_0}}}^{\max})=\iota(\Lambda_{D_{\rho_{p_0}},-}^{\max})=\Lambda_{D_{\bar \rho_{p_0}},+}^{\max}=\Lambda_{D_{\bar \rho_{p_0}}}^{\max}.\] Here we use the identification \[D_{p_0}\cong\bigoplus_{i=0}^4 u^i L_{p_0}\cong \bigoplus_{i=0}^4 u^i L_{\rho_{p_0}}\times \bigoplus_{i=0}^4 u^i L_{\bar \rho_{p_0}}\cong D_{\rho_{p_0}}\times D_{\bar \rho_{p_0}},\] and \[\iota((\sum_{i=0}^4u^il_{i},\sum_{i=0}^4u^il_{i}'))=(\sum_{i=0}^4\bar l_{i}'u^{-i},\sum_{i=0}^4\bar l_{i}u^{-i}),\ l_{i}\in L_{\rho_{p_0}},\ l_{i}'\in L_{\bar \rho_{p_0}}.\]
As a result, $\iota(\Lambda_{D_{p_0}}^{\max})=\Lambda_{D_{p_0}}^{\max}$. 
\end{enumerate}

\end{lem}

\begin{proof}

The proof of (1) is straightforward. For $p\neq 2$, we have that
$$\Lambda_{D_p}=\bigoplus_{i=0}^4u^{i} \mc O_{L_p},$$
where we write $\mc O_{L_p}:= \mc O_{L}\otimes_{\Z}\Z_p$. It is $\iota$-stable since in this case $\mc O_{L_p}$ is $\Gal(L/K)$-stable, $\iota(u^i)=u^{5-i}a^{-1}$ and $a\in \mc O_{L_p}^{\times}$.
For $p=2$, it follows from the fact that both $\iota(\Lambda_{D_2})$ and $\Lambda_{D_2}$ are  maximal orders of $D_{2}=D_{\rho_2}\times D_{\bar \rho_2}$, and thus they are the same.

For (2), by symmetry and Proposition \ref{prop Lambdamax}.(2)(3), we only need to show that $\iota(\Lambda_{D_{\bar \rho_{p_0}},+}^{\max})\subset\Lambda_{D_{\rho_{p_0}},-}^{\max}$. Here, $\Lambda_{D_{\rho_{p_0}},-}^{\max}$ is a maximal order of $D_{\rho_{p_0}}$ and $\Lambda_{D_{\bar \rho_{p_0}},+}^{\max}$ is a maximal order of $D_{\bar \rho_{p_0}}$. Since 
\[\Lambda_{D_{\bar{\rho}_{p_0}},+}^{\max}=\bigoplus_{i=0}^4y_{+}^i\mc O_{L_{\bar \rho_{p_0}}},\]
we only need to show that \[\iota(y_{+})\in \Lambda_{D_{\rho_{p_0}},-}^{\max}=\bigoplus_{i=0}^4y_{-}^i\mc O_{L_{ \rho_{p_0}}}.\]
By definition, we have
\begin{equation*}
\begin{aligned}
\iota(y_+)&=\iota((b-u_{+})\delta^4/\bar \rho_{p_0})=   \iota(b\delta^4-u\delta^{4+de})/\rho_{p_0}=(\delta^4\bar b-\delta^{4+de}u^{-1})/\rho_{p_0}\\
&= (\delta^4\bar b-\delta^{4+de}u^{4}a^{-1})/\rho_{p_0}=(\delta^4\bar b-u_-^{4}\delta^{4de}(\delta^{\sigma})^{4+de}a^{-1})/\rho_{p_0}\\
&=[\delta^4\bar b-(y_-\rho_{p_0}\delta^{-4}-b)^{4}\delta^{4de}(\delta^{\sigma})^{4+de}a^{-1}]/\rho_{p_0}\\
&=\delta^4\bar b/\rho_{p_0}-(y_-\rho_{p_0}\delta^{-4}-b)^{4}\delta^{4de}(\delta^{\sigma})^{de}a^{-1}(\delta^{\sigma})^4/\rho_{p_0}
\end{aligned}    
\end{equation*}

\begin{lem}

$(y_-\rho_{p_0}\delta^{-4}-b)^{4}=b^4+\sum_{i=1}^4y_-^ia_i$, where $a_i$ lies in $\mf p_{L_{\rho_{p_0}}}$. 

\end{lem}

\begin{proof}

By direct calculation, for $l\in L$ we have 
\[ly_-\rho_{p_0}\delta^{-4}=l(b-u\delta^{-de})=-ul^{\sigma^{-1}}\delta^{-de}+bl=-y_{-}\rho_{p_0}\delta^{-4}l^{\sigma^{-1}}+b(l-l^{\sigma^{-1}}).\]
Using this equation several times to move the $y_-$ occurring in each monomial of $(y_-\rho_{p_0}\delta^{-4}-b)^{4}$ to the leftmost and the fact that $b\in\mc O_{L_{\rho_{p_0}}}\cap E$ and $\rho_{p_0}\delta^{-4}\in \mf p_{L_{\rho_{p_0}}}$ the result follows. 

\end{proof}

As a result, we only need to show that
\[\delta^4\bar b/\rho_{p_0}-b^4\delta^{4de}(\delta^{\sigma})^{de}a^{-1}(\delta^{\sigma})^4/\rho_{p_0}\in \mc O_{L_{\rho_{p_0}}},\]
or equivalently,
\[\bar b-b^4\delta^{4de}(\delta^{\sigma})^{de}a^{-1}(\delta^{\sigma})^4\delta^{-4}\in \mf p_{L_{\rho_{p_0}}}.\]
This is simply because its image in $\F_{p_0}\cong \mc O_{L_{\rho_{p_0}}}/ \mf p_{L_{\rho_{p_0}}}$ is
\[b^{-1}-b^4\delta^{4de}\delta^{de}a^{-1}\delta^4\delta^{-4}=b^{-1}(1-b^5p_0^{de}a^{-1})=0,\]
since $\delta^{5}=p_0$, $\bar b = b^{-1}$ and the $\sigma$-action is trivial on the residue field $\F_{p_0}$.

\end{proof}

As a result, we have 
$\iota(\Lambda_D^{\max})=\Lambda_{D}^{\max}$, since the same equations hold after taking the completion at every prime $p$. We note it as the following formal proposition.

\begin{prop}\label{prop iotastable} $\Lambda_{D}^{\max}$ is $\iota$-stable.

\end{prop}

\subsection{Description of \lm{$\Lambda_D^{\max}$} and \lm{$(\Lambda_D^{\max})^{\iota}$}} \label{subsection despLambdaDmaxiota}

The main goal here is to describe $\Lambda_{D}^{\max}$ and its $\iota$-invariant part $(\Lambda_{D}^{\max})^{\iota}$ as free $\Z$-modules, and then to propose a computational way to find a basis.

We first study the $\iota$-invariant part $D^{\iota}$ of $D$. By direct calculation, for $\gamma=\sum_{i=0}^4u^il_i\in D^{\iota}$ with $l_i\in L$, we have
\[\bar l_0+\sum_{i=1}^4\bar l_ia^{-1}u^{5-i}=\iota(\gamma)=\gamma=l_0+\sum_{i=1}^4l_i^{\sigma^i}u^i.\]
Comparing the coefficients, we have
\[\bar l_0=l_0,\ l_1^{\sigma}=a^{-1}\bar l_{4},\ l_2^{\sigma^2}=a^{-1}\bar l_{3},\ l_3^{\sigma^3}=a^{-1}\bar l_{2},\ l_4^{\sigma^4}=a^{-1}\bar l_{1}.\]
Write $l_i=m_i^0+\rho_2m_i^1$ for $m_i^0,m_i^1\in M$, then $m_0^0, m_1^0, m_1^1, m_2^0, m_2^1$ are independent variables, and the rest variables $m_0^1,m_3^0,m_3^1,m_4^0,m_4^1$ are determined by the former ones. More precisely, write $a=a_0+\rho_2 a_1$ with $a_0,a_1\in \Q$, then by direct calculation we have 
\begin{equation}\label{eq mij}
\begin{aligned}
m_0^1&=0,\ m_{4}^{0}=(a_0+a_1)(m_{1}^0)^{\sigma}+(a_0-a_1)(m_1^1)^{\sigma},\\ m_{4}^1&=-a_1 (m_1^0)^{\sigma}-(a_0+a_1)(m_1^1)^{\sigma},\\ m_{3}^{0}&=(a_0+a_1)(m_{2}^0)^{\sigma^2}+(a_0-a_1)(m_2^1)^{\sigma^2},\\ m_{3}^1&=-a_1 (m_2^0)^{\sigma^2}-(a_0+a_1)(m_2^1)^{\sigma^2}.   
\end{aligned}  
\end{equation} 
Now we consider $\Lambda_D^{\max}$ and $(\Lambda_{D}^{\max})^{\iota}$. A priori, they are free $\Z$-modules of rank $50$ and $25$ respectively.
Write
\[
D=\bigoplus_{i=0}^4\bigoplus_{j=0}^1u^i\rho_2^j M
\]
Also, as free $\mc O_{M}$-modules:
\[
\Lambda_{D,+}^{\max}=\bigoplus_{i=0}^4\bigoplus_{j=0}^1 y_{+}^i a_{+,i}\rho_2^j\mc O_M\quad\text{and}\quad\Lambda_{D,-}^{\max}=\bigoplus_{i=0}^4 \bigoplus_{j=0}^1y_{-}^i a_{-,i}\rho_2^j\mc O_M
\]
with 
\[
a_{+,i}:=\bar{\rho}_2^{\lceil \frac{ei}{5}\rceil }\rho_2 ^{\lceil -\frac{ei}{5} \rceil}\rho_{p_0}^{\lceil -\frac{4i}{5} \rceil}\quad\text{and}\quad a_{-,i}:=\bar{\rho}_2^{\lceil \frac{ei}{5}\rceil }\rho_2 ^{\lceil -\frac{ei}{5} \rceil}\bar\rho_{p_0}^{\lceil \frac{2dei}{5}\rceil},\ i=0,1,2,3,4.
\]
Define $\vec u, \vec y_+,\vec y_-$ as row vectors in $D^{10}$ with their $(2i+j+1)$-th coordinate being 
\[
u^i\rho_2^j,\quad y_{+}^i a_{+,i}\rho_2^j,\quad y_{-}^i a_{-,i}\rho_2^j
\] 
respectively for $i=0,1,2,3,4$ and  $j=0,1$. Define 
\[
T_{\vec u,\vec y}(s,t)=\begin{pmatrix}
1 & s & s^2 & s^3 & s^4 \\
0 & t & \Sigma_{1,1}(s)t & \Sigma_{2,1}(s)t & \Sigma_{3,1}(s)t  \\
0 & 0 & \Pi_2(t) & \Sigma_{1,2}(s)\Pi_2(t) & \Sigma_{2,2}(s)\Pi_2(t)  \\
0 & 0 & 0 & \Pi_3(t) & \Sigma_{1,3}(s)\Pi_3(t)  \\
0 & 0 & 0 & 0 & \Pi_4(t)
  \end{pmatrix},
  \]
where for $s,t\in L$, we use shorthand: 
\[
\Sigma_{1,1}(s)=s+s^{\sigma^{-1}},\quad \Sigma_{2,1}(s)=s^2+ss^{\sigma^{-1}}+(s^{\sigma^{-1}})^2,\] \[\Sigma_{3,1}(s)=s^3+s^2s^{\sigma^{-1}}+s(s^{\sigma^{-1}})^2+(s^{\sigma^{-1}})^3,\quad \Sigma_{1,2}(s)=s+s^{\sigma^{-1}}+s^{\sigma^{-2}},
\] 
\[
\Sigma_{2,2}(s)=s^2+ss^{\sigma^{-1}}+ss^{\sigma^{-2}}+(s^{\sigma^{-1}})^2+s^{\sigma^{-1}}s^{\sigma^{-2}}+(s^{\sigma^{-2}})^2,\] \[\Sigma_{1,3}(s)=s+s^{\sigma^{-1}}+s^{\sigma^{-2}}+s^{\sigma^{-3}},\quad \Pi_{k}(t)=\prod_{i=0}^{k-1}t^{\sigma^{-i}}.
\]  
Recall that $y_\pm=s_\pm +ut_\pm $. Then, by direct calculation, we have\footnote{Here and throughout, the notation $y_{\pm}$, $s_{\pm}$, $t_{\pm}$, etc. indicates that all signs are chosen consistently, giving two identities corresponding to the $+$ case and the $-$ case.}
\[
(
 1 ,
 y_{\pm},
  y_{\pm}^2, 
 y_{\pm}^3,
 y_{\pm}^4) 
=(1,u,u^2,u^3,u^4)T_{\vec u,\vec y}(s_{\pm},t_{\pm}),
\]
where 
\[
s_+=\delta^4b/\bar \rho_{p_0},\quad t_+=-\delta^{de+4}/\bar \rho_{p_0},\quad s_-=\delta^4b/ \rho_{p_0},\quad t_-=-\delta^{4-de}/\rho_{p_0}.
\]
From this, we define matrices $A_+,A_-\in\GL_{5}(L)$ by 
\[
A_{\pm}=T_{\vec u,\vec y}(s_{\pm},t_{\pm})\diag(a_{\pm,0},a_{\pm,1},a_{\pm,2},a_{\pm,3},a_{\pm,4}).
\]
Moreover, for an element $l=m_0+m_1\rho_2\in L$ with $m_0,m_1\in M$, we identify $l$ with the $2\times 2$ matrix 
\[
\begin{pmatrix} m_0 & -2m_1\\ m_1 & m_0+m_1 \end{pmatrix}
\]
under which multiplication by $l$ is realized as the right multiplication of the matrix on row vectors. In this way, we realize $A_+$ and $A_-$ as matrices in $\GL_{10}(M)$. From our construction, we have
\[
\Lambda_{D,\pm}^{\max}=\vec y_{\pm} \cdot \mc O_M^{10}=\vec uA_{\pm} \cdot\mc O_M^{10},
\]
where elements in $\mc O_M^{10}\subset M^{10}$ are regarded as column vectors of 10 coordinates.

For a lattice $\mc L$ of column vectors over $\mc O_M$ (resp. over $\Z$ and of number of rows divisible by 5), we denote by $\mc L_\Z$ (resp. $\mc L_{\mc O_M}$) the corresponding lattice of column vectors over $\Z$ using the identification $\mc O_M\cong\bigoplus_{i=0}^4\Z\alpha_i$ for each $\mc O_M$-coordinate. 

We consider the sum lattice $A_{+} \cdot\mc O_M^{10}+A_{-} \cdot\mc O_M^{10}$ in $M^{10}$, which a priori might not  be a free $\mc O_{M}$-module. Instead,  $(A_{+} \cdot\mc O_M^{10}+A_{-} \cdot\mc O_M^{10})_{\Z}$ is a free $\Z$-module of rank $50$. Then, using the Hermite normal form (HNF) algorithm over $\Z$, we can find an explicit matrix $A_{\max}\in \GL_{50}(\Q)$ such that 
\[
A_{\max}\cdot\Z^{50}=(A_{+} \cdot\mc O_M^{10}+A_{-} \cdot\mc O_M^{10})_{\Z}.
\]
As a result, we get
\[
\Lambda_D^{\max}=\vec u \cdot (A_{\max} \cdot \Z^{50})_{\mc O_M}.
\]
Consider
\begin{multline*}
\text{Col}_{\max} := \\ \{(m_i^j)_{2i+j+1}^T\in  (A_{\max}\cdot \Z^{50})_{\mc O_M}:m_i^j\ \text{satisfies}\  \eqref{eq mij},\ i=0,\dots,4,\ j=0,1\}.
\end{multline*}
Then 
\[
(\Lambda_D^{\max})^{\iota}=\vec{u}\cdot \text{Col}_{\max}.
\]
Consider the $\mc O_M$-lattice $\text{Col}_{\max}^{\text{half}}$ consisting of the $1,3,4,5,6$-th rows of $\text{Col}_{\max}$; it determines $\text{Col}_{\max}$ via the equations \eqref{eq mij}.  Also $(\text{Col}_{\max}^{\text{half}})_{\Z}$ is a free $\Z$-module of rank 25. Thus we may use again an HNF algorithm to find a matrix $B_{\max}\in\GL_{25}(\Q)$ such that 
\[(\text{Col}_{\max}^{\text{half}})_{\Z}=B_{\max}\cdot \Z^{25}.\] 
In other words, we found concrete $\Z$-basis of $\Lambda_D^{\max}$ and $(\Lambda_D^{\max})^{\iota}$ via $A_{\max}$ and $B_{\max}$.

\section{Constructing the Golden Subgroup \lm{$K_0^{\infty}$}}\label{sec:thelattice}

In this part, we  use the maximal order $\Lambda_D^{\max}$ constructed in \S\ref{sec:integralstructure} to define $K^\infty_0$ and thus $\Lambda_{0,p}$ for each $p$. We show that $K^\infty_0$ is golden by checking its consistency with \cite[Theorem A(2)]{MSG12}, as explained in the beginning of Part 2 of this article.

\subsection{Construction of \lm{$K_{0}^{\infty}$}}\label{sec:K0infty}

Define a group scheme $\ms G_0$ over $\Z$ by:
\[
\ms G_0(R):=\{g\in \Lambda_D^{\max}\otimes_{\Z}R :  \iota(g)g=1\}
\]
for any $\Z$-algebra $R$. In particular, it is an integral model of $G/\Q$.
For all $p$, we have
\[
\ms G_0(\Z_p) = G(\Q_p)\cap \Lambda_{D_p}^{\max}.
\]

\begin{prop}\label{prop model}
The following hold:
\begin{enumerate}
\item For $p \neq 2,7$, we have that $K_{0,p}:=\ms G_0(\Z_p)$ is a hyperspecial parahoric subgroup in $G(\Q_p)$. 
\item $\ms G_0(\Z_7)$ has a 
special parahoric subgroup in $G(\Q_7)$ of index-2, which we denote by $K_{0,7}$.
\item $K_{0,2}:=\ms G_0(\Z_2)$ is the unique maximal open compact subgroup in $G(\Q_2)$. 
\end{enumerate}

\end{prop}

\begin{proof}

We have $(\Res_{\Q}^{E}D)_{\Q_p}\cong \prod_{v |  p}\Res_{\Q_p}^{E_v}D_v$, which is endowed with an involution $\theta:=\iota\circ(\cdot)^{-1}$.  Using Proposition \ref{prop Lambdamax} and \ref{prop iotastable}, $(\Lambda_{D_p}^{\max})^\times$ is a $\theta$-stable maximal open compact subgroup of $(\Res_{\Q}^{E}D^\times)(\Q_p)=D_p^{\times}$. Let $x_0$ be the $\theta$-stable vertex in the reduced building $\mc B((\Res_{\Q}^{E}D^\times)_{\Q_p})$ corresponding to  $(\Lambda_{D_p}^{\max})^\times$.

Assume $p \neq 2$ is split. Write $v$, $\bar{v}$ for the two places over $p$, and note that $D_p^{\times} = D_v^\times \times D_{\bar{v}}^\times$. Up to a change of basis, we may identify $D_v^\times$ and $D_{\bar{v}}^\times$ with $\GL_5(\Q_p)$, and
moreover the involution $\theta$ is given by
\begin{equation*}\label{eqDvinvolution}
\begin{aligned}
\theta:D_v^\times\times D_{\bar{v}}^\times&\rightarrow D_v^\times\times D_{\bar{v}}^\times,\\
(g_1,g_2)&\rightarrow(g_2^{-1\dagger},\bar g_1^{-1\dagger}).
\end{aligned}
\end{equation*}
Here, $\dagger$ denotes the conjugate-transpose of matrices. As  a maximal open compact $\theta$-stable subgroup of $D_v^\times\times D_{\bar{v}}^\times$, we may identify $(\Lambda_{D_{p}}^{\max})^\times$ with $K_v\times K_v^{-1\dagger}$, where $K_v$ (resp. $K_v^{-1\dagger}$) is a maximal compact subgroup of $D_v^\times$ (resp. $D_{\bar v}^\times$).
By projecting $D_v^\times \times D_{\bar{v}}^\times$ onto $D_v^\times$, the group $G(\Q_p)=(D_p^{\times})^\theta$ is identified with $D_{v}^\times$ and $K_{0,p}$ is identified with $K_v$. Thus $K_{0,p}$ is hyperspecial in $G(\Q_p)$.

Assume $p=2$. Similarly we have $D_2^{\times}=D_{\rho_2}^\times\times D_{\bar \rho_2}^\times$ with 
\[D_{\rho_2}=\bigoplus_{i=0}^4u^iL_{\rho_2}\quad\text{and}\quad D_{\bar \rho_2}=\bigoplus_{i=0}^4u^iL_{\bar \rho_2}.\]
In particular, we have $D_{\bar \rho_2}\cong D_{ \rho_2}^{\text{opp}}$. We define the anti-homomorphism
\[\dagger:D_{\rho_2}\rightarrow D_{\bar \rho_2},\quad (l\in L_{\rho_2},u)\mapsto (\bar l \in L_{\bar \rho_2},u^{-1}).\]
Then the involution $\theta$ is given by 
\begin{equation*}
\begin{aligned}
\theta:D_{\rho_2}^\times\times D_{\bar \rho_2}^\times&\rightarrow D_{ \rho_2}^\times\times D_{\bar\rho_2}^\times,\\
(g_1,g_2)&\rightarrow(g_2^{-1\dagger},g_1^{-1\dagger}).
\end{aligned}
\end{equation*}
Under this setting, $\Lambda_{D,2}^{\max}$ is identified with the product of orders $\mc O_{D_{\rho_2}}\times \mc O_{D_{\bar \rho_2}}$. Still using the projection $D_{\rho_2}^\times\times D_{\bar \rho_2}^\times\onto D_{\rho_2}^\times$, the group $G(\Q_2)=(D_2^{\times})^\theta$ is identified with $D_{\rho_2}^\times$ and $K_{0,2}$ is identified with $\mc O_{D_{\rho_2}}^{\times}$. Thus $K_{0,2}$ is the only maximal compact subgroup of $G(\Q_2)$. 

Assume $p$ is non-split. Then  $(\Res_{\Q}^{E}D^\times)_{\Q_p}\cong \Res_{\Q_p}^{E_v}D_v^\times\cong \Res_{\Q_p}^{E_v}\GL_5$ and $G(\Q_p) = (D_p^{\times})^\theta$ is a unitary group.  Since $p\neq 2$, we have the identification of sets $\mc B(G(\Q_p) )=\mc B((\Res_{\Q}^{E}D^\times)_{\Q_p})^\theta$ and we may realize $x_0$ as a vertex in $\mc B(G(\Q_p) )$.  The affine root system of $\mc B(G(\Q_p) )$ is given in \cite[\S 4.3, \S 4.4]{Tit79}: it is of type $\mathrm{C}$-$ \mathrm{BC}_2^{\mathrm{IV}}$ if $E_v/\Q_p$ is unramified, and of type $\mathrm{C}$-$ \mathrm{BC}_2^{\mathrm{III}}$ if $E_v/\Q_p$ is ramified (here we are referring to the labels in \cite[p29-p30, Tableau]{BT72} for the types of affine roots). Using \cite[Sections 6 and 8]{AN02}, we see that $x_0$, realized as a vertex in $\mc B(G(\Q_p))$, corresponds to the leftmost node in the affine Dynkin diagram.\footnote{Indeed, the general result in \cite[Remark 15]{AN02} only deduces that $x_0$ corresponds either to the leftmost or the rightmost node (the authors did not study the direction of the arrows). But in our case, one could generalize their calculation to verify that $x_0$ must correspond to the leftmost node.} As a result, $x_0$ is hyperspecial if $E_v/\Q_p$ is unramified ($p\neq 7$), and special if  $E_v/\Q_p$ is ramified ($p=7$). Thus, in the unramified case ($p\neq 7$),  $K_{0,p}:=\ms G_0(\Z_p)$ is the hyperspecial parahoric subgroup of $G(\Q_p) $ related to $x_0$ with  special fiber being a unitary group. In the ramified case ($p=7$), $\ms G_0(\Z_7)$ is not connected with the reductive quotient of the special fiber being an orthogonal group. It has an index-2 connected special parahoric subgroup denoted by $K_{0,7}$.

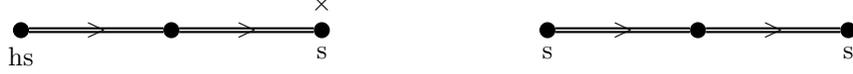
\begin{figure}[h!]
    \centering
\begin{tikzpicture}

    \node[circle, draw, fill=black, inner sep=2pt, label=below:{$\mathrm{hs}$}] (0) at (0,0) {};
    \node[circle, draw, fill=black, inner sep=2pt, label=below:{}] (1) at (2,0) {};
    \node[circle, draw, fill=black, inner sep=2pt, label=below:{$\mathrm{s}$}, label=above:{$\times$}] (2) at (4,0) {};
    
    \node[circle, draw, fill=black, inner sep=2pt, label=below:{$\mathrm{s}$}] (3) at (7,0) {};
    \node[circle, draw, fill=black, inner sep=2pt, label=below:{}] (4) at (9,0) {};
    \node[circle, draw, fill=black, inner sep=2pt, label=below:{$\mathrm{s}$}] (5) at (11,0) {};
    
    \draw[thick, double, -] (0) -- (1) node[midway] {$>$}; 
    \draw[thick, double, -] (1) -- (2) node[midway] {$>$};
    \draw[thick, double, -] (3) -- (4) node[midway] {$>$}; 
    \draw[thick, double, -]  (4) -- (5) node[midway] {$>$}
    ; 
\end{tikzpicture}
\caption{Affine Dynkin Diagram of type $\mathrm{C}$-$ \mathrm{BC}_2^{\mathrm{IV}}$ and $\mathrm{C}$-$ \mathrm{BC}_2^{\mathrm{III}}$}
\label{fig:affine_dynkin_c_bc_2}
\end{figure}
\end{proof}

\begin{rmk}\label{rmk: at7}

Let $\rho_7$ be the place of $E$ over $7$. Then, given $x\in\mc G_0(\Z)$, we have that $\mathrm{Nrd}_{D/E}(x)\equiv \pm 1 \pmod{\rho_7}$, depending on whether $x$ lies in the identity connected component of $\ms G_0(\Z_7)$.
    
\end{rmk}

From now on, we fix the parahoric subgroups $K_{0,p}$ as in the above proposition. We define $K_{0,p}^{\der}=G_p^{\der}\cap K_{0,p}$ for each $p$.

\begin{prop}\label{prop:coherent}

The subgroups $K^{\der}_{0,p}$ of $G_{p}^{\der}$ for all $p$ form a coherent set of parahoric subgroups satisfying the requirements on page \pageref{parahoricconditions}.

\end{prop}

\begin{proof}

For all prime $p$, we need to show that the group $K_{0,p} \cap G^\der(\Q_{p})$ is a related hyperspecial or special parahoric subgroup of $G^\der(\Q_{p})$, which follows from a case-by-case discussion: 
\begin{itemize}

\item When $p=2$, we have $G(\Q_{2})=U_1(D_{\Q_2})$ and $G^\der(\Q_{2})=\SU_1(D_{\Q_2})$ and the claim is clear;

\item When $p$ is split, we have $G(\Q_{p})=\GL_5(\Q_{p})$ and $G^\der(\Q_{p})=\SL_5(\Q_{p})$ and the claim is also clear;

\item When $p$ is inert or ramified, we may without loss of generality assume $G^\der(\Q_{p})=U_5(\Q_{p})$ and $G^\der(\Q_{p})=\SU_5(\Q_{p})$ related to the unitary involution $\theta$ defined by identity matrix. From our choice, the parahoric subgroup $K_{0,p}$ corresponds to the connected integral model $U_5(\Z_{p})$, since it is the one related to the $\theta$-stable hyperspecial parahoric subgroup $\GL_5(\mc O_{E_{p}})$ of $\GL_5(E_{p})$. Its intersection with $\SU_5(\Q_{p})$ being $\SU_5(\Z_{p})$ gives the related special or hyperspecial parahoric subgroup.

\end{itemize} 
Finally, let $\ms G_0^{\der}$ be the group scheme over $\Z$ such that
\[
\ms G_0^{\der}(R):=\{g\in \Lambda_D^{\max}\otimes_{\Z}R :  \iota(g)g=1,\ \det(g)=1\}.
\]
Then for $p$ split or inert, we have
$\ms G_0^{\der}(\Z_p)=K^{\der}_{0,p}$. Hence we get a coherent set.
\end{proof}

\subsection{Simply transitive action of \lm{$\Lambda_{0,p}$}}

For each $p \neq 2$, define
\[
\Lambda_{0,p} := K_{0}^{\infty, p} \cap G(\Q) \into G_p,\quad \text{where}\ K_0^{\infty, p}=\prod_{p'\neq p}K_{0,p'}. 
\]
Then by definition we have $\Lambda_{0,p}^{\der}=\Lambda_{0,p}\cap G^{\der}(\Q)$.
Let $\Lambda_{0,p, \ad}$ be the image of $\Lambda_{0,p}$ in $G_{\ad,p}$\footnote{Beware that this may be different from the previously defined $\bar \Lambda_{0,p}$ since it is a quotient by the entire center instead of just the part that is split at $p$.}. If furthermore $p$ is split with respect to $E/\Q$, we identify $(\ms G_0)_{\Z_p}$ with the group scheme $ \GL_5$ over $\Z_p$.

To translate \cite{MSG12} to our case of forms of $\GL_5$ (instead of $\SL_5$), we need the following lemma and its corollary:
\begin{lem}\label{lem Lambdapsimtran}
There exists a split prime $p$ such that

\begin{enumerate}
\item We have $N_{\PGL_5(\Q_p)} (\Lambda_{0,p}^\der) = \Lambda_{0,p, \ad}$.
\item $\Lambda_{0,p, \ad}$ acts simply transitively on the set of vertices $\mc B^0(G_p)$ of $\mc B(G_p)$.
\item $\Lambda_{0,p}$ acts transitively on $G_p/K_{0,p}$. 
\end{enumerate}
 \end{lem}

\begin{proof}
\underline{Claim (1)}: 

By the table in \cite[Proposition 30]{MSG12}, the image of $\Lambda_{0,p}^\der$ is an index-$5$ subgroup of $N_{\PGL_5(\Q_p)} (\Lambda_{0,p}^\der )$. Since $\Lambda_{0,p}$ normalizes $\Lambda_{0,p} ^\der$, we have $\Lambda_{0,p, \ad} \subset N_{\PGL_5(\Q_p)} (\Lambda_{0,p}^{\der})$. Thus, we only need to show that the images of $\Lambda_{0,p}$ and $\Lambda_{0,p}^\der$  in $\PGL_5(\Q_5)$ do not coincide, or a fortiori $\Lambda_{0,p}$ as a subgroup of $G(\Q_p)=\GL_5(\Q_p)$
is not in $\Q_p^{\times} G^\der(\Q_p)=\Q_p^{\times} \SL_5(\Q_p)$.

We next choose $p$ to satisfy:
\begin{itemize}

\item $p$ is split with respect to the field extension $L/\Q$;

\item Write $p=r^2+7s^2$ for $r,s\in\Z$, then $(r+s\rho)/(r-s\rho)$ is a 5th power in $\Z/p_0$, where $\rho\in\Z/p_0$ such that $\rho^2\equiv -7$ (mod $p_0$).

\end{itemize}
For example, when $p_0=11$ we have
\[
p= 947 = 10^2 + 7\cdot 11^2, \quad \sqrt{-7} \equiv 2 \pmod{11}, \quad \f{10 + 2\cdot11}{10 - 2 \cdot 11} \equiv 1^5 \pmod{11}
\] 

Next, we claim that $x_0=(r+s\sqrt{-7})/(r-s\sqrt{-7})$ is a norm from $L$ to $E$. By the Hasse norm theorem, it suffices to check this locally. For any prime $v$ of $E$, 
\begin{itemize}
\item if $v$ is not over $p_0$ or $p$, then $x_0$ is in $\mc O_{E_v}^\times$ and thus a norm element with respect to $L_v/E_v$;
\item if $v$ is over $p$, then $x_0$ is a norm element with respect to the split extension $L_v/E_v$;
\item if $v$ is over $p_0$, then $x_0\in \mc O_{E_v}^\times$ is a norm element with respect to the totally tamely ramified extension $L_v/E_v$ since it is a 5th power modulo $1+\mf p_{E_v}$.
\end{itemize}
Let $l_0\in L$ such that $N_{L/E}(l_0)=x_0$. By  Lemma \ref{lem H1N1trivialglobal}, we may further assume that $l_0\bar l_0=1$. Then we realize $l_0$ as an element in $G(\Q)$. Since $N_{L/E}(l_0)=x_0\in\mc O_E[1/p]$, by definition we have $l_0\in \Lambda_{D_v}^{\max}$ for $v$ not over $p$. Moreover, from the expression of $x_0$, we have
\[
N_{L/E}(l_0)=x_0=(r+s\sqrt{-7})/(r-s\sqrt{-7})\equiv 1 \pmod{\rho_7}.
\]
In total, $l_0\in G(\Q)\cap K_0^{\infty,p}=\Lambda_{0,p}$.
Since the determinant of $l_0$ in $G_p$ is a uniformizer, we have $l_0\notin\Q_p^{\times} \SL_5(\Q_p)$ which completes the proof of (1).

\noindent \underline{Claim (2):} 

This follows from \cite[Theorem A(2)]{MSG12} and Proposition \ref{prop:coherent}.

\noindent \underline{Claim (3):} 

Let $\mc B^0(G_p)$ be the set of vertices in $\mc B(G_p)$. We consider the map 
\[
G_p/K_{0,p}\mapsto \mc B^0(G_p),\quad g_pK_{0,p}\mapsto g_p\cdot x,
\]
where $x$ is the vertex with related parahoric group being $K_{0,p}$. Since $\Lambda_{0,p}$ acts transitively on $\mc B^0(G_p)$, we have that given any $g_pK_{0,p}$, there exists $\lambda\in\Lambda_{0,p}$ such that $\lambda\cdot x=g_p\cdot x$, or equivalently, $g_p\lambda^{-1}z_p^{-1}\in K_{0,p}$ for some $z_p\in Z(G_p)$. Moreover,  we may pick an element $z\in E^{\times}$ such that \begin{itemize}
\item $z\bar z=1$;
\item both $z$ and $z^{-1}$ are integral outside $p$;
\item $zK_{0,p}=z_p K_{0,p}$.
\end{itemize}
Regarding $z$ as an element in $G(\Q)$, from our construction we indeed have that $z\in\Lambda_{0,p}$. Thus we have $g_p K_{0,p}=\lambda z K_{0,p}$. On the other hand, if $g_1 K_{0,p}=g_2 K_{0,p}$ for $g_1,g_2 \in \Lambda_{0,p}$, then by (2)  $g_1g_2^{-1}$ fixes the vertex corresponding to $K_{0,p}$ and thus is trivial in $\Lambda_{0,p, \ad}$. Therefore $g_1$ and $g_2$ differ by a central element in $G_p$.
\end{proof}

\begin{cor}

The group $K_{0}^{\infty}$ is golden. In particular, for all $p$, $\Lambda_{0,p}$ acts transitively on $G_p/K_{0,p}$ and $\bar \Lambda_{0,p}$ acts simply transitively on the $G_p\cdot x_p$, where $x_p$ denotes the vertex in $\mc B^0(G_p)$ corresponding to $K_{0,p}$.

\end{cor}

\begin{proof}
Given the first statement, the second follows from \cite[Lemma 4.1.2(2)]{DEP}. 

For the first statement, Lemma \ref{lem Lambdapsimtran}(3) and \cite[Corollary 2.2.4]{DEP} gives that $G(\Q) K_0^\infty G_\infty = G(\A)$. Lemma \ref{lem Lambdapsimtran}(2) gives that $G(\Q) \cap K_0^\infty \subseteq Z_G(\Q)$. Therefore it suffices to show that $K_0^\infty \cap Z_G(\Q) = 1$. 

By Lemma \ref{lem:Gcenter}, $Z_G(\Q) = N^1(E)$. We verify that at each local place $v$ of $E$ over a prime number $p$, the element $x\in N^1(E)$ is in $\mc O_{E_v}^{\times}$. When $p$ is split with respect to $E/F$, we have that $x\in E_v^{\times} \cap K_p=\mc O_{E_v}^{\times}$. When $p$ is not split, we have that the valuations of $x$ and $\bar x$ at $v$ are both 0, meaning that $x\in \mc O_{E_v}^{\times}$. As a result, $x\in \mc O_{E}^{\times}=\{\pm1\}$. Finally, since $x$ lies in $K_7$, we must have $x=1$ (\emph{cf.} Remark \ref{rmk: at7}).

\end{proof}

\section{Gate Sets for \lm{$G$}}\label{sec:gatesets}

The last piece we need to run the constructions of Proposition \ref{prop:explicitramanujan} is the computation of the gate sets $S_{\Lambda_{0,p}}$. In this section, we will often use boldface to represent elements of $\Mat_{5 \times 5}(L)$ or elements of $G$ interpreted as elements of $\Mat_{5 \times 5}(L)$ through the map $\eta$ from \eqref{eq Dembedding}. We will also use the shorthand $\m A^\dagger := \bar{\m A}^T$.

\subsection{Explicit Trivializations}\label{sec: Explicit Tri}
We will need explicit trivializations (i.e., explicit isomorphisms to standard matrix algebras) of $\Lambda_D^{\max}$ and $K_0^\infty$ locally and mod $n$. 

\subsubsection{Local Trivializations of $\Lambda_D^{\max}$}\label{sssec:localtriofDmax}

Let $q \neq 2,7$. Assume first $q\neq p_0$ and let us construct explicit isomorphisms $\Lambda_{D,q}^{\max} \to \Mat_{5 \times 5}(\mc O_{E_q})$.  

Note first that
\[
\Lambda_{D,q}^{\max}=\Lambda_{D,q}=\bigoplus_{i=0}^4u^i\mc O_{L_q},\quad u^5=a.
\]
Since $L/E$ is either split or unramified at places of $E$ over $q$ and $a\in\mc O_{L_q}^{\times}$, we may find $\gamma_q\in\mc O_{L_q}^{\times}$ such that $N_{L_q/E_q}(\gamma_q)=a$. Define
\[
\m U_q =\mathrm{diag}(1,\gamma_q,\gamma_q\gamma_q^\sigma,\gamma_q\gamma_q^\sigma\gamma_q^{\sigma^2},\gamma_q\gamma_q^\sigma\gamma_q^{\sigma^2}\gamma_q^{\sigma^3})\in\GL_5(\mc O_{L_q})
\]
Then by definition
\[
\m U_q^{-1}\eta(l)\m U_q=\eta(l)\ \text{for}\ l\in L_q\quad \text{and}\quad \m U_q^{-1}\eta(u)\m U_q=\begin{pmatrix}
 & \gamma_q & & & \\
& & \gamma_q^{\sigma} & &\\
& & & \gamma_q^{\sigma^2} &\\
& & & & \gamma_q^{\sigma^3}\\
\gamma_q^{\sigma^4} & & & & 
\end{pmatrix}.
\]
Next, we take $\alpha\in\mc O_M$ such that the matrix
\[
\m A = \begin{pmatrix}
\alpha & \alpha^{\sigma} & \alpha^{\sigma^2} & \alpha^{\sigma^3} & \alpha^{\sigma^4} \\
\alpha^{\sigma} & \alpha^{\sigma^2} & \alpha^{\sigma^3} & \alpha^{\sigma^4} & \alpha \\
\alpha^{\sigma^2} & \alpha^{\sigma^3} & \alpha^{\sigma^4} & \alpha & \alpha^{\sigma} \\
\alpha^{\sigma^3} & \alpha^{\sigma^4} & \alpha & \alpha^{\sigma} & \alpha^{\sigma^2} \\
\alpha^{\sigma^4} & \alpha & \alpha^{\sigma} & \alpha^{\sigma^2} & \alpha^{\sigma^3} \\
\end{pmatrix}
\]
can be regarded as an element in $\GL_5(\mc O_{L_q})$. In the case where \[\mc O_M=\Z\langle \alpha,\alpha^{\sigma},\alpha^{\sigma^2},\alpha^{\sigma^3},\alpha^{\sigma^4}\rangle\] for some $\alpha\in\mc O_M$, this is so since $\det(\m A)=\Disc(\mc O_M)=p_0^4$.
Then
\[
\m A\m U_q^{-1}\eta(l)\m U_q\m A^{-1}, \m A\m U_q^{-1}\eta(u)\m U_q\m A^{-1}\in\Mat_{5\times 5}(\mc O_{E_q}),\ \text{where}\ l\in\mc O_{L_q}. 
\]
Combining with Proposition \ref{prop Lambdamax}, we must have
\begin{equation}\label{eq:Dlocaltriv}
\m A\m U_q^{-1}\eta(\Lambda_{D,q}^{\max})\m U_q\m A^{-1}=\m A\m U_q^{-1}\eta(\Lambda_{D,q})\m U_q\m A^{-1}=\Mat_{5\times 5}(\mc O_{E_q}).
\end{equation}

Now we drop the condition $q\neq p_0$. We may still find $\gamma_{p_0}\in L_{p_0}^{\times}$ such that $N_{L_{p_0}/E_{p_0}}(\gamma_{p_0})=a$. We define $\m U_{p_0}$ and $\m A$ as above; then 
\[
\m A\m U_{p_0}^{-1}\eta(\Lambda_{D,p_0}^{\max})\m U_{p_0}\m A^{-1}
\] 
is a maximal $\mc O_{E_{p_0}}$-order of $\Mat_{5\times 5}(E_{p_0})$. Using the explicit description of $\Lambda_D^{\max}$ in \S \ref{subsection despLambdaDmaxiota} and the Chinese remainder theorem, for any fixed $n_0$ relatively prime to $2\cdot 7\cdot p_0$, we may find a matrix $\m V_{p_0}\in \GL_5(L)$ such that \[\m V_{p_0}\in \m I_{5}+\Mat_{5\times 5}(n_0 \mc O_{E,(n_0)})\]
and
\[\m V_{p_0} \m A\m U_{p_0}^{-1}\eta(\Lambda_{D,p_0}^{\max})\m U_{p_0}\m A^{-1} \m V_{p_0}^{-1}=\Mat_{5\times 5}(\mc O_{E_{p_0}}).\]

\subsubsection{Mod $n$ trivializations of $\Lambda_D^{\max}$}

Now, let $n_0$ be relatively prime to $2 \cdot 7 \cdot p_0$ and let $n=n_0\cdot p_0^{n_{p_0}}$ for some $n_{p_0}\geq 0$.
Equation \eqref{eq:Dlocaltriv} can be extended to a trivialization over $\mc O_L/n_0$. 
We apply Hensel's lemma and the Chinese remainder theorem to produce $\gamma_n \in \mc O_{L,(n_0)}\times L_{p_0}$ such that 
\[
N_{L/E}(\gamma_n) \equiv a \pmod n.
\]
Define 
\[
\m U_n := \mathrm{diag}(1,\gamma_n,\gamma_n\gamma_n^\sigma,\gamma_n\gamma_n^\sigma\gamma_n^{\sigma^2},\gamma_n\gamma_n^\sigma\gamma_n^{\sigma^2}\gamma_n^{\sigma^3})\in\GL_5(L)\]
and\[\m B_n=\m U_n\m A^{-1}\m V_{p_0}^{-\min(1,n_{p_0})}
\]
Then we have $\m B_n\in\GL_5(\mc O_{L_q})$ for all $q|n_0$.

We consider the reduction mod $n$ map
\[(\Lambda_D^{\max})^\infty := \prod_{v \text{ finite}} \Lambda_{D_v}^{\max} \to (\Lambda_D^{\max})^\infty/n(\Lambda_D^{\max})^\infty\cong \Lambda_{D,(n)}^{\max}/n\Lambda_{D,(n)}^{\max}.\]

\begin{lem}\label{lem:Dtrivmodn}
We have the isomorphism
\[
\m B_n^{-1}\eta((\Lambda_D^{\max})^\infty)/\eta (n(\Lambda_D^{\max})^\infty)\m B_n \cong \Mat_{5 \times 5}(\mc O_E/n).
\]

\end{lem}

\begin{proof}

The definition of $\m B_n$ guarantees that
\[\m B_n^{-1}\eta(\Lambda_{D,(n)}^{\max})\m B_n\cong \prod_{q\mid n}\Mat_{5\times 5}(\mc O_{E_q}).\]
Modding out $n$ we get the result.

\end{proof}

\subsubsection{Trivializations in $G$}

If $\m H$ is a $5 \times 5$ invertible Hermitian matrix over $E$, we define an algebraic group over $\mc \Z$ by
\begin{equation}\label{eq:Umodndef}
U_5^\m H(R) := U_5^{\mc O_E/\Z, \m H} := \{\m X \in \Mat_{5 \times 5}(\mc O_E \otimes_\Z R) \, :\, \bar{\m X}^T \m H \m X = \m H\}
\end{equation} 
for any $\Z$-algebra $R$ and shorthand $U_5(R) := U_5^{\mc O_E/\Z, \m J}$ where $\m J$ is the antidiagonal Hermitian form.

\begin{lem}\label{lem:Kmodn}
Let $n$ be relatively prime to $2 \cdot 7$. Then reduction mod $n$ induces a surjection
\[
\m B_n^{-1} \eta(K_0^\infty) \m B_n \onto U^{\m B_n^\dagger \m B_n}_5(\Z/n) \subseteq \Mat_{5 \times 5}(\mc O_E/n). 
\]
\end{lem}

\begin{proof}

Let $S$ be the set of primes dividing $n$. 
It suffices to consider the reduction mod $n$ of $\m B_n^{-1} \eta(K_{0,S}) \m B_n$. By definition, we have $K_{0,q} = G_q \cap \Lambda_{D,q}^{\max}$ for $q|n$. Since $\eta(G_q) = \{\m X \in \eta(D_q) \,:\, \m X^\dagger \m X = 1\}$, the result follows from Lemma \ref{lem:Dtrivmodn}. 

\end{proof}

We remark that for $q\mid n_0$, the matrix $\m B_n^{\dagger}\m B_n$ already lies in $\GL_5(\mc O_q)$. When $q=p_0$, since the maximal compact subgroup $\eta(\Lambda_{D_q}^{\max})$ is $\dagger$-stable, the group \[\m B_n^{-1} \eta(\Lambda_{D_q}^{\max}) \m B_n\cong \GL_5(\mc O_{E_{q}}) \cong \GL_5(\mc O_{E_{\rho_q}})\times\GL_5(\mc O_{E_{\bar \rho_q}})\] is $\Ad(\m B_n^{\dagger}\m B_n)\circ (\cdot)^{\dagger}$-stable, thus $\m B_n^{\dagger}\m B_n\in \Q_q^{\times}\GL_5(\mc O_{E_{q}})$.

We can conjugate further to get $U_5(\Z/n)$:

\begin{lem}\label{lem:unitarymodn}
Let $n$ be relatively prime to $2\cdot7$. Assume that Hermitian $\m  H \in \Q_q^{\times}\GL_5(\mc O_{E_q})$ for all  $q|n$. Then there is $\m C_n \in \GL_5(\mc O_E/n)$ such that 
\[
\m C_n^{-1} U_5^{\m H}(\Z/n)\m C_n = U_5(\Z/n).
\]
\end{lem}

\begin{proof}

For $q|n$, we first claim that there exists $\m C_q\in\GL_5(\mc O_{E_q})$ such that
\[
\m C_q^{-1} U_5^{\m H}(\mc O_{E_q}/q^{v_q(n)})\m C_q = U_5(\mc O_{E_q}/q^{v_q(n)}).
\]
If $q$ is inert in $E/\Q$, using \cite[Theorem 7.1]{Jac62}, there exists $\m C_q\in\GL_5(\mc O_{E_q})$ and a scalar $\lambda\in \Q_q^{\times}$ such that $\m C_q^{\dagger}\m H\m C_q=\lambda  \m J$. 
If $q=\rho_q\bar \rho_q$ is split over $E/\Q$, then both unitary groups become split and isomorphic to $\GL_5(\Z/q^{v_q(n)})$. Identify  $\GL_5(\mc O_{E_q})$ with $\GL_5(\mc O_{E_{\rho_q}})\times\GL_5(\mc O_{E_{\bar \rho_q}})$ and write $\m H=(\m X,\m X^{\dagger})$ and $\m J=(\m J,\m J)$.  Then we may construct $\m C_q=(\m X^{-1}\m J,\m I_5)$ such that $\m C_q\in \GL_5(\mc O_{E_q})$ and $\m C_q^{\dagger}\m H\m C_q= \lambda\m J$ for some $\lambda\in\ \Q_q^{\times}$. 

In general, let $\m C_n$ be a matrix in $\GL_5(\mc O_E/n)$ such that $\m C_n \equiv \m C_q \pmod{q^{v_q(n)}}$ for all $q|n$, which exists by the Chinese remainder theorem. Then $\m C_n^\dagger \m H \m C_n \equiv \lambda \m J \pmod n$ for some $\lb \in \Q^\times$ which 
implies that $\m C_n^{-1} U_5^{\m H}(\Z/n)\m C_n = U_5(\Z/n)$. 

\end{proof}

\subsection{Explicit Gates}\label{ssec explicitgates}

By Proposition \ref{prop model}, the order $\Lambda_D^{\max}$ satisfies the conditions of Assumption \ref{a:orderforgates}. The trivialization \eqref{eq:Dlocaltriv} also identifies
\[
\Lambda_{D_{p}}^{\max} = \Mat_{5 \times 5}(\mc O_{E_p}). 
\]
which we can, without loss of generality, conjugate to satisfy the other conditions in \S\ref{ssec:findinggates}. 

Therefore we get from \S\ref{ssec:findinggates} that
\[
S_{\bar \Lambda_{0,p}} = \{p^{-1} g \, : \,  \iota(g)g= p^2,\, g \in \Lambda_D^{\max},\, g \notin p \Lambda_D^{\max} \}
\]
at inert $p$ and 
\[
S_{\bar \Lambda_{0,p}} = \{\bar \rho_p^{-1} g \, : \,  \iota(g)g = p,\, g \in \Lambda_D^{\max},\, g \notin \bar \rho_p \Lambda_D^{\max} \cup \rho_p \Lambda_D^{\max}\}
\]
at split $p=\rho_p\bar \rho_p$. 

\subsubsection{Finding gates}

Finding explicit gate elements is the hardest part of our algorithm, which we leave to Section \ref{sec:findinggates} for a more detailed discussion.

\subsubsection{Counts}\label{sssec: Countgates}
When $p$ is split, gates are in bijection with vertices at distance $1$ from $x_0$. By a standard formula in terms of $p$-binomial coefficients, this is
\[
|S_{\bar \Lambda_{0,p}}| = \sum_{i=1}^4 {\binom 5i}_p \qquad p \text{ split}. 
\]
For example, $|S_{\bar \Lambda_{0,p}}| =  3961830$ when $p=11$. 

When $p$ is inert, gates are in bijection with distance-$2$ hyperspecial vertices. Each of these corresponds to a unique non-special vertex at distance $1$. 
As we will see in \S\ref{sec:complextype}, there are $(p^5 + 1)(p^2 + 1)$ such non-special vertices and $(p^3+1)(p^5+1)$ such special but non-hyperspecial vertices. Each of these non-special (resp. special but non-hyperspecial) $x_1$ corresponds to $p$ (resp. $(p+1)(p^3+1)-1$) distance-$2$ vertices $x_2$ (in the notation of \S\ref{sec:complextype}, if we fix $x_1$, then $x_0$ and all possible $x_2$ form the set of vertices of type 0 adjacent to $x_1$, which is of type $1$ or $2$). Therefore
\[
|S_{\bar \Lambda_{0,p}}| = p(p^2 + 1)(p^5 + 1)+[(p+1)(p^3+1)-1](p^3+1)(p^5+1) \qquad p \text{ inert}. 
\]
For example, $|S_{\bar \Lambda_{0,p}}| =  765672$ when $p=3$ and $297782760$ when $p=5$.

\subsection{Reductions of Gates}

We finally discuss our choices of $K^{\infty, p}$ and the reductions of $S_{\bar \Lambda_{0,p}}$ in $\bar \Lambda_{0,p}/\bar \Lambda_p$.

\subsubsection{Choice of $K^{\infty, p}$}

We now consider $n$ relatively prime to $2 \cdot 7 \cdot p$. Let 
\[
n = \prod_{q_S \text{ split}} q_S^{n_{q_S}} \prod_{q_I \text{ inert}} q_I^{n_{q_I}} 
\]
and define
\[
K^{\infty,p} := \prod_{q \neq p} K_{0,q}(q^{n_q})
\]
where $K_{0,q}(q^{n_q})$ denotes the principal congruence subgroup of level $q^{n_q}$ in $K_{0,q} \subseteq G_q$. Then
\[
K_{0,q}/K_{0,q}(q^{n_q}) \cong 
\begin{cases}
\GL_5(\Z/q^{n_q}) & q \text{ split,} \\
U_5(\Z/q^{n_q}) & q \text{ non-split,}
\end{cases}
\]
where $U_5$ is defined as in equation \eqref{eq:Umodndef}. 

Then, since $G(\Q) \subseteq G^{p, \infty}$ is dense:
\begin{equation}\label{eq:Lambdaquotient}
\begin{aligned}
\Lambda_{0,p}/\Lambda_p &= (G(\Q) \cap K_0^{\infty, p})/(G(\Q) \cap K^{\infty, p}) = K_0^{\infty, p}/K^{\infty, p} \\ &
= \prod_{q_S \text{ split}} \GL_5(\Z/q^{n_{q_S}}) \prod_{q_I \text{ inert}} U_5(\Z/q^{n_{q_I}}) =: U_5(\Z/n). 
\end{aligned}
\end{equation}
Since $\Lambda_{0,p} \cap Z_{G_p}^\spl = \Lambda_{0,p} \cap Z_G^{\spl_p}$ where $Z_G^{\spl_p}$ denotes the part of the center split over $\Q_p$, we also get
\begin{equation}\label{eq:Lambdabarquotient}
\bar \Lambda_{0,p}/\bar \Lambda_p = U^\sharp_5(\Z/n) := 
\begin{cases}
    U_5(\Z/n)/U_1(\Z/n) & p \text{ split}, \\
    U_5(\Z/n)  & p \text{ inert}.
\end{cases}
\end{equation}

We can also compute the sizes of these groups as products over primes dividing $n$. For prime $n$, the sizes are standard and can be found in \cite[\S2]{ATLAS}. For prime powers $n$, we additionally use that at unramified places $[K_{0,q}(q^{i+1}) : K_{0,q}(q^i)] = q^{\dim G}$ for $i \geq 1$ (see, e.g., \cite[Theorem 13.5.1(1)]{KP23}). The formulas for general $n$ then follow from the Chinese remainder theorem.
\begin{align*}
|U_5(\Z/n)| &=  n^{25}  \prod_{\substack{q_S | n \\ q_S \text{ split}}}  \lf( \f1{q_S^{25}} \prod_{i=0}^4 (q_S^5 - q_S^i) \ri) \times \prod_{\substack{q_I | n \\ q_I \text{ inert}}} \lf( \f1{q_I^{25}} \prod_{i=0}^4 (q_I^5 + (-1)^i q_I^i) \ri), \\
|U_1(\Z/n)| &= n\prod_{\substack{q_S | n \\ q_S \text{ split}}} \f{q_S - 1}{q_I} \times \prod_{\substack{q_I | n \\ q_I \text{ inert}}} \f{q_I + 1}{q_I}.
\end{align*}

\subsubsection{Explicit Reductions}
We next need to compute the images of the elements of $S_{\bar \Lambda_p}$ in $U^\sharp_5(\Z/n)$. As in Lemmas \ref{lem:Dtrivmodn} and \ref{lem:unitarymodn}, choose $\m B_n \in \GL_5(\mc O_{L,(n)})$ and $\m C_n \in \GL_5(\mc O_E/n)$ so that the reduction modulo $n$ defines a surjection
\[
\m B_n^{-1} \eta((\Lambda_D^{\max})^\infty) \m B_n \onto \Mat_{5 \times 5}(\mc O_E),
\]
and 
\[
(\m B_n \m C_n )^\dagger (\m B_n \m C_n ) = 1.
\]
Then:
\begin{lem}\label{lem:gatereduction}
Reduction mod $n$ induces an isomorphism
\[
\m C_n^{-1} \lf( \m B_n^{-1} \eta(\Lambda_{0,p}) \m B_n/ \m B_n^{-1} \eta(\Lambda_p) \m B_n \ri) \m C_n \iso U_5(\Z/n).
\]

\end{lem}

\begin{proof}
By equation \eqref{eq:Lambdaquotient}, we have
\[
\m B_n^{-1} \eta(\Lambda_{0,p}) \m B_n/ \m B_n^{-1} \eta(\Lambda_p) \m B_n = \m B_n^{-1} \eta(K_0^{\infty, p}) \m B_n/ \m B_n^{-1} \eta(K^{\infty,p}) \m B_n,
\]
where $S$ is the set of primes dividing $n$. By the definition of principal congruence subgroups, $\m B_n^{-1} \eta(K^{\infty,p}) \m B_n$ is exactly the set of elements of $\m B_n^{-1} \eta(K_0^{\infty,p}) \m B_n$ that are $1$ mod $n$. Therefore, by Lemma \ref{lem:Kmodn}, this quotient is exactly $U_5^{\m B_n^\dagger \m B_n}(\Z/n)$. 

The result then follows from the definition of $\m C_n$. 
\end{proof}

\subsection{Gates to Vertices in the Building}\label{sec: localstru}

The last piece of data we need is an explicit assignment of $x \in \mc X$  to gates $s_x \in S_{\bar \Lambda_{0,p}}$ as in Construction \ref{cons:split} or \ref{cons:nonsplit}.

\subsubsection{Labeling $\mc X$}
We first need to label the vertices of $\mc X$ and describe its simplices. 

In the split case, let $p = \rho_p \bar \rho_p$, and fix an isomorphism $G_p \cong \GL_5( E_{\rho_p})$. We assume that $x_0$ is stabilized by $\GL_5(\mc O_{E_{\rho_p}})$ under this isomorphism. Then, by a standard construction (e.g., \cite[\S15.1]{KP23}), the vertices of $\mc X$ correspond to lattices $\mc M \subset E_{\rho_p}^5$ such that
\[
\rho_p \mc O_{E_{\rho_p}}^5 \subsetneq \mc M \subsetneq \mc O_{E_{\rho_p}}^5.
\]
For $1 \leq k \leq 5$, the $k$-simplices containing $x_0$ correspond to sequences of lattices $\mc M_i$ with
\[
\rho_p\mc O_{E_{\rho_p}}^5 \subsetneq \mc M_1 \subsetneq \cdots \subsetneq \mc M_k \subsetneq \mc O_{E_{\rho_p}}^5
\]
Finally, the $G_p$-action is through the isomorphism $G_p \cong \GL_5( E_{\rho_p})$. 

Reducing mod $\rho_p$, we get the following construction:
\begin{construction}\label{cons:Qsplit}
Let $p = \rho_p \bar \rho_p$ be split in $E$. Then the vertices of $Q$ in Construction \ref{cons:split} can be identified with the subspaces 
\[
0 \neq \bar{\mc M} \subsetneq (\mc O_{E_{\rho_p}}/\rho_p)^5.
\]
with the $k$-simplices containing $x_0$ corresponding to sequences of such $\bar{\mc M_i}$ for $1 \leq i \leq k$ satisfying that
\[
0 \neq \bar{\mc M}_1 \subsetneq \cdots \subsetneq \bar{\mc M}_k \subsetneq (\mc O_{E_{\rho_p}}/\rho_p)^5.
\]
Furthermore, given $\bar{\mc M}$ as above, then for any $g \in \GL_5(E_{\rho_p})$ such that $gx_0\in Q$, we have $gx_0 = \bar{\mc M}$ 
if and only if there is $\lb \in E_{\rho_p}$ such that $\lb g \in \GL_5(E_{\rho_p}) \cap \Mat_{5 \times 5}(\mc O_{E_{\rho_p}})$ and $\lb g (\mc O_{E_{\rho_p}}/\rho_p)^5 = \bar{\mc M}$.

\end{construction}

In the inert case, we have explicitly constructed a Hermitian form $\m 
H_p:=\m B_{p^2}^{\dagger}\m B_{p^2}\in\GL_5(\mc O_{E_p})$ on $E_{p}^5$ and the related identification $G_p\cong U_{5}^{\m  H_p}(\Q_p)$. Assume that $x_0$ is stabilized by $\GL_5(\mc O_{E_p})$. Again, by a standard construction (e.g., \cite[\S15.2]{KP23}),
 
 \begin{itemize}

 \item the vertex $x_0$ corresponds to the $\m H_p$-selfdual\footnote{Here we say that an $\mc O_{E_p}$-module $\mc M_2$ is the dual of $\mc M_1$ if $\mc M_2=\{x\in E_p^5 :  \m H_p(\mc M_1,x)\in \mc O_{E_p}\}$.} lattice $\mc O_{E_p}^5$;

 \item the vertices $x_1$ of $\mc X$ at distance 1 correspond to pairs of lattices $\mc M_1,\mc M_2\subset E_{p}^5$ such that 
 \[
 p \mc O_{E_{p}}^5 \subsetneq \mc M_1 \subsetneq \mc M_2 \subsetneq \mc O_{E_{p}}^5\]
 and $\mc M_1$ is the $\m H_p$-dual lattice of $p^{-1}\mc M_2$; 

 \item Given two pairs of lattices $(\mc M_1,\mc M_2)$ and $(\mc M_1',\mc M_2')$ corresponding to two distance 1 lattices $x_1,x_1'$ as above, $\{x_0,x_1,x_1'\}$ forms a simplex if and only if either $\mc M_1 \subsetneq \mc M_1'$ or $\mc M_1' \subsetneq \mc M_1$;

 \item The type-$0$ vertices $x_2$ of $\mc X$ at distance 2 correspond to lattices $\mc M \subset E_{p}^5$ such that $p^{-1}\mc M$ is $\m H_p$-selfdual and 
 \[
 p^2 \mc O_{E_{p}}^5 \subsetneq \mc M \subsetneq \mc O_{E_{p}}^5\]
 Furthermore, given a distance 2 vertex $x_2$ and the associated lattice $\mc M$, the distance 1 vertices $x_1$ between $x_0$ and $x_2$ correspond exactly to those pairs $(\mc M_1,\mc M_2)$ such that
 \[\mc M\subsetneq \mc M_1 \subsetneq \mc M_2.\]
Moreover, $x_1$ lies in the convex hull of $x_0$ and $x_2$ if and only if the associated $\mc M_1$ is minimal among all possible $x_1$.

If we fix an apartment containing $x_0$ and $x_2$, there may be either one or three such $x_1$, depending on whether the convex hull consists of two long edges or two short edges in the reduced building. In the latter case, let $(\mc M_1,\mc M_2)$, $(\mc M_1',\mc M_2')$, and $(\mc M_1'',\mc M_2'')$ be the three pairs corresponding to distance 1 vertices between $x_0$ and $x_2$. Then $(\mc M_1,\mc M_2)$ corresponds to the midpoint of $x_0$ and $x_2$ if and only if $\mc M_1 \subsetneq \mc M_1'$ and $\mc M_1 \subsetneq \mc M_1''$.
 
 \end{itemize} 

 \begin{ex}

 Take $\m H_p=\m J$. Then, $x_0$ relates to $\mc O_{E_{p}}^5$, $x_2$ relates to $p^2\mc O_{E_{p}}\oplus p\mc O_{E_{p}}^{3}\oplus \mc O_{E_{p}}$, $x_2'$ relates to $p^2\mc O_{E_{p}}^{2 }\oplus p\mc O_{E_{p}}\oplus \mc O_{E_{p}}^{2}$, $x_1$ relates to $p\mc O_{E_{p}}^{4}\oplus \mc O_{E_{p}}\subset p\mc O_{E_{p}}\oplus \mc O_{E_{p}}^{4}$, $x_1'$ relates to $p\mc O_{E_{p}}^{3}\oplus \mc O_{E_{p}}^{2}\subset p\mc O_{E_{p}}^{2}\oplus \mc O_{E_{p}}^{3}$, $x_1''$ relates to $p\mc O_{E_{p}}\oplus \mc O_{E_p} \oplus p\mc O_{E_{p}}^{2}\oplus \mc O_{E_p}\subset p\mc O_{E_{p}}\oplus \mc O_{E_p}^{2} \oplus p\mc O_{E_{p}}\oplus \mc O_{E_p}$. Then $x_1'$ is the unique vertex between $x_0$ and $x_2'$, and $x_1,x_1',x_1''$ are the three vertices between $x_0$ and $x_2$ in a fixed apartment (but there could be more such vertices in other apartments).
 
 \end{ex} 

 Reducing modulo $p^2$, we get the following construction:

\begin{construction}\label{cons:Qinert}

Let $p$ be inert in $E$. Consider the rank 5 free $\mc O_{E_p}/p^2$-module $(\mc O_{E_p}/p^2)^5$ and realize $\m H_p$ as a Hermitian form on it.
Then the vertices in $Q$ in Construction \ref{cons:nonsplit} can be identified with the $\mc O_{E_p}/p^2$-submodules 
\[
0 \neq \bar{\mc M} \subsetneq (\mc O_{E_{p}}/p^2)^5
\]
that are $\m H_p$-selfdual\footnote{Here we say that an $\mc O_{E_{p}}/p^2$-module $\bar{\mc M}_2$ is the dual of $\bar{\mc M}_1$ if $\bar{\mc M}_2=\{x\in (\mc O_{E_{p}}/p^2)^5 :  \m H_p(\bar{\mc M}_1,x)=0\}$.}. Furthermore, for any $g \in U_5^{\m H_p}(\Q_p)$ such that $gx_0\in Q$, we have $gx_0 = \bar{\mc M}$ 
if and only if there is $\lb \in E_{p}$ satisfying
\[
\lb g \in \GL_5(E_{p}) \cap \Mat_{5 \times 5}(\mc O_{E_{p}})\quad \text{and}\quad \lb g (\mc O_{E_{\rho_p}}/p^2)^5 = \bar{\mc M}.
\] 
The vertices in $\mc X$ at distance 1 can be identified with pairs of $\mc O_{E_p}/p^2$-submodules 
\[
(p\mc O_{E_{p}}/p^2)^5 \subsetneq \bar{\mc M}_1 \subsetneq  \bar{\mc M}_2 \subsetneq (\mc O_{E_{p}}/p^2)^5.
\]
such that $\bar{\mc M}_1$ is the $\m H_p$-dual of $p \bar{\mc M}_2$. A 2-simplex containing $x_0$ consists of two such pairs $(\bar{\mc M}_1,\bar{\mc M}_2)$ and $(\bar{\mc M}_1',\bar{\mc M}_2')$ such that 
\[
(p\mc O_{E_{p}}/p^2)^5 \subsetneq \bar{\mc M}_1 \subsetneq \bar{\mc M}_1' \subsetneq  \bar{\mc M}_2' \subsetneq \bar{\mc M}_2 \subsetneq (\mc O_{E_{p}}/p^2)^5.
\] 
A vertex at distance 1 related to $(\bar{\mc M}_1,\bar{\mc M}_2)$ is between $x_0$ and a vertex at distance 2 related to $\bar{\mc M}$ if $\bar{\mc M}\subsetneq \bar{\mc M}_1\subsetneq (\mc O_{E_{p}}/p^2)^5$, and moreover, it lies in the convex hull if the related $\bar{\mc M}_1$ is minimal.

\end{construction}

\subsubsection{Matching to Gates}
Now we can match $s_x$ to $x \in Q$. 

In the split case $p = \rho_p \bar \rho_p$, choose $\m B_p$ as in Lemma \ref{lem:Dtrivmodn}. Then 
\[
\m B_p^{-1}\eta(S_{\bar \Lambda_{0,p}})\m B_p \subseteq \GL_5(E_{\rho_p}) \cap \Mat_{5 \times 5}(\mc O_{E_{\rho_p}})\]
and 
\[\m B_p^{-1}\eta(S_{\bar \Lambda_{0,p}})\m B_p\cap \Mat_{5 \times 5}(\rho_p\mc O_{E_{\rho_p}})=\emptyset.
\]
So Lemma \ref{lem:Dtrivmodn} gives a reduction-mod-$\rho_p$ map as an injection
\[
\m B_p^{-1} \eta(S_{\bar \Lambda_{0,p}}) \m B_p \hookrightarrow \Mat_{5 \times 5}(\mc O_E/\rho_p). 
\]
Now we apply Construction \ref{cons:Qsplit}. Note that the condition $g (\mc O_{E_{\rho_p}}/\rho_p)^5 = \bar{\mc M}$ only depends on the reduction of $g$ mod $\rho_p$. In total, for each vertex $\bar{\mc M} \in Q$, we have that $s_{\bar{\mc M}}$ is the element in $S_{\bar \Lambda_{0,p}}$ such that 
\[
(\m B_p^{-1} \eta(s_{\bar{\mc M}}) \m B_p )(\mc O_{E_{\rho_p}}/\rho_p)^5 = \bar{\mc M}.
\]

In the inert case, we instead choose $\m B_{p^2}$ as in Lemma \ref{lem:Dtrivmodn}. Then similarly, 
\[
p\m B_{p^2}^{-1}\eta(S_{\bar \Lambda_{0,p}})\m B_{p^2} \subseteq \GL_5(E_{p}) \cap \Mat_{5 \times 5}(\mc O_{E_{p}})
\]
and
\[p\m B_{p^2}^{-1}\eta(S_{\bar \Lambda_{0,p}})\m B_{p^2}\cap \Mat_{5 \times 5}(p^2\mc O_{E_{p}})=\emptyset.\]
So Lemma \ref{lem:Dtrivmodn} gives a reduction-mod-$p^2$ map as an injection
\[
\m B_{p^2}^{-1} \eta(pS_{\bar \Lambda_{0,p}}) \m B_{p^2} \hookrightarrow \Mat_{5 \times 5}(\mc O_E/p^2). 
\]
We similarly input Construction \ref{cons:Qinert}. In total, for each vertex $\bar{\mc M} \in Q$, we have that $s_{\bar{\mc M}}$ is the element in $S_{\bar \Lambda_{0,p}}$ such that 
\[
(\m B_{p^2}^{-1} \eta(ps_{\bar{\mc M}}) \m B_{p^2} )(\mc O_{E_{p}}/p^2)^5 = \bar{\mc M}.
\]

\section{Finding Gates} \label{sec:findinggates}

Choose a prime $p \neq 2,7$ and let $e$ be $1$ if $p$ is split in $E$ and $2$ if $p$ is inert in $E$. In practice, we focus on $p^e=9$ or $11$, which correspond to the smallest inert and split cases respectively.

\subsection{Basic Reduction}

\begin{lem}\label{lem:otherfields}
Let $x \in \Lambda_D^{\max} \setminus \mc O_E$ such that $\iota(x)x = p^e$. Then 
\begin{itemize}

    \item $L' := E[x] \subseteq D$ is a CM field of degree 5 over $E$ on which $\iota$ acts as complex conjugation.

    \item Write $M'=(L')^\iota$ which is a totally real field of degree 5 over $\Q$. Write $x=\alpha+\rho_2\beta$ for $\alpha,\beta\in M'$ and $\rho_2=(1+\sqrt{-7})/2$. Then $M'=\Q(\alpha)=\Q(\beta)=\Q(2\alpha+\beta)$. 
    \item  $7\alpha,7\beta,2\alpha+\beta\in\mc O_{M'}\cap \Lambda_{D}^{\max}$. Furthermore $\alpha,\beta\in \mc O_{M'}$ when $M'/\Q$ is unramified at $7$. 
    
    \item The absolute value of all the real embeddings of $\beta$ (resp. $2\alpha+\beta$) are smaller than $\sqrt{4p^e/7}$ (resp. $\sqrt{4p^e}$).
    
    \item  The discriminant \[\Disc(M') \leq \begin{cases}
        C \cdot (4 p^e/7)^{10}&\quad M'\ \text{not ramified at }7,\\
        C \cdot (4 p^e)^{10}&\quad M'\ \text{ramified at }7,
    \end{cases}\]
    where the constant $C\approx 0.134288$.
\end{itemize}
\end{lem}

\begin{proof}

From our construction, $L'$ is an abelian division algebra over $E$ not equal to $E$. Since all maximal abelian subalgebras of $D$ are degree $5$ over $E$, $L'$ must be maximal and also degree $5$. 

Since $\iota(x) = p^e x^{-1}$, $\iota(x) \in E(x) = L'$ which forces $\iota(L') = L'$. Then, since $\{x \in D^\times_\infty : \iota(x)x = 1\}$ is compact, so is $\{x \in (L')^\times_\infty : \iota(x)x = 1\}$. This is only possible if $M'=L'^\iota$ is totally real, in which case $\iota$ has to act as complex conjugation since it is an involution.

Let $x = \alpha + \rho_2 \beta$ with $\alpha,\beta\in M'$. Then,
\[
\alpha^2 + \alpha \beta + 2 \beta^2 = p^e. 
\]
We cannot have $\Q(\alpha)$, $\Q(\beta)$ or $\Q(2\alpha+\beta)$ being strictly contained in $M'$, or equivalently $\alpha$, $\beta$ or $2\alpha+\beta$ are contained in $\Q$. Otherwise using the above equation $E[x]/E$ is of degree $1$ or $2$, which is impossible. 

By direct calculation, we have 
\[7\alpha=(\rho_2\iota(x)-\bar \rho_2 x)\sqrt{-7},\ 7\beta=(x-\iota(x))\sqrt{-7}\in \sqrt{-7}\Lambda_D^{\max}\cap M'\subset\mc O_{M'},\]
\[2\alpha+\beta=x+\iota(x)\in \Lambda_{D}^{\max}\cap M'\subset\mc O_{M'}.\]
Since the norm form as a polynomial in $\alpha$ has a real root: 
\[
(\beta^{(i)})^2 - 4(2 (\beta^{(i)})^2 - p^e) \geq 0 \implies  |\beta^{(i)}| \leq \sqrt{4 p^e/7}, 
\] 
where $\beta^{(i)}$ denote the image of $\beta$ under the $i$-th real embedding $M'\hookrightarrow\R$ for $i=1,\dots,5$. 
Since $(2\alpha+\beta)^{2}=4p^e-7\beta^2$, we have that every real embedding of $2\alpha+\beta$ is of absolute value not greater than $\sqrt{4p^e}$.

If $M'$ is not ramified at $7$, we have that $\alpha, \beta \in \mc O_{M'}$ since $(\Lambda_D^{\max})^\times \cap L'^\times \subseteq \mc O_{L'}$ and $\mc O_{L'} = \mc O_{M'} \mc O_E$. Therefore, we have that $\Z[\beta]$ is an order in $\mc O_{M'}$ and 
\[
|\Disc(M')| \leq |\Disc(\Z[\beta])| = \prod_{1\leq i < j\leq 5} |\beta^{(i)} - \beta^{(j)}|^2 
\]
If we had $-1 \leq \beta^{(i)} \leq 1$, we can numerically maximize the product on the right-hand side of the above equation and bound it by $C \approx 0.134288$. Scaling this by $(4 p^e/7)^{20/2}$ gives
\[
|\Disc(M')| \leq C \cdot (4 p^e/7)^{10}.
\]
If $M'$ is ramified at $7$,  as in the unramified case, we have
\[
|\Disc(M')| \leq |\Disc(\Z[2\alpha+\beta])| \leq   C \cdot (4 p^e)^{10}.
\]

\end{proof}

\begin{cor}\label{cor gates}

Let $x\in\Lambda_D^{\max}$ such that $\iota(x)x=p^e$. Then $x\in \bar \rho_p \Lambda_D^{\max} \cup \rho_p \Lambda_D^{\max}$ (resp. $x\in p\Lambda_D^{\max}$) in the case $E/\Q$ is split (resp. inert) over $p$ if and only if $x\in\mc O_E$.

\end{cor}

\begin{proof}

If $x\in\mc O_E$, then noting that $\mc O_E$ is a PID, we have that the ideal $(x)$ equals $(\rho_p)$ or $(\bar \rho_p)$ (resp. $(p)$) in the split case (resp. inert case). This implies that $x\in \bar \rho_p \Lambda_D^{\max} \cup \rho_p \Lambda_D^{\max}$ (resp. $x\in p\Lambda_D^{\max}$). Conversely if $x$ is not in $\mc O_E$, then $L'=E[x]$ is a CM field of degree 5 over $E$ and $x$ is in $\mc O_{L'}$. Assume that $x\in \bar \rho_p \Lambda_D^{\max} \cup \rho_p \Lambda_D^{\max}$ (resp. $x\in p\Lambda_D^{\max}$), then for $y=\rho_p^{-1}x$ or $\bar \rho_p^{-1}x$ (resp. $y=p^{-1}x$) in $\mc O_{L'}$, we have $\iota(y)y=1$. It is straightforward to check that the only units in $\mc O_{L'}$ are $1$ or $-1$. Then $y=\pm1$ and $x$ is in $E$, contradictory!

\end{proof}

\subsection{Quadratic Form, Characteristic Polynomial and Minimal Vector}\label{subsection QFCPMV}

We regard $(\Lambda_{D}^{\max})^{\iota}$ as a free $\Z$-module of rank 25 and consider the following quadratic form
\begin{equation}\label{eq quadformredtr}
(\Lambda_D^{\max})^{\iota}\rightarrow\Z,\quad \gamma\mapsto \Trd_{D/E}(\gamma^2), 
\end{equation}
induced by the reduced trace map. If $\gamma\notin E$, then $E(\gamma)$ is a  CM field of degree 5 over $E$, with $\Q(\gamma)=E(\gamma)^{\iota}$ being a totally real field of degree $5$ over $\Q$. Then
$\Trd_{D/E}(\gamma^2)=\mathrm{Tr}_{\Q(\gamma)/\Q}(\gamma^2)$ is a non-negative integer. If $\gamma\in E\cap (\Lambda_{D}^{\max})^{\iota}=\Z$, then
still $\mathrm{Tr}_{\Q(\gamma)/\Q}(\gamma^2)=5\gamma^2$ is a non-negative integer. As a result, the quadratic form is positive definite.  

Write $\gamma=\sum_{i=0}^4u^il_i$ with $l_i\in L$, then $\iota(\gamma)=\gamma$ implies that 
\[l_0=\bar l_0,\ l_1^\sigma=a^{-1}\bar l_4,\ l_2^{\sigma^2}=a^{-1}\bar l_3,\ l_3^{\sigma^3}=a^{-1}\bar l_2,\ l_4^{\sigma^4}=a^{-1}\bar l_1,\]
and thus
\[\Trd_{D/E}(\gamma^2)=\mathrm{Tr}_{M/\Q}(\sum_{i=0}^4l_i\bar l_i)\]
We write $l_0=m_0^0$, $l_1=m_1^0+\rho_2m_1^1$ and $l_2=m_2^0+\rho_2 m_2^1$ with $m_i^j\in M$, then we have
\begin{equation}\label{eq Trd beta2}
\begin{aligned}
\mathrm{Trd}_{D/E}(\gamma^2)=
\mathrm{Tr}_{M/\Q}( (m_0^0)^2+2((m_1^0)^2+m_1^0m_1^1+2(m_1^1)^2)+
\\2((m_2^0)^2+m_2^0m_2^1+2(m_2^1)^2))
\end{aligned}
\end{equation}
Notice that this equation concerns only  the totally real cyclic extension $M/\Q$. 

Let 
\[\vec{m}=((m_0^0)_{\Q},(m_1^0)_{\Q},(m_1^1)_{\Q},(m_2^0)_{\Q},(m_2^1)_{\Q})\] 
be a row vector in $\Q^{25}$, where $(m_i^j)_{\Q}$ denotes the corresponding row vector of $m_i^j$ in $\Q^5$ via the identification $M\cong \bigoplus_{i=0}^4\Q\alpha_i$. Using our discussion in \S \ref{subsection despLambdaDmaxiota}, each $\gamma\in(\Lambda_{D}^{\max})^{\iota}$ corresponds to a column vector $\vec{x}^T$ in $\Z^{25}$ satisfying
\[\vec{m}^T=B_{\max}\cdot\vec{x}^T.\] 
Then, consider the following explicit quadratic form on $\Z^{25}$:
\[Q_{\max}(\vec{x}):=\vec{x}\cdot B_{\max}^TA_0B_{\max}\cdot\vec{x}^T,\]
where we define
\[A_0=\diag\bigg(T,\begin{pmatrix}
    2T & T\\ T &4T
\end{pmatrix},\begin{pmatrix}
    2T & T\\ T &4T
\end{pmatrix}\bigg)
\quad\text{with}\quad T=(\mathrm{Tr}_{M/\Q}(\alpha_i\alpha_j))_{0\leq i,j\leq 4}.\]
By our discussion above, we have 
\[\Trd_{D/E}(\gamma^2)=Q_{\max}(\vec{x}).\]
It means that we have a concrete realization of the above quadratic form \eqref{eq quadformredtr}. Moreover, we may use the embedding \eqref{eq Dembedding} to calculate the characteristic polynomial of $\gamma$, which could be expressed as 
\begin{equation}\label{eq charpolygamma}
P_{\gamma}(t)=t^5+P_1(\vec x)t^4+P_2(\vec x)t^3+P_3(\vec x)t^2+P_4(\vec x)t+P_5(\vec x),   
\end{equation}
where for each $i=1,\dots,5$, $P_i(\vec{x})$ is a homogeneous polynomial of degree $i$ in $\Z[\vec{x}]$ (thus is of 25 variables). Using the explicit relation between $\vec{x}$ and $\gamma$, these polynomials $P_i(\vec x)$ could be calculated explicitly. In particular, we have $P_1(\vec x)=-\Trd_{D/E}(\gamma)$ and $Q_{\max}(\vec{x})=P_1(\vec x)^2-2P_2(\vec x)$.

Now we discuss a possible minimal short vector that can be picked. Consider a solution of $\iota(x)x=p^e$ in $\Lambda_D^{\max} \setminus \mc O_E$.
Using Lemma \ref{lem:otherfields},  we write $x=\alpha+\rho_2\beta$ such that 
\begin{itemize}

\item $\alpha,\beta$ lie in $(\Lambda_D^{\max})^{\iota}$ such that $M'=\Q(\alpha)=\Q(\beta)=\Q(2\alpha+\beta)$ is a totally real field of degree 5  and $L'=EM'$ is a CM field of degree 10 over $\Q$.

\item We have $7\alpha,7\beta,2\alpha+\beta\in \mc  O_{M'}\cap \Lambda_{D}^{\max}$. When $M'/\Q$ is unramified over $7$, we further have $\alpha,\beta\in \mc  O_{M'}$.

\item We have $\mathrm{Trd}_{D/E}(4\iota(x)x)=\mathrm{Trd}_{D/E}((2\alpha+\beta)^2)+\mathrm{Trd}_{D/E}(7\beta^2)=20p^e$.

\end{itemize}

Let $\gamma$ be an element in $\Lambda_D^{\max}\cap (\mc O_{M'}-\Z)$, such that $\Trd_{D/E}(\gamma^2)$ reaches the smallest value. Let $\vec{x}\in\Z^{25}$ be the vector related to $\gamma$. We have
\[Q_{\max}(\vec{x})=\Trd_{D/E}(\gamma^2)\leq (7/8)[\Trd_{D/E}((2\alpha+\beta)^2)+\mathrm{Trd}_{D/E}((7\beta)^2)/7]= 35p^e/2.\]

\subsection{Finding Subfields \lm{$L'\subset D$}}\label{subsection subfieldL'inD}

We explain how to list all possible maximal subfields $L'\subset D$, such that there exists an element $x\in L'\cap \Lambda_D^{\max} \setminus \mc O_E$ satisfying $\iota(x)x=p^e$. In particular, we necessarily have $x\in\mc O_{L'}$. 
Consider $M'=L'^{\iota}$, which is a totally real field of degree $5$ over $\Q$.

We need to check all CM fields $L'$ of degree $5$ over $E$, such that
\begin{itemize}

\item $L'$ can be embedded into $D$, or equivalently by the local-global compatibility, $L'$ is inert at $\rho_2$ and $\bar \rho_2$, or equivalently $M'=L'^{\iota}$ is inert at $2$.

\item By Lemma \ref{lem:otherfields}, $\Disc(M')\leq C\cdot 7^{10\varepsilon}\cdot (4 p^e/7)^{10}$ where $\varepsilon=0$ or $1$ depending on if $M'/\Q$ is unramified at $7$ or not. 

\item The equation $\iota(x)x=p^e$ has a solution in $\mc O_{L'} \setminus \mc O_E$.

\end{itemize}
The strategy is to list all $L'$ satisfying the first two conditions above, and then check one by one if the third condition is satisfied. 

To check if $\iota(x)x=p^e$ has a solution in $\mc O_{L'} \setminus \mc O_{E}$, a necessary (and effectively checkable) condition is that there is a principal ideal $I$ dividing $(p^e)$ and different from $(p)$ (resp. $(\rho_p)$ and $(\bar \rho_p)$) when $e=2$ (resp.  $e=1$) such that $\iota(I)I=(p^e)$. To find $I$, we only need to consider the decomposition of $(p^e)$ into the product of prime ideals, and check if we may divide these prime ideals into two multisets with their products equal to principal ideals $I$ and $\iota(I)$ respectively.  Remark that two solutions to $\iota(x)x=p^e$ such that $I=(x)$ differ by a unit $u\in\mc O_{L'}^\times$ satisfying $\iota(u)u=1$. Also, since the unit groups of $\mc O_{L'}$ and $\mc O_{M'}$ are of the same rank, both the kernel and the cokernel of the map $\mc O_{L'}^\times\rightarrow\mc O_{M'}^\times, y\mapsto \iota(y)y$ are finite. Thus $u$ is a root of unity. Finally, any CM field $L'$ of degree 5 over $E$ only has roots of unity $1$ and $-1$. Thus for each ideal $I$, there are at most 2 solutions $x,-x$ with $I=(x)$. Since there are only finitely many possible $I$, for each $L'$ there are finitely many solutions to $\iota(x)x=p^e$.

Using Sagemath, for $p^e=9$ or $11$ we checked the first 8000 CM fields $L'$ listed on the website LMFDB satisfying the first two conditions, and found around 200 fields among them having a solution $\iota(x)x=p^e$. We listed the related solutions as well. In particular, this list contains all the fields $L'$ such that $M'/\Q$ is unramified at $7$. If we furthermore  focus on the case where $M'/\Q$ are ramified at $7$, then for $p^e=9$ or $11$ there remain less than 40000 fields to be checked. When the discriminant becomes large it gets slower to check if a certain field $L'$ has a solution and to find the solutions. But in any case it is a one-time check. 

Hence at last we expect that for small $p^e$ we may list all CM fields $L'$ satisfying the above conditions, and $x\in \mc O_{L'} \setminus \mc O_{E}$ such that $\iota(x)x=p^e$. On the other hand, it is only a necessary condition, saying that some solutions $x$ might not be embedded into $\Lambda_D^{\max}$.

\subsection{Strategy}\label{subsection strategy}

Now, computing all the solutions to $\iota(x)x=p^e$ in $\Lambda_{D}^{\max} \setminus \mc O_{E}$ can be divided into the following three steps:

\begin{enumerate}

\item List all the possible totally real fields $M'/\Q$ of degree 5 and the related CM field $L'=EM'$, such that there exists a solution $\iota(x)x=p^e$ in $\mc O_{L'} \setminus \mc O_E$, and for each $L'$ find all the solutions $x$. Using Lemma \ref{lem:otherfields} we need to search over all possible totally real fields $M'$ with bounded discriminant, which can be done for small $p^e$.

\item List all $\gamma\in (\Lambda_{D}^{\max})^{\iota}$ such that $\mathrm{Trd}_{D/E}(\gamma^2)\leq 35p^e/2$. Equivalently, list all $\vec{x}\in \Z^{25}$ such that $Q_{\max}(\vec{x})\leq 35p^e/2$. Noting that here $Q_{\max}$ is a specific positive definite quadratic form of 25 variables and explicit coefficients. The main problem is that the number of solutions is gigantic---when $p^e=11$ and $35p^e/2=375/2$, there are around $10^{14}$ solutions $\vec{x}$, thus can hardly  be fully listed and further checked in Step (3) below. On the other hand, by taking a smaller bound, we are still able to find certain $\gamma$, which possibly leads to some solutions $x$. 

\item For each $\gamma$ in the list of Step (2), check if the field $M'=\Q(\gamma)$ lies in the list of Step (1), which can be done by calculating the characteristic polynomial $P_{\gamma}(t)$ of $\gamma$. 
After that, for any $x$ in $\mc O_{L'} \setminus \mc O_{E}$ such that $\iota(x)x=p^e$, we may express it as an explicit $\Q$-coefficient polynomial of $\gamma$ and $\rho_2$. Then, we may check directly if $x$ lies in $\Lambda_{D}^{\max}$ or not. This will theoretically give all the solutions to $\iota(x)x=p^e$ in $\Lambda_{D}^{\max} \setminus \mc O_E$. Combining with \S \ref{ssec explicitgates} and Corollary \ref{cor gates}, we find all the gates.

\end{enumerate}

More precisely, in Step (3) if we write $x=\alpha+\rho_2\beta$, then we have $M'=\Q(\alpha)=\Q(\beta)=\Q(2\alpha+\beta)$ and $\gamma$ can be considered to be a shortest vector in $\Z[7\alpha,7\beta,2\alpha+\beta] \setminus\Z\subset (\Lambda_{D}^{\max})^{\iota} \setminus \Z$. Thus by
solving the equation $\iota(x)x=p^e$ in $L'$, we may precompute the characteristic polynomial of $\gamma$ for each candidate $L'$ in Step (1). In that sense, we are essentially required to solve the following problem:

\begin{problem}\label{prob Qmaxboundcharpoly}

Find all vectors $\vec{x}\in\Z^{25}$ such that $Q_{\max}(\vec{x})\leq 35p^e/2$ and $(P_{1}(\vec{x}),P_{2}(\vec{x}),P_{3}(\vec{x}),P_{4}(\vec{x}),P_{5}(\vec{x}))$ lies in a fixed precomputed list of vectors in $\Z^5$. 
    
\end{problem}

The above problem is quite concrete. We hope that some techniques, including the modulo $n$ technique for small $n$, will reduce the complexity of the problem and lead to an effective algorithm of finding all the solutions for small $p^e$ (for instance $p^e=9$ or $11$). However solving this problem or even determining if it is computationally tractable on modern hardware is far out of the scope of the authors' current expertise.

\subsection{Initial Data}\label{subsection initialBmax}

For convenience of the reader, we list our choice $B_{\max}$ and $Q_{\max}$ in the setting of \S \ref{subsection initialdata}. 

Implementing the procedure of \S \ref{subsection despLambdaDmaxiota} in Sagemath, finds an explicit $25\times 25$ matrix $B_{\max}$ with rows:
\begin{align*}
\text{Row 1: } & (\tfrac{1}{11}, 0, \dots, 0);\\
\text{Row 2: } & (0, \tfrac{1}{11}, 0, \dots, 0);\\
\text{Row 3: } & (0, 0, \tfrac{1}{11}, 0, \dots, 0);\\
\text{Row 4: } & (\tfrac{6}{11}, \tfrac{2}{11}, \tfrac{1}{11}, 1, 0, \dots, 0);\\
\text{Row 5: } & (\tfrac{5}{11}, \tfrac{8}{11}, \tfrac{3}{11}, 0, 1, 0, \dots, 0);\\
\text{Row 6: } & (\tfrac{3}{121}, \tfrac{8}{121}, \tfrac{2}{121}, 0, 0, \tfrac{1}{11}, 0, \dots, 0);\\
\text{Row 7: } & (\tfrac{6}{121}, \tfrac{5}{121}, \tfrac{4}{121}, 0, 0, 0, \tfrac{1}{11}, 0, \dots, 0);\\
\text{Row 8: } & (\tfrac{10}{121}, \tfrac{1}{121}, \tfrac{3}{121}, 0, 0, 0, 0, \tfrac{1}{11}, 0, \dots, 0);\\
\text{Row 9: } & (\tfrac{7}{121}, \tfrac{4}{121}, \tfrac{1}{121}, 0, 0, 0, 0, 0, \tfrac{1}{11}, 0, \dots, 0);\\
\text{Row 10: } & (\tfrac{5}{121}, \tfrac{6}{121}, \tfrac{7}{121}, 0, 0, 0, 0, 0, 0, \tfrac{1}{11}, 0, \dots, 0);\\
\text{Row 11: } & (\tfrac{97}{121}, \tfrac{178}{121}, \tfrac{226}{121}, 0, 0, \tfrac{3}{11}, 0, 0, 0, 0, 2, 0, \dots, 0);\\
\text{Row 12: } & (\tfrac{194}{121}, \tfrac{81}{121}, \tfrac{210}{121}, 0, 0, 0, \tfrac{3}{11}, 0, 0, 0, 0,  2, 0, \dots, 0);\\
\text{Row 13: } & (\tfrac{162}{121}, \tfrac{113}{121}, \tfrac{97}{121}, 0, 0, 0, 0, \tfrac{3}{11}, 0, 0, 0, 0, 2, 0, \dots, 0);\\
\text{Row 14: } & (\tfrac{2485}{121}, \tfrac{2146}{121}, \tfrac{1807}{121}, 0, 0, 16, 4, 8, \tfrac{201}{11}, 6, 12, 4, 2, 22, 0, \dots, 0);\\
\text{Row 15: } & (\tfrac{1291}{121}, \tfrac{1162}{121}, \tfrac{791}{121}, 0, 0, 10, 20, 12, 4, \tfrac{25}{11}, 10, 16, 6, 0, 22, 0, \dots, 0);\\
\text{Row 16: } & (0, 0, 0, 0, 0, 0, 0, 0, 0, 0, 0, 0, 0, 0, 0, \tfrac{1}{121}, 0, \dots, 0);\\
\text{Row 17: } & (\tfrac{4}{121}, \tfrac{2}{121}, \tfrac{5}{121}, 0, 0, \tfrac{1}{121}, \tfrac{10}{121}, \tfrac{6}{121}, \tfrac{1}{121}, \tfrac{7}{121}, \tfrac{7}{121}, \tfrac{4}{121}, \tfrac{1}{121}, 0, 0, \tfrac{2}{121}, \\
&\ \tfrac{1}{11}, 0, \dots, 0);\\
\text{Row 18: } & (\tfrac{9}{121}, \tfrac{6}{121}, \tfrac{5}{121}, 0, 0, \tfrac{10}{121}, \tfrac{9}{121}, \tfrac{10}{121}, \tfrac{5}{121}, \tfrac{9}{121}, \tfrac{5}{121}, \tfrac{6}{121}, \tfrac{7}{121}, 0, 0, \tfrac{7}{121}, 0,\\
&\ \tfrac{1}{11}, 0, \dots, 0);\\
\text{Row 19: } & (\tfrac{7}{121}, \tfrac{9}{121}, \tfrac{1}{121}, 0, 0, \tfrac{1}{121}, \tfrac{7}{121}, 0, \tfrac{7}{121}, \tfrac{1}{121}, \tfrac{8}{121}, \tfrac{3}{121}, \tfrac{9}{121}, 0, 0, \\
&\ \tfrac{6}{121}, 0, 0, \tfrac{1}{11}, 0, \dots, 0);\\
\text{Row 20: } & (\tfrac{10}{121}, \tfrac{6}{121}, \tfrac{4}{121}, 0, 0, \tfrac{5}{121}, \tfrac{8}{121}, \tfrac{1}{121}, \tfrac{1}{121}, \tfrac{6}{121}, \tfrac{5}{121}, \tfrac{6}{121}, \tfrac{7}{121}, 0, 0, \tfrac{9}{121}, 0, 0, 0, \\
&\ \tfrac{1}{11}, 0, \dots, 0 ) ;\\
\text{Row 21: } & (12, 18, 4, 0, 0, 12, 20, 4, 14, 16, 0, 0, 0, 0, 0, \tfrac{597}{121}, 0, 0, 0, 0, 22, 0, 0, 0, 0);\\
\text{Row 22: } & (\tfrac{2146}{121}, \tfrac{710}{121}, \tfrac{81}{121}, 0, 0, \tfrac{839}{121}, \tfrac{162}{121}, \tfrac{1888}{121}, \tfrac{1323}{121}, \tfrac{65}{121}, \tfrac{1517}{121}, \tfrac{2388}{121}, \tfrac{597}{121}, 0, 0, \tfrac{1194}{121}, \tfrac{113}{11},\\
&\ 0, 0, 0, 0, 22, 0, 0, 0 );\\
\text{Row 23: } & (\tfrac{1501}{121}, \tfrac{920}{121}, \tfrac{565}{121}, 0, 0, \tfrac{162}{121}, \tfrac{1017}{121}, \tfrac{1372}{121}, \tfrac{1533}{121}, \tfrac{291}{121}, \tfrac{323}{121}, \tfrac{920}{121}, \tfrac{1517}{121}, 0, 0, \tfrac{1517}{121}, 0, \tfrac{113}{11},\\
&\ 0, 0, 0, 0, 22, 0, 0);\\
\text{Row 24: } & (\tfrac{549}{121}, \tfrac{1501}{121}, \tfrac{597}{121}, 0, 0, \tfrac{1323}{121}, \tfrac{65}{121}, 2, \tfrac{1033}{121}, \tfrac{1565}{121}, \tfrac{2114}{121}, \tfrac{1791}{121}, \tfrac{49}{121}, 0, 0, \tfrac{920}{121}, 0, 0, \tfrac{113}{11},\\
&\ 0, 0, 0, 0, 22, 0);\\
\text{Row 25: } & (\tfrac{12988}{121}, \tfrac{23910}{121}, \tfrac{16908}{121}, 0, 0, \tfrac{16053}{121}, \tfrac{18570}{121}, \tfrac{28911}{121}, \tfrac{13181}{121}, \tfrac{18344}{121}, \tfrac{8309}{121}, \tfrac{8906}{121}, \tfrac{1517}{121},\\
&\ 0, 0, \tfrac{26669}{121}, 176, 132, 220, \tfrac{597}{11}, 44, 154, 176, 132, 242
).
\end{align*}

Using this $B_{\max}$, we find an explicit $Q_{\max}$ such that the related symmetric matrix $B_{\max}^TA_0 B_{\max}$ has coefficients that are all non-negative integers. Since it is already explicit, we won't display its entries for brevity. Neither will we display the $50 \times 50$ matrix $A_{\max}$, although such a matrix is helpful in verifying if an element $x$ in $D$ lies in $\Lambda_{D}^{\max}$. To compensate, by observing $\eta(x)\in\GL_{5}(L)$ one can still judge if $x$ is in $\Lambda_{D}^{\max}$ or not.

\subsection{Finding an Explicit Gate Element}\label{subsection findingexplicitgates} 

In this part, we explain our strategy for finding an explicit gate element. We assume $p^e=9$ or $11$.

\begin{enumerate}

\item We listed all the totally real fields $M'$ with the CM fields $L'$ having a solution of $\iota(x)x=p^e$ with
\begin{itemize}
\item $M'/\Q$ inert at 2;
\item $\Disc(M')$ bounded by $0.134288\cdot (44/7)^{10}\approx 12959555$ for $M'/\Q$ unramified at 7, and $0.134288\cdot 44^{10}\approx 3.6523\times 10^{15}$ for $M'/\Q$ ramified at 7.

\end{itemize}
We checked the first 8000 totally real fields $M'$ listed on LMFDB and found 139 candidates that have a norm solution in $L'$. Here, we list the smallest and the largest $M'$ and their discriminant: $\Q[t]/(t^5 - 2t^4 - 3t^3 + 5t^2 + t - 1),\ 36497,\ \Q[t]/(t^5 - 2t^4 - 9t^3 + 7t^2 + 7t + 1),\ 8733025.$  

\item We listed all $\vec{x}$ such that $Q_{\max}(\vec{x})\leq 40$. It turns out that there are $463798$ solutions. We restore the corresponding $\gamma$.

\item We checked for $\gamma$ in Step (2) if $\Q(\gamma)$ occurs in the precomputed list in Step (1). It turns out that there are only several  hundred candidates.

\item  We checked if each related solution $x$ in Step (3) as a polynomial of $\rho_2$ and $\gamma$ is in $\Lambda_D^{\max}$. 

\end{enumerate}

\begin{ex}

Within the above range we already found an element $\vec{x}$ and a corresponding $\gamma\in(\Lambda_{D}^{\max})^{\iota}$ such that
\begin{itemize}
\item $Q_{\max}(\vec{x})=\Trd_{D/E}(\gamma^2)=34$;

\item $(m_0^0,m_1^0,m_1^1,m_2^0,m_2^1)=(
-\frac{2}{11}\alpha_1^4 + \frac{7}{11}\alpha_1^2 - \frac{3}{11}\alpha_1 - \frac{2}{11}, 
-\frac{34}{121}\alpha_1^4 + \frac{36}{121}\alpha_1^3 + \frac{64}{121}\alpha_1^2 - \frac{10}{121}\alpha_1 + \frac{61}{121}, 
\frac{30}{121}\alpha_1^4 - \frac{46}{121}\alpha_1^3 - \frac{28}{121}\alpha_1^2 + \frac{80}{121}\alpha_1 - \frac{125}{121}, 
\frac{45}{121}\alpha_1^4 + \frac{43}{121}\alpha_1^3 - \frac{229}{121}\alpha_1^2 - \frac{9}{11}\alpha_1 + \frac{177}{121}, 
\frac{3}{121}\alpha_1^4 + \frac{19}{121}\alpha_1^3 + \frac{17}{121}\alpha_1^2 - \frac{5}{11}\alpha_1 - \frac{85}{121}
)$.

\item $\gamma$ has characteristic polynomial $x^5 - 17x^3 + 21x^2 + 21x - 25$, which defines a field $M'$ of discriminant $89417$ and unramified at $7$. 

\item For $\alpha = \frac{1}{11}\gamma^4 + \frac{2}{11}\gamma^3 - \frac{8}{11}\gamma^2 + \frac{3}{11}\gamma - \frac{17}{11}, 
\beta = -\frac{3}{11}\gamma^4 - \frac{4}{11}\gamma^3 + \frac{43}{11}\gamma^2 - \frac{9}{11}\gamma - \frac{42}{11}$, we checked that $x=\alpha+
\beta\rho_2$ is a solution of $\iota(x)x=11$.

\item Finally, $x\in\Lambda_{D}^{\max}$. This can be checked by first converting $x$ to the related column vector $(\vec x_L)^T\in\Q^{50}$, and then verifying that $ A_{\max}\cdot (\vec x_L)^{T}\in \Z^{50}$.

\end{itemize}
    
\end{ex}

\section{Algorithm Summary}\label{sec:algo}
We now summarize the complete algorithm constructing our Ramanujan complexes $X_G(K^{\infty, p})$. 

\subsection{Structure of Output Complexes}\label{sec:complextype} 
We compute the details of the local structure of the output complexes as in \S\ref{ssec:abstractstructure}. 

\subsubsection{Split case}\label{sssec:splitcomplextype}
If $v_0$ is split, $\mc B(G_{v_0})$ has type $A_4$. The types of simplices are orbits of subsets of $\{0,1,2,3,4\}$ under the cyclic permutation $(1234)$. We can compute how many simplices of each type a vertex is contained in terms of sizes of various flag varieties: 
\[
\begin{array}{@{} c l @{\qquad} c l @{}}
\{0,1\} &
\dfrac{2(q_v^{5}-1)}{(q_v-1)} &
\{0,1,3\} &
\dfrac{3(q_v^{5}-1)(q_v^{4}-1)(q_v^{3}-1)}{(q_v^{2}-1)(q_v-1)^{2}} \\[8pt]
\{0,2\} &
\dfrac{2(q_v^{5}-1)(q_v^{4}-1)}{(q_v^{2}-1)(q_v-1)} &
\{0,1,2,3\} &
\dfrac{4(q_v^{5}-1)(q_v^{4}-1)(q_v^{3}-1)}{(q_v-1)^{3}} \\[8pt]
\{0,1,2\} &
\dfrac{3(q_v^{5}-1)(q_v^{4}-1)}{(q_v-1)^{2}} &
\{0,1,2,3,4\} &
\dfrac{(q_v^{5}-1)(q_v^{4}-1)(q_v^{3}-1)(q_v^{2}-1)}{(q_v-1)^{4}}
\end{array}
\]
Each $4$-simplex is also contained in $q_v  + 1$ chambers. 

Note that as in \S\ref{ssec:abstractstructure}, these complexes have universal covers different from those constructed in \cite{LSV05a, LSV05b}.

\subsubsection{Inert case}\label{sssec:urcomplextype}
If $v_0$ is inert, then $\mc B(G_{v_0})$ has type ${}^2 \! A'_4$. There are three types of vertices $\{0,1,2\}$, without loss of generality, labeled such that $0$ is hyperspecial and $2$ is non-hyperspecial special. The higher-dimensional simplices have type that are subsets of $\{0,1,2\}$ and $\mc B(G_{v_0})$ is made up of apartments that look like Figure \ref{fig:inertapart}.

We again calculate the degrees of each type of chamber at each vertex in terms of sizes of flag varieties:
\[ 
\renewcommand{\arraystretch}{1.25} 
\begin{array}{@{} c|cccc@{}}
& \{0,1,2\} &
  \{0,1\} & \{0,2\} & \{1,2\}\\ \hline 
0 &
(q_v^2 + 1)(q_v^{3}+1)(q_v^{5}+1)&
(q_v^{2}+1)(q_v^{5}+1) &
(q_v^{3}+1)(q_v^{5}+1) &
\text{---}\\
1 &
(q_v + 1)(q_v^{3}+1) &
(q_v+1) &
\text{---} &
(q_v^{3}+1)\\
2 &
(q_v+1)(q_v^2 + 1)(q_v^3 + 1) &
\text{---} &
(q_v+1)(q_v^{3}+1) &
(q_v^{2}+1)(q_v^{3}+1)
\end{array}
\]
Note that the reductive quotients of special fibers of the integral models at a vertex of type $0,1,2$ are $U(5), U(3) \times U(2)$, and $U(4)$ respectively. Formulas for point counts of unitary groups over finite fields may be found in \cite[\S2]{ATLAS}. 

Alternatively, the numbered labels on vertices in the tables of \cite{Tit79} give: 
\begin{itemize}
    \item Each edge of type $\{0,1\}$ is contained in $(q_v^3 + 1)$-chambers.
    \item Each edge of type $\{0,2\}$ is contained in $(q_v^2 + 1)$-chambers.
    \item Each edge of type $\{1,2\}$ is contained in $(q_v + 1)$-chambers.
\end{itemize}

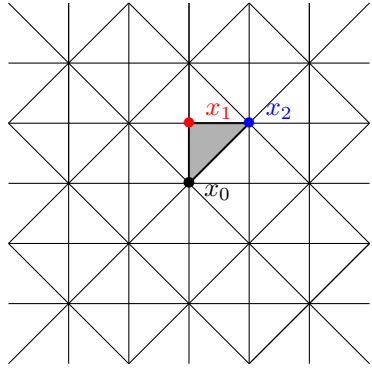
\begin{figure}[htbp]
				\begin{center}
					\tikzstyle{every node}=[scale=1]
					\begin{tikzpicture}[line width=0.4pt,scale=0.8][>=latex]
						\pgfmathsetmacro\ax{1}
						\pgfmathsetmacro\ay{1*sin(90)}

                        \foreach \k in {-2,-1,...,2}{
						\draw[
                            black,-] (-3*\ax,\k*\ay) -- (3*\ax,\k*\ay); 	}

                            \foreach \k in {-2,-1,...,2}{
						\draw[
                            black,-] (\k*\ax,-3*\ay) -- (\k*\ax,3*\ay); 	}

                            \foreach \k in {-2,0,2}{
						\draw[
                            black,-] (\k*\ax,0) -- (\k*\ax,3*\ay); 	}

                            \draw[
                            black,-] (-3*\ax,-3*\ay) -- (3*\ax,3*\ay); 	

                            \draw[
                            black,-] (-\ax,-3*\ay) -- (3*\ax,\ay); 
                            \draw[
                            black,-] (-3*\ax,-\ay) -- (\ax,3*\ay); 
                            \draw[
                            black,-] (-3*\ax,3*\ay) -- (3*\ax,-3*\ay); 

                            \draw[
                            black,-] (-3*\ax,\ay) -- (\ax,-3*\ay);

                            \draw[
                            black,-] (-\ax,3*\ay) -- (3*\ax,-\ay);

                            \draw[
                            black,-] (-3*\ax,-\ay) -- (-\ax,-3*\ay);

                            \draw[
                            black,-] (-3*\ax,\ay) -- (-\ax,3*\ay);

                            \draw[
                            black,-] (3*\ax,\ay) -- (\ax,3*\ay);
                        
                            \draw[
                            black,-] (3*\ax,-\ay) -- (\ax,-3*\ay);

                            \draw[
                            black,-] (3*\ax,-\ay) -- (\ax,-3*\ay);

                            \draw[
                            black,-,thick] (0,0) -- (0,\ay);

                            \draw[
                          black,-,thick] (0,0) -- (\ax,\ay);

                          \draw[
                          black,-,thick] (0,\ay) -- (\ax,\ay);

                          \draw[ fill=gray,opacity=0.6] (0,0) -- (\ax,\ay) -- (0,\ax) -- (0,0);
						
                        \node at (\ax,\ay) [blue] {\(\bullet\)};
						\node at (0,0) [thick] {\(\bullet\)};
                       \node at (0,0) [thick] {\(\bullet\)};
                       \node at (0,\ay) [red] {\(\bullet\)};
						\node at (1.5*\ax,1.2*\ay)[text=blue] {\(x_2\)};
						\node at (-18:0.5) {\(x_0\)};
                           \node at (0.5*\ax,1.2*\ay)[text=red] {\(x_1\)}
                        ;

					\end{tikzpicture}

				\end{center}
            \caption{Apartment and chamber of type ${}^2 \! A'_4$}\label{fig:inertapart}
			\end{figure}

\subsection{Algorithm}\label{ssec:algodetails}
In the notation of the previous sections, we are picking $G, K_0^\infty, \Lambda_{0,p}$ with $a := a_1$. The number $n$ will define $K^{\infty, p} = K_0^{\infty,p}(n)$ and therefore $\Lambda_p$. 

\subsubsection*{Input} Prime $p \neq 2,7$ and $n \in \Z_{>0}$ relatively prime to $2 \cdot 7 \cdot p$. 

\subsubsection*{Output} Ramanujan simplicial complex of the type described in
\[
\begin{cases}
\S\ref{sssec:splitcomplextype} & p \equiv 1,2,4 \pmod 7 \\
\S\ref{sssec:urcomplextype} & p \equiv 3,5,6 \pmod 7 
\end{cases}
\]
with
\[
n^{24}  \prod_{\substack{q | n, \\ q \equiv 1,2,4 \\ \bmod 7}}  \lf( \f1{q^{24}(q-1)} \prod_{i=0}^4 (q^5 - q^i) \ri) \times \prod_{\substack{q|n, \\ q \equiv 3,5,6 \\ \bmod 7}} \lf( \f1{q^{24}(q+1)} \prod_{i=0}^4 (q^5 + (-1)^i q^i) \ri)
\]
vertices if $p \equiv 1,2,4 \pmod 7$ and
\[
n^{25}  \prod_{\substack{q | n, \\ q \equiv 1,2,4 \\ \bmod 7}}  \lf( \f1{q^{25}} \prod_{i=0}^4 (q^5 - q^i) \ri) \times \prod_{\substack{q|n, \\ q \equiv 3,5,6 \\ \bmod 7}} \lf( \f1{q^{25}} \prod_{i=0}^4 (q^5 + (-1)^i q^i) \ri)
\]
type-$0$ vertices if $p \equiv 2,3,5 \pmod 7$.

\subsubsection*{Constants and Notation}\label{sssec algo}
We define constants:
\begin{itemize}
    \item $p_0=11$,    
    
    \item $\alpha_1$ is a root of the polynomial $x^5 - x^4 - 4x^3 + 3x^2 + 3x - 1$ and $\alpha_i := \alpha_1^i$, 
    \item $\rho_2 := (1 + \sqrt{-7})/2$, $\rho_{11}=2+\sqrt{-7}$,
    \item If $p$ is split, $p = \rho_p \bar \rho_p$ in $\Z[\rho_2]$, 
    \item $a = \rho_2\rho_{11}^3/\bar \rho_2\bar \rho_{11}^3$,
    \item $b = 10$,
    
    \item \[
    \m A = \begin{pmatrix}
    \alpha_1 & \alpha_1^{\sigma} & \alpha_1^{\sigma^2} & \alpha_1^{\sigma^3} & \alpha_1^{\sigma^4} \\
    \alpha_1^{\sigma} & \alpha_1^{\sigma^2} & \alpha_1^{\sigma^3} & \alpha_1^{\sigma^4} & \alpha_1 \\
    \alpha_1^{\sigma^2} & \alpha_1^{\sigma^3} & \alpha_1^{\sigma^4} & \alpha_1 & \alpha_1^{\sigma} \\
    \alpha_1^{\sigma^3} & \alpha_1^{\sigma^4} & \alpha_1 & \alpha_1^{\sigma} & \alpha_1^{\sigma^2} \\
    \alpha_1^{\sigma^4} & \alpha_1 & \alpha_1^{\sigma} & \alpha_1^{\sigma^2} & \alpha_1^{\sigma^3}
    \end{pmatrix}^{-1}.
    \]

\item Write $n=n_0 p_0^{n_{p_0}}$ with $\gcd(n_0,p_0)=1$ and $n_{p_0}\geq 0$. 
\end{itemize}
As notation
\begin{itemize}
\item $E$ is the number field $\Q[\rho_2]$,
\item $M$ is the totally real field $\Q[\alpha_1]$,
\item $L$ is the number field $\Q[\rho_2, \alpha_1]$ with ring of integers 
\[
\mc O_L = \Z\langle  \alpha_i,\ \rho_2\alpha_i : 0 \leq i \leq 4 \rangle,
\]
\item $\sigma:\alpha_1\mapsto\alpha_1^4+4\alpha_1^2-2,\ \rho_2\mapsto \rho_2$  and   $\ \bar {}:\alpha_1\mapsto\alpha_1,\ \rho_2\mapsto \bar \rho_2$ are automorphisms of $L$,
\item Given $\gamma \in L$,
\[
\m U_\gamma := \begin{pmatrix}
        \gamma & & & &\\
        & \gamma\gamma^\sigma & & &\\
        & & \gamma\gamma^\sigma\gamma^{\sigma^2} & & \\
        & & & \gamma\gamma^\sigma\gamma^{\sigma^2}\gamma^{\sigma^3} & \\
        & & & & \gamma\gamma^\sigma\gamma^{\sigma^2}\gamma^{\sigma^3}  \gamma^{\sigma^4}
        \end{pmatrix} .
\]
\end{itemize}
\subsubsection*{Algorithm: Precomputation}
We first do some precomputations for the prime $p$ determining the local structure of our expander complexes
\begin{enumerate}
    \item Find the gate set $S_{\bar \Lambda_{0,p}}$:
    \begin{enumerate}
        \item Following \S\S \ref{subsection despLambdaDmaxiota}, \ref{subsection QFCPMV} to find the transition matrices $A_{\max},B_{\max}$ and quadratic form $Q_{\max}$. One possible choice of $B_{\max}$ and $Q_{\max}$ is listed in \S \ref{subsection initialBmax}.

        \item Following the three-step strategy in \S \ref{subsection strategy} to solve the norm equation for $x\in\Lambda_D^{\max} \setminus \mc O_E$ such that
        \[\iota(x)x=p^e:=\begin{cases}p\quad &p\equiv 1,2,4\pmod{7},\\
        p^2&p\equiv 3,5,6\pmod{7}.\end{cases}\]
        \item Set $S_{\bar \Lambda_{0,p}}$ to be the set of $\psi x$ for these $x$ such that $\iota(x) x = p^e$ and where $e=1$ or $2$ and $\psi = \bar \rho_p^{-1}$ or $p^{-1}$ in the split and inert cases respectively.
        \end{enumerate}
    \item Compute the map $Q \iso S_{\bar \Lambda_{0,p}} : x \mapsto s_x$ needed for Construction \ref{cons:split}, \ref{cons:nonsplit}:
    \begin{enumerate}
        \item When $p$ or $n$ is divisible by $p_0$, use the expression of $\Lambda_{D}^{\max}$ and Chinese remainder theorem to calculate the matrix $\m V_{p_0}\in \GL_5(L)$ as in \S \ref{sssec:localtriofDmax}.
        
        \item If $p \equiv 1,2,4 \pmod 7$,
        \begin{enumerate}
            \item Find $\gamma_p \in \mc O_{L,(p)}$ such that $N_{L/E}(\gamma_p) \equiv a \pmod p$.
            \item Define the matrix $\m B_p = \m U_{\gamma_p} \m A^{-1}\m V_{p_0}^{-\log_{p_0}(\gcd(p,p_0))}$. 
            \item The vertices $\bar{\mc M} \subsetneq (\mc O_{E_{\rho_p}}/\rho_p)^5$ in $Q$ and simplices containing these vertices are determined by Construction \ref{cons:Qsplit}. 
            \item $s_{\bar{\mc M}}$ is the element of $S_{\bar \Lambda_{0,p}}$ such that
            \[
            (\m B_p^{-1} s_{\bar{\mc M}} \m B_p )(\mc O_{E_{\rho_p}}/\rho_p)^5 = \bar{\mc M}.
            \]
        \end{enumerate}
        \item If $p \equiv 3,5,6 \pmod 7$:
        \begin{enumerate}
            \item Find $\gamma_{p^2} \in \mc O_{L,(p)}$ such that $N_{L/E}(\gamma_{p^2}) \equiv a \pmod{p^2}$.
            \item Define the matrix $\m B_{p^2} = \m U_{\gamma_{p^2}} \m A^{-1}\m V_{p_0}^{-\log_{p_0}(\gcd(p,p_0))}$.
             \item The vertices $\bar{\mc M} \subsetneq (\mc O_{E_{\rho_p}}/p^2)^5$ in $Q$ and simplices containing these vertices are determined by Construction \ref{cons:Qinert} inputting $\m H = \m B_{p^2}^\dagger \m B_{p^2}$. 
            \item $s_{\bar{\mc M}}$ is the element of $S_{\bar \Lambda_{0,p}}$ such that
            \[
            (\m B_{p^2}^{-1} ps_{\bar{\mc M}} \m B_{p^2} )(\mc O_{E_{p}}/p^2)^5 = \bar{\mc M}.
            \]
            (if desired, it is possible to further conjugate by a $\m C_{p^2}$ as in Lemma \ref{lem:unitarymodn} so that the $\m H$ inputted to Construction \ref{cons:Qinert} is the antidiagonal Hermitian form $\m J$.) 
        \end{enumerate}
    \end{enumerate}

\end{enumerate}

\subsubsection*{Algorithm: Construction} Then we can construct our infinite family of complexes indexed by $n$:
\begin{enumerate}
    \item Reduce $S_{\bar \Lambda_{0,p}}$ mod $\bar \Lambda_p$:
    \begin{enumerate}
        \item Find $\gamma_n \in \mc O_{L,(n)}$ such that $N_{L/E}(\gamma_n) \equiv a \pmod n$ through Hensel's lemma and the Chinese remainder theorem.
        \item Define the matrix $\m B_n = \m U_{\gamma_n} \m A^{-1}\m V_{p_0}^{-\min(1,n_{p_0})}$.
        \item Find $\m C_n$ such that $\m C_n^\dagger \m B_n^\dagger \m B_n \m C_n \equiv \lb \cdot \m J \mod n$, where $\m J$ is the antidiagonal matrix with all ones and $\lambda \in (\Z/n)^\times$. 
        \item Define $U_5(\Z/n\Z)$ to be the subgroup $\GL_5(\Z[(1 + \sqrt{-7})/2]/n)$ preserving the Hermitian form $\m J$. Let 
        \[
        U^\sharp_5(\Z/n) := \begin{cases}
          U_5(\Z/n\Z)/U_1(\Z/n\Z)  & p \equiv 1,2,4 \pmod 7 \\
          U_5(\Z/n\Z) & p \equiv 3,5,6 \pmod 7
        \end{cases}
        \]
        \item $\m C_n^{-1} \m B_n^{-1} S_{\bar \Lambda_{0,p}} \m B_n \m C_n \pmod n$ gives a set of representatives of $U^\sharp_5(\Z/n\Z)$.  
    \end{enumerate}
    \item Construct $X_G(K(n)^{p, \infty})$ using the generators $S_{\bar \Lambda_{0,p}}$ of the group $\bar \Lambda_{0,p}/\bar \Lambda_p \cong U^\sharp_5(\Z/n\Z)$ and the assignment $Q \iso S_{\bar \Lambda_{0,p}} : x \mapsto s_x$ through:
    \[
    \begin{cases}
        \text{Construction } \ref{cons:split} & p \equiv 1,2,4 \pmod 7, \\
        \text{Construction } \ref{cons:nonsplit} & p \equiv 3,5,6 \pmod 7.
    \end{cases} 
    \]
\end{enumerate}

The most computationally intensive step is (1) in the precomputation, where the key is to solve Problem \ref{prob Qmaxboundcharpoly}. Luckily, the precomputation only needs to be done once for each infinite family of complexes---for the theoretical purpose of the algorithm being polynomial time in $n$ this is just an (extremely large) additive constant factor.

\newcommand{\etalchar}[1]{$^{#1}$}
\providecommand{\bysame}{\leavevmode\hbox to3em{\hrulefill}\thinspace}
\providecommand{\MR}{\relax\ifhmode\unskip\space\fi MR }
\providecommand{\MRhref}[2]{%
  \href{http://www.ams.org/mathscinet-getitem?mr=#1}{#2}
}
\providecommand{\href}[2]{#2}

\end{document}